\documentclass[reqno,oneside,12pt]{amsart}

\usepackage[normalem]{ulem}

\usepackage[color=yellow!60!red,textsize=tiny,textwidth=18mm]{todonotes} 
\setuptodonotes{inline}

\usepackage{amsmath,amsfonts,amssymb,amsthm}
\usepackage{MnSymbol}

\usepackage{tikz}
\usetikzlibrary{arrows.meta, positioning}
\usepackage{tikz-cd}
\tikzset{node distance=1.5cm, auto}
\usepackage{color}
\definecolor{darkgreen}{rgb}{0,0.45,0}
\usepackage[colorlinks,citecolor=darkgreen,linkcolor=darkgreen]{hyperref}
\usepackage[capitalize]{cleveref}
\usepackage{stmaryrd}
\usepackage{enumerate}
\usepackage[shortlabels]{enumitem}
\usepackage{mathtools}

\usepackage{graphicx}
\usepackage{enumerate}
\usepackage[all]{xy}
\usepackage{bbold}
\usepackage[latin9]{inputenc}
\usepackage{cancel}

\newcommand{\widedots}{\,\cdots\,} 

\def\ul{\underline}
\newcommand{\epspi}{\epsilon^{\pi}}
\newcommand{\eps}{\epsilon}
\newcommand{\um}{_A}
\newcommand{\one}{1_A}

\newcommand{\Hom}{{\sf Hom}}

\newcommand{\End}{{\sf End}}

\newcommand{\Mod}{{\sf Mod}}

\newcommand{\Mat}{{\sf Mat}}

\newcommand{\are}{\ar@{-}}

\newcommand{\co}{\ar@{--}}
\newcommand{\umgx}{1_{G_X}}

\newtheorem*{proposicao*}{Proposition}
\newtheorem*{teorema*}{Theorem}
\newtheorem*{corolario*}{Corollary}

\newtheorem{definicao}{Definition}[section]
\newtheorem{teorema}[definicao]{Theorem}
\newtheorem{exemplo}[definicao]{Example}

\newtheorem{lema}[definicao]{Lemma}
\newtheorem{obs}[definicao]{Remark}
\newtheorem{corolario}[definicao]{Corollary}
\newtheorem{proposicao}[definicao]{Proposition}

\newcommand{\K}{\mathbb{K}}

\newcommand{\Z}{\mathbb{Z}}

\newcommand{\larr}{\longrightarrow}
\newcommand{\lmap}{\longmapsto}
\newcommand{\x}{\otimes}

\newcommand{\Hpar}{H_{\textrm{par}}}

\newcommand{\bk}[1]{\left<#1\right>}

\newcommand{\set}[1]{\{#1\}}

\def\black{\color{black}}

\newcommand{\Px}{\mathcal{P}}

\def\eps{\epsilon}

\newtheorem{prop}{Proposition}[section]

\newtheorem{lemma}[prop]{Lemma}

\newtheorem{theorem}[prop]{Theorem}

\theoremstyle{definition}

\newtheorem{definition}[prop]{Definition}

\newtheorem{example}[prop]{Example}

\newcommand{\benu}{\begin{enumerate}}
\newcommand{\enu}{\end{enumerate}}

\newcommand{\beqna}{\begin{eqnarray}}
\newcommand{\eqna}{\end{eqnarray}}
\newcommand{\beqnast}{\begin{eqnarray*}}
\newcommand{\eqnast}{\end{eqnarray*}}
\newcommand{\beqn}{\begin{equation}}
\newcommand{\eqn}{\end{equation}}
\newcommand{\beqnst}{\begin{equation*}}
\newcommand{\eqnst}{\end{equation*}}

\usepackage{latexsym}

\usepackage{xspace}
\usepackage{amscd}
\usepackage{amssymb}
\usepackage{amsfonts}

\makeatother

\newcommand{\bema}{\left ( \begin{array}}
\newcommand{\ema}{\end{array} \right )}

\newcommand{\kpar}{\K_{par}G}
\newcommand{\apar}{A_{par}}

\newcommand{\calG}{\mathcal{G}}

\begin{document}

	\title[On partial representations of pointed Hopf algebras]{On partial representations of pointed Hopf algebras}

 \author[A. R.\ Alves Neto]{Arthur Rezende Alves Neto}
	\address{\textup{A. R. Alves Neto} Pesquisador de P{\'o}s-Doutorado, Bolsista FAPESC, Universidade Federal de Santa Catarina, Brazil}
    \email{arthurntcwb@yahoo.com.br}
	
	\author[M.\ M. Alves]{Marcelo Muniz Alves}
	\address{\textup{M.\ M. \ Alves} Departamento de Matem\'atica, Universidade Federal do Paran{\'a}, Brazil}
    \email{marcelomsa@ufpr.br}

    \author[E.\  Batista]{Eliezer Batista}
	\address{\textup{E.\ Batista} Departamento de Matem\'atica, Universidade Federal de Santa Catarina, Brazil}
    \email{ebatista@mtm.ufsc.br}

 \thanks{This work was partially financed by the Coordena\c{c}\~ao de Aperfei\c{c}oamento de Pessoal de N\'ivel Superior - Brasil (CAPES), Finance Code 001; by the National Council for Scientific and Technological Development of Brazil (CNPq), Research Productivity Fellowship (Bolsa de Produtividade em Pesquisa) grant  309469/2019-8 and  Universal Grant 405540/2023-0; and 
Funda\c{c}\~ao de Amparo \`a Pesquisa do Estado de Minas Gerais (FAPEMIG) under Universal Grant APQ-03491-25.
 } 
	
	\flushbottom
	
	\begin{abstract} 
        The theory of partial representations first appeared in the context of groups to deal with restrictions of representations to non-invariant subspaces. Partial representations of a group $G$ correspond to representations of a groupoid algebra constructed out of $G$. The notion of a partial representation can be extended to the context of Hopf algebras. For a Hopf algebra $H$, its partial representations correspond to representations of a Hopf algebroid $H_{par}$, also constructed out of data coming from $H$. In this work, we describe the structure of the algebra $H_{par}$ for $H$ being a pointed Hopf algebra with a finite group of grouplikes $G(H)$. This algebra can be explicitly written as a direct sum, reflecting the groupoid algebra associated to the group $G(H)$. Explicit examples of rank one Hopf algebras are done in the nilpotent and the non-nilpotent cases.
	\end{abstract}
	
	\maketitle
 
	\tableofcontents


	\section*{Introduction}
	\thispagestyle{empty}

The theory of partial representations of Hopf algebras traces back to the theory of partial representations of groups, which, in turn, was motivated by partial group actions.

Partial group actions can be defined on sets, topological spaces, semigroups, and algebras. The notion of a partial group action was introduced by Exel in the 1990s \cite{exel1994circle,E98}. Its motivating problem was the existence of $C^*$-algebras that carry an action of the group $\mathbb{S}^1$ by automorphisms but are not isomorphic to a crossed product of their fixed point subalgebras by the group of integers. This apparent anomaly led to the exploration of the action of the group of integers by partial automorphisms, that is, with each group element implementing an isomorphism between ideals within the algebra.

In the purely algebraic case \cite{DE05}, a partial action of a group $G$ on a $\K$-algebra $A$ is a datum assigning to each element $g \in G$ an isomorphism $\theta_g :A_{g^{-1}}\rightarrow A_g$ between two ideals of $A$ such that the equality $\theta_g \circ \theta_h = \theta_{gh}$ holds in the largest ideal where both sides make sense. 
\black 

These domain restrictions of the partial action give rise to an unexpected by-product: the notion of a \textit{partial representation} of a group. A partial representation of a group $G$ on a unital algebra $B$ is a map $\pi :G \rightarrow B$ such that $\pi (e_G )=1_B$ and, for every $g,h\in G$, $\pi(g)\pi(h)\pi(h^{-1})=\pi(gh) \pi(h^{-1})$ and $\pi(g^{-1})\pi(g)\pi(h)=\pi(g^{ -1})\pi(gh)$. Intuitively speaking, a partial representation of a group $G$ allows two elements to be associated only when there is a ``witness'', either on the left side or on the right side. A first example is given by a partial action of a group $G$ on an algebra $A$ in which every ideal $A_g$ is generated by a central idempotent element $1_g \in A$; these partial actions will henceforth be called unital partial actions. In this case, one can define a partial representation $\pi :G \rightarrow \text{End}_{\K}(A)$ by $\pi (g)(a)=\theta_g (a 1_{g^{-1}})$. 

Every partial action on an algebra gives rise to a $G$-graded algebra $A \ast_\alpha G \cong \bigoplus_{g\in G} A_g$, the partial skew group algebra, which extends the construction of the usual skew group algebra \cite{DE05}. For the case of a unital partial action of $G$ on an algebra, the partial skew group algebra $A \ast_\alpha G$ carries another paradigmatic instance of a partial representation of the group $G$, defined as $\pi (g)=1_g \delta_g$. In fact, as we shall see later, partial skew group algebras will play an important role in the whole theory of partial group representations. A partial representation of a group $G$ can be viewed as a true representation of a larger algebra known as the \textit{partial group algebra} $\K_{par} G $ \cite{piccione}. This algebra is a semigroup algebra defined by generators and relations and is also isomorphic to the algebra of a groupoid $\calG(G)$ associated to the group $G$. It is also important to mention that $\K_{par} G $ is isomorphic to a partial skew group algebra:  $\K_{par} G $ contains a commutative subalgebra $A_{par}(G)$ upon which $G$ acts partially, and  $\K_{par} G \simeq A_{par}(G) *_\alpha G$ \cite{DE05}. 

Moving from groups to Hopf algebras, partial actions and coactions of Hopf algebras on algebras were originally introduced by Caenepeel, de Groot and Janssen \cite{caenepeel2004galois, caenepeel2008partial} as a means of applying the theory of Galois extensions via corings to the theory of partial Galois extensions developed by Dokuchaev, Ferrero and Paques \cite{ferrero2007}. Unlike the case of groups, partial actions of Hopf algebras are not expressed via partially defined isomorphisms between ideals of an algebra. In the case of partial actions of Hopf algebras, we have to deal with the breakdown of associativity in the module structure. The role of the central idempotents $1_g \in A$ for each element $g\in G$ is played by central idempotents $h\cdot 1_A$ in the convolution algebra $\text{Hom}_\K (H,A)$. This understanding of the significant role of idempotents in the convolution algebra proved to be a veritable ``Rosetta Stone'' for the study of partial actions of Hopf algebras, allowing one to translate results from groups to Hopf algebras. We shall see later that the use of convolution idempotents to lift techniques from groups to the context of Hopf algebras constitutes the main driving force of the present work.

As in the case of groups, a partial action of a Hopf algebra $H$ on an algebra $A$ gives rise to an $H$-comodule algebra, the partial smash product $\underline{A \# H}$, which generalizes the usual smash product associated to a global action of $H$ on $A$, and also extends the partial skew group algebra construction, since a $\K G$-partial action on $A$ is the same as a unital partial $G$-action, and the algebras $\underline{A \# \K G}$ and $ \simeq A *_\alpha G$ are isomorphic. These partial smash products will also play a central role in the theory of partial representations.

Partial representations of Hopf algebras were first introduced by Alves, Batista and Vercruysse in \cite{alves2015partial} (see Definition \ref{definicao.representacao.parcial}). There, the authors showed that, as in the case of groups, there are two paradigmatic examples coming from partial actions. First, a partial $H$-action on an algebra $A$ leads to a partial representation of $H$ on the algebra $\text{End}_{\K}(A)$ of linear endomorphisms of $A$, given by $\pi (h)(a) =h\cdot a$. Also, the partial smash product $\underline{A \# H}$ carries a partial representation of $H$, namely, $\pi (h) =1_A \# h$. As in the case of groups, partial representations of Hopf algebras lack the possibility of associating elements, except in the presence of appropriate ``witnesses''. Then we need to define a larger algebra, so that partial representations of $H$ may be treated as ordinary representations. This new algebra is called the ``partial Hopf algebra'' $\Hpar$ and, by construction, it factorizes each partial representation of $H$ through an algebra map. In particular, when we consider partial $H$-modules, that is, pairs $(M,\pi)$ where $M$ is a vector space and $\pi : H \rightarrow \text{End}_{\K} (M)$ is a partial representation, we have that the category of partial $H$-modules is isomorphic to the category of (left) $H_{par}$-modules. When the Hopf algebra $H$ is the group algebra $\K G$, there is a canonical algebra isomorphism $\K_{par} G \simeq (\K G)_{par}$. 

The structure of $H_{par}$ comes with a subalgebra  $\apar$, the ``base algebra'',  which carries a canonical partial $H$-action. This algebra $A_{par}$ measures, in a certain sense, how much a Hopf algebra is able to admit partial representations. Extending what was already known for partial group representations, we have an isomorphism of algebras $\underline{\apar \# H} \simeq \Hpar$. A fundamental feature of this construction is that $\Hpar$ is a \textit{Hopf algebroid} over the base algebra $A_{par}$. Therefore, the category of partial $H$-modules turns out to be a closed monoidal category whose forgetful functor to the category of $A_{par}$ bimodules is strictly monoidal and preserves inner homs. This rich structure of $H_{par}$ nicely accommodates the previous theory of partial group representations since $\K_{par} G$, being a groupoid algebra, is a weak Hopf algebra, which in turn is a particular case of a Hopf algebroid. 

The association of each Hopf algebra $H$ to its ``partial Hopf algebra'' $\Hpar$ defines a functor from the category of Hopf $\K$-algebras to the category of Hopf algebroids. There are quite unexpected phenomena associated with the construction of $H_{par}$: on the one hand, it may happen that $\Hpar$ is a Hopf algebra isomorphic to $H$, meaning that there is no partiality at all. This is the case of enveloping algebras of Lie algebras and other connected Hopf algebras \cite{arthur2024partial}. On the other hand, there are finite-dimensional Hopf algebras $H$ whose $\Hpar$ is infinite dimensional. This is the case, for example, of the Sweedler algebra \cite{alves2015partial}. 

The aim of this article is to explore the inner structure of the ``partial Hopf algebra'' $H_{par}$ for the case of a pointed Hopf algebra $H$. For this purpose, we will use the structure of the partial group algebra of the group $G(H)$ of grouplike elements of $H$. More precisely, we decompose the base algebra $A_{par}$ into a direct sum using the projections associated to subsets $X\subseteq G$ containing the unit of $G$. This decomposition of the partial group algebra into direct summands was first highlighted in \cite{piccione}, where the authors used the isomorphism between $\K_{par} G$ and the groupoid algebra $\K \mathcal{G}(G)$. Using the coradical filtration of a pointed Hopf algebra $H$, we prove that the properties of the generators associated to the coradical $H_0 = G(H)$, in a certain way, control the structure of the algebra $H_{par}$ as a whole. This method depends on the lifting of identities among convolution idempotent maps, summarized in Theorem \ref{teorema.idempotent.convolution.coradical}. The decomposition of $A_{par}$ into a direct sum allows one to write the whole algebra $H_{par}$ as direct sum of matrix algebras. This is proven in Theorem \ref{teorema.Hpar.as.summation.of.MatnHparPGX.}, which constitutes the main result of this work.

This article is organized as follows. In Section \ref{se:prelim} we recall some mathematical preliminaries: partial representations and partial actions of Hopf algebras, partial smash products and the ``partial Hopf algebra'' $H_{par}$. We conclude the preliminaries bydiscussing the behavior of convolution idempotent maps in $\text{Hom}_{\K} (H,A)$, where $H$ is a Hopf algebra and $A$ is a $\K$-algebra. The most important result of this section, which shows the relation of convolution idempotent maps with the coradical filtration, is Theorem \ref{teorema.idempotent.convolution.coradical}. 

Section \ref{se:smash} is devoted to the study of partial representations of a pointed Hopf algebra. We outline the technique of decomposing partial $H$-modules as a direct sum whose summands are associated to appropriate subsets of the group $G=G(H)$ of grouplike elements of $H$. The  structure of a partial left $H$-module algebra is described in Theorem \ref{teorema.A.GammaA.direct.somma}. In Subsection \ref{se:Apar_Pointed} the structure of $H_{par}$, for $H$ being a pointed Hopf algebra, is analyzed using the same techniques. Basically, the decomposition of the algebra $\Hpar$ into blocks depends on the corresponding decomposition of the subalgebra $A_{par}$ in a direct sum of $\Hpar$-invariant ideals. One of the main technical problems is then to describe the partial $H$-action on those ideals, and this is done in Proposition \ref{proposicao.APGn.iso.AGammaX}. Finally, Theorem \ref{teorema.Hpar.as.summation.of.MatnHparPGX.}, which is the central result of this paper, summarizes the decomposition of $H_{par}$ as a direct product of matrix algebras indexed by the components of the groupoid $\calG(G)$. We also discuss how this decomposition allows one to identify the global core (the maximal global submodule of a partial $H$-module) and the global shadow (the maximal global quotient module of a partial $H$-module). In the case of pointed Hopf algebras with a finite group of grouplike elements, the global core and the global shadow are isomorphic $H$-modules (see Proposition \ref{globalcore=globalshadow})

In Section \ref{se:examples}, the previously obtained decomposition of $\Hpar$ is performed in concrete examples. We begin by considering $H = \K G$, the group algebra of a finite group $G$, and we recover the matrix decomposition obtained in \cite{piccione} for the ``partial group algebra'' $\K_{par}G$. After that, we present the structure of $H_{par}$ for two pointed Hopf algebras of rank one, as introduced in \cite{krop2006finite}, which share the same group of grouplike elements and the same dimension. The nilpotent case is done in Subsection \ref{section.nilpotent} and the 
non-nilpotent case in Subsection \ref{section.non-nilpotent}. 


    \section{Mathematical Preliminaries}
	\label{se:prelim}
	
	In this section we will recall the basic notions about partial representations of Hopf algebras and the coradical filtration.

	\subsection{Partial representations of Hopf algebras}
	
	\begin{definition}[\cite{alves2015partial}]\label{definicao.representacao.parcial}
		A partial representation of \(H\) on an algebra \(B\) is a linear map \(\pi : H \to B\) such that for all \(h, k \in H\)
		\begin{enumerate}[(PR1)]
			\item \(\pi(1_H) = 1_B\); \label{PR1}
			\item \(\pi(h) \pi(k_{(1)}) \pi(S(k_{(2)})) = \pi(hk_{(1)}) \pi(S(k_{(2)}))\); \label{PR2}
			\item \(\pi(h_{(1)}) \pi(S(h_{(2)})) \pi(k) = \pi(h_{(1)}) \pi(S(h_{(2)})k)\). \label{PR3}
		\end{enumerate}
		By \cite[Lemma 2.11]{alves2021partial},  axioms \ref{PR2} and \ref{PR3} can be replaced with
		\begin{enumerate}[(PR1)]
			\setcounter{enumi}{3}
			\item \(\pi(h) \pi(S(k_{(1)})) \pi(k_{(2)}) = \pi(hS(k_{(1)})) \pi(k_{(2)})\); \label{PR4}
			\item \(\pi(S(h_{(1)})) \pi(h_{(2)}) \pi(k) = \pi(S(h_{(1)}))\pi(h_{(2)}k)\). \label{PR5}
		\end{enumerate}
	\end{definition}
    \begin{definition}
    A \emph{left partial \(H\)-module} is a vector space \(M\) endowed with a partial representation \( \pi: H \to \End_k (M) \).     
    \end{definition}
    
	From the axioms it follows immediately that an algebra morphism \(\pi: H \to B\) is a partial representation. We will often refer to those as \textit{global representations} of \(H\) on \(B\). 
	
	It turns out that the partial representations of \(H\) factor uniquely through the ``partial Hopf algebra'' \(H_{par}\). This algebra is constructed as the quotient of the tensor algebra \(T(H)\) by the relations
	\begin{gather}
		1_H = 1_{T(H)}; \label{eq:Hparrelation1} \\
		h \otimes k_{(1)} \otimes S(k_{(2)}) = hk_{(1)} \otimes S(k_{(2)}); \label{eq:Hparrelation2}\\
		h_{(1)} \otimes S(h_{(2)}) \otimes k = h_{(1)} \otimes S(h_{(2)})k; \label{eq:Hparrelation3}
	\end{gather}
	for all \(h, k \in H\). The class of \(h \in H\) in \(H_{par}\) is denoted as \([h]\).
	It is easy to see that the linear map 
	\begin{equation}
		\label{eq:bracket}
		[-] : H \to H_{par} : h \mapsto [h]
	\end{equation}
	is a partial representation. In the other direction there is an algebra map
	\begin{equation}
		\label{eq:removing_brackets}
		H_{par} \to H : [h_{(1)}]  \widedots  [h_{(n)}] \mapsto h_{(1)}  \widedots  h_{(n)}.
	\end{equation}
	
	The algebra \(H_{par}\) satisfies the following universal property.
	\begin{theorem}[{\cite[Theorem 4.2]{alves2015partial}}]
		\label{th:universal_property}
		For every partial representation \(\pi : H \to B\) there is a unique algebra morphism \(\hat{\pi} : H_{par} \to B\) such that \(\pi = \hat{\pi} \circ [-]\). Conversely, given an algebra morphism \(\varphi : H_{par} \to B,\) the map \(\varphi \circ [-] : H \to B\) is a partial representation.
	\end{theorem}
	This tells us that there is a bijective correspondence between partial representations of \(H\) and representations of \(H_{par}\), and that the category of left partial \(H\)-modules is isomorphic to the category of left \(H_{par}\)-modules.

An essential tool in understanding the algebra structure of \( H_{par}\) is its description as a \textit{partial smash product}, a construction introduced in \cite{caenepeel2008partial} associated to left partial actions. 

\begin{definition}[\cite{alves2015partial}]\label{definicao.acao.parcial} A \textit{left partial action} of a Hopf algebra $H$ on an algebra $A$ is a linear map \(H \otimes A \to A\), given by $h \otimes a \mapsto h \cdot a$, such that
\begin{enumerate}[(\textrm{PA}1)]
    \item \label{PA1} $1_H\cdot a = a$,
    \item \label{PA2} $h\cdot (ab) = (h_{(1)} \cdot a)(h_{(2)} \cdot b)$,
    \item \label{PA3} $h\cdot (k\cdot a) = (h_{(1)} \cdot 1_A)(h_{(2)} k\cdot a)$,
\end{enumerate}
for all $a, b\in A$ and $h,k \in H$. The algebra $A$ is called a \textit{partial left $H$-module algebra}. A left partial action is \textit{symmetric} if in addition, it satisfies
\begin{enumerate}[(\textrm{PA}1)]
			\setcounter{enumi}{3}
    \item \label{PA4}$h\cdot (k\cdot a) = (h_{(1)} k\cdot a)(h_{(2)} \cdot 1_A)$, 
\end{enumerate}
for all $h,k \in H$ and $a\in A$.
In this paper \textit{every partial action will be symmetric}.

     Let $A$ be a partial $H$-module algebra. The partial action endows \(A \otimes H\) with the associative product 
     $$ (a\x h)(b\x k) = a(h_{(1)}\cdot b)\x h_{(2)} k $$
     and the left ideal  $\underline{A \# H} = (A\x H)(1_A\x 1_H)$ generated by the idempotent $1_A \otimes 1_H$ is a unital algebra. As a vector space, $\underline{A \# H}$ is generated by elements of the form
$$ a\# h = a(h_{(1)} \cdot 1_A)\x h_{(2)} $$
for \(a \in A\) and \(h \in H\).
This algebra is called the \textit{partial smash product} of $A$ by $H$ \cite{caenepeel2008partial}.
\end{definition}

\begin{exemplo}\label{exemplo.A.symmetric.EndA.partial}\cite{alves2015partial}
Let $A$ be a symmetric partial left $H$-module algebra, and let $B = \textrm{End}(A)$. Then the map $\pi:H \to B$, given by $\pi(h)(a)=h\cdot a$ is a partial representation; besides \ref{PA1}-\ref{PA4}, we also have the equalities
\begin{eqnarray}
h \cdot ( k_{(1)} \cdot (S(k_{(2)}) \cdot a)) & = &   hk_{(1)} \cdot (S(k_{(2)}) \cdot a), \label{eq.ex.1.1-PR2}\\ 
h_{(1)} \cdot (S(h_{(2)}) \cdot ( k \cdot a)) & = &  h_{(1)} \cdot (S(h_{(2)}) k \cdot a)
\label{eq.ex.1.1-PR3}
\end{eqnarray}
and
\begin{eqnarray}
h \cdot ( S(k_{(1))} \cdot (k_{(2)}) \cdot a)) & = &    hS(k_{(1)}) \cdot (k_{(2)} \cdot a), 
\label{eq.ex.1.1-PR4}
\\ S(h_{(1)}) \cdot (h_{(2)} \cdot ( k \cdot a)) & = &  S(h_{(1)}) \cdot (h_{(2)} k \cdot a)
\label{eq.ex.1.1-PR5}
\end{eqnarray}
for every $h,k \in H$ and $a \in A$. In particular, if $g \in H$ is a grouplike,
\begin{equation*} \label{eq:g.ginverso.g}
    g \cdot a  = g \cdot (g^{-1} \cdot (g \cdot a))
\end{equation*}
for every $a$ in $A$.

The partial smash product also carries a partial representation associated to the $H$-partial action on $A$, the linear map  $\eta:H \to \underline{A \# H}$ given by $\eta(h) = 1_A\# h$.
\end{exemplo}

	 \(H_{par}\) itself can be described as a partial smash product. 
\begin{definicao}
     Given a Hopf algebra $H$,  \(\apar\) is the subalgebra of \( H_{par}\) generated by the elements of the form
	\begin{equation}
		\label{eq:epsilonh}
		\eps_h = [h_{(1)}] [S(h_{(2)})] \qquad \text{for } h \in H.
	\end{equation}
\end{definicao}     
    
There is a canonical partial action of $H$ on the algebra \(\apar\), and in fact the algebra \(H_{par}\) is isomorphic to the corresponding partial smash product. 
	
	\begin{theorem}[{\cite[Theorems 4.8 and 4.10]{alves2015partial}}] \label{teorema.Hpar.igual.AH} Let $H$ be a Hopf algebra. 
		\begin{enumerate}[(i)] 
			 \item There is a partial action of \(H\) on \(\apar\) defined by
			\begin{equation*}
				h \cdot a := [h_{(1)}] a [S(h_{(2)})]
			\end{equation*}
			for \(h \in H, a \in \apar\).
			\item The ``partial Hopf algebra'' \(H_{par}\) is isomorphic to the partial smash product algebra \(\underline{\apar \# H}\). The partial representation \eqref{eq:bracket} becomes
			\begin{equation}
				\label{eq:bracket_smash}
				H \to \underline{\apar \# H} : h \mapsto 1_{\apar} \# h.
			\end{equation}
        \item If \(H\) has invertible antipode, then $H_{par}$ has the structure of a Hopf algebroid over the subalgebra $\apar$.
		\end{enumerate}
	\end{theorem}

The subalgebra $A_{par}$ can be defined by generators and relations, and therefore it can be computed on its own. This result provides an approach to studying $\Hpar$ via the isomorphism $\Hpar \simeq \underline{ \apar \# H}$.  

Given any left partial $H$-action on an algebra $B$, let $e_B: H \to B$ be the map defined as $e_B (h) = h \cdot 1_B$. 
It can be shown that   
	
	\begin{eqnarray}
    e (1_H) & = & 1_B; \label{eq:Arelation11} \\
		e_B(h) & = &  e_B(h_{(1)}) e_B(h_{(2)}); \label{eq:Arelation22} \\
		e_B(h_{(1)}) e_B(h_{(2)}k) & = &  e_B(h_{(1)}k) e_B(h_{(2)}), 		 \label{eq:Arelation33}
    \end{eqnarray}
	for all \(h, k \in H\).
In particular, for the partial action of $H$ on $\apar$ we have that $e_{\apar} (h) = \eps_h$, and hence 
		\begin{eqnarray}
		1_{\apar} & = & \eps_{1_H}; \label{eq:Arelation1} \\
		\eps_h & = & \eps_{h_{(1)}}  \eps_{h_{(2)}}; \label{eq:Arelation2} \\
		\eps_{h_{(1)}} \eps_{h_{(2)}k} & = &\eps_{h_{(1)}k} \eps_{h_{(2)}}, \label{eq:Arelation3}
        \end{eqnarray}
	for all \(h, k \in H\).
 
The next result shows that those are in fact defining relations for $\apar$. 
 \begin{teorema} [{\cite[Theorem 4.12]{alves2015partial}}]
 \label{teorema.A.propriedade.universal}
     Let $H$ be a Hopf algebra with invertible antipode and let $\apar$ be the base algebra of the Hopf algebroid $H_{par}$. If $(B, e_B)$ is any pair consisting of an algebra $B$ and a linear map $e_B:H \to B$ that satisfies Equations \eqref{eq:Arelation11}, \eqref{eq:Arelation22} and \eqref{eq:Arelation33} then there exists a unique morphism of algebras $u :\apar \to B$ such that $e_B (h) =u (\eps_h) $ for every $h \in H$.
 \end{teorema}   

	Let us look at some examples to illustrate how partial representations of different Hopf algebras can behave very differently. 

 \begin{example}
     [Partial representations of finite groups] The structure of the algebra $(\K G)_{par}$ of a finite group $G$ was completely described in \cite{piccione}. Let $\Px_1(G)=\{X \subseteq G \mid e \in X\}$. 
     Let $\calG(G)$ be the groupoid whose arrows are the pairs $(X,g) \in \Px_1(G) \times G$ where $g^{-1} \in X$, and whose composition is defined as follows: the pair $((X,g), (Y,h))$ is composable only when 
     $X = hY$, and in this case $(hY, g) \circ  (Y,h) = (Y,gh)$. The identities are the arrows $(X,e)$ and we identify the set of objects of $\calG(G)$ with the set $\Px_1(G)$. The point is that the algebras $(\K G)_{par}$  
     and $\K \calG(G)$ are isomorphic. An isomorphism can be constructed as follows: to each $X \in \Px_1(G)$ one associates the idempotent 
$$ P_X = \prod_{x\in X} \epsilon_x\prod_{y\in G\smallsetminus X} (1_{\K_{par}G} - \epsilon_y) $$
of $\K_{par} G$. This is a complete set of central orthogonal idempotents of $\apar \subset \K_{par} G$. It follows from the proof of \cite[Thm. 2.6]{piccione} that the linear map \[\psi: \K \calG(G) \to \K_{par} G , \ \ \ (X,g) \mapsto [g]P_X\] is an isomorphism of algebras. 
     
 Now if $\calG(G) = \bigsqcup_{i=1}^n \calG(G)_i$ is the decomposition as a union of connected groupoids then $\K \calG(G) \simeq \bigoplus_{i=1}^n\K \calG(G)_i$. The components can be described: two objects $X, Y \in \Px_1(G)$ lie in the same component if and only if there exists $g\in G$, such that $g^{-1}\in X$ and $gX = Y$. Let $\mathcal{P}_1(G) = \bigcup_{i=1}^n \mathcal{O}_i$ be the partition of $\mathcal{P}_1(G)$ according to the components of $\calG(C)$ and let $X_1, X_2,  \ldots , X_n$ be a complete set of representatives in $\mathcal{P}_1(G)$ be the sets of objects of the components. The isotropy subgroup of $X_i$ is isomorphic to the group $ G_i = \{g \in G \mid g\cdot X_i = X_i\}$. Letting $m_i = \mid\mathcal{O}_i\mid$, we have the algebra isomorphisms
\[
\K_{par}G 
\simeq \K \calG(G) 
\simeq \bigoplus_{i=1}^n\K \calG(G)_i \simeq \bigoplus_{i=1}^n Mat_{m_i} (\K G_i).
\]
This decomposition can be refined: 
\[
\K_{par}G = \bigoplus_{\substack{L \leq G \\ 1 m \leq [G:L]}} c_m (L) Mat_m (\K L)
\]
where the multiplicity factor $c_m(L)$ can be recursively computed (a slight mistake in this formula was later corrected in  \cite{dokuchaev2004isomorphisms}).
     \end{example}
	
	\begin{example}
		\label{ex:lie}
		Let \(\mathfrak{g}\) be a Lie algebra and \(U(\mathfrak{g})\) its universal enveloping algebra. Every partial representation of \(U(\mathfrak{g})\) is global, as is shown in \cite[Example 4.4]{alves2015partial}.
	\end{example}

	\begin{example}
		\label{ex:sweedler}
		Let \(H\) be Sweedler's 4-dimensional Hopf algebra over a field of characteristic different of 2 (see for instance \cite[Example 1.5.6]{montgomery1993hopf}). 
		In \cite[Example 4.13]{alves2015partial}, a description of \(H_{par}\) was given, and in particular it was shown that the base algebra \(\apar\) of the Hopf algebroid \(H_{par}\) is infinite-dimensional and isomorphic to \(k[x, y]/(2x^2 - x, 2xz - z)\). Moreover, this shows that \(H_{par}\) does not have the structure of a weak Hopf algebra, since those have a separable Frobenius (in particular finite-dimensional) base algebra.
	\end{example}

\subsection{Convolution idempotent maps}

Let $A$ be a partial left $H$-module algebra. The map 
$$ e_A: H \to A, \ \ e_A(h) = h\cdot 1_A, $$ satisfies the fundamental equalities \eqref{eq:Arelation11}, \eqref{eq:Arelation22}, \eqref{eq:Arelation33}; the equality \eqref{eq:Arelation22} says that $e_A$ is an idempotent in the convolution algebra 
$\Hom_{\K}(H,A)$.

\begin{definicao}\label{definicao.idempotente.convolution.f.f*f}
Let $C$ be a coalgebra and let $A$ be an algebra. A map $f:C \to A$ is a \textbf{convolution idempotent} map if it is an idempotent of  the $\K$-algebra $\Hom_{\K}(C,A)$, that is,  
$$ f(c) = f(c_{(1)})f(c_{(2)}), \ \ \forall c\in C. $$
\end{definicao}

If $A$ is a partial left $H$-module algebra then the map $e_A : H \to A$, $e_A (h) = h \cdot 1_A$ is a convolution idempotent map. 
A natural extension of this example is the following: let $E$ be an idempotent element of $A$. By axiom (PA$2$) of Definition \ref{definicao.acao.parcial}, we have that
$$ h\cdot E = h\cdot(E^2) = (h_{(1)}\cdot E)(h_{(2)}\cdot E), \ \forall h \in H, $$
therefore the linear map $\alpha_E:H \to A$ defined by $\alpha_E(h) = h\cdot E$ is also a convolution idempotent map.

Our next objective is to show that if two convolution idempotent maps in $\Hom_{\K}(C,A)$ coincide on a particular subcoalgebra of $C$, its \textit{coradical}, then these maps are the same. In order to do so, it will be important to point out a fundamental property of the kernel of a convolution idempotent map.

Recall that if $V$ and $W$ are subspaces of the coalgebra $C$, their \textbf{wedge product} is the subspace 
\begin{equation} \label{wedge}
V \wedge W = \Delta^{-1}(V\x C + C\x W).
\end{equation}
Let  $f:C \to A$ be a convolution idempotent map and let $V$, $W$ be subspaces of $\ker f \subseteq C$. If $$c \in V \wedge W$$ then there exist $v_i \in V$, $w_j \in W$, and $x_i, y_j \in C$ ($i=1, \ldots ,n$; $j=1, \ldots ,k$) such that
$$ \Delta(c) = \sum_{i=1}^nv_i\x x_i + \sum_{j=1}^ky_j\x w_j. $$ 
Since $f$ is a convolution idempotent we have that
$$ f(c) = f(c_{(1)})f(c_{(2)}) = \sum_{i}\underbrace{f(v_i)}_{=0}f(x_i) +
\sum_{j}f(y_j)\underbrace{f(w_j)}_{=0} = 0. $$

We have just proved the following fact.

\begin{lema}\label{lema.convolution.wedge.product}

Let $C$ be a coalgebra, $A$ be an algebra, and let $f:C \to A$ be a convolution idempotent map. If $V$ and $W$ are subspaces of $\ker f \subseteq C$ then 
$$ f(c) = 0, \ \ \forall c \in V\wedge W.$$
\end{lema}

\begin{lema}\label{lema.convolution.f-g.idempotent}
 
 Let $C$ be a coalgebra, $A$ an algebra, and consider $f,g:C \to A$ two convolution idempotent maps. If for any $c \in C$, 
 $$ f(c_{(1)})g(c_{(2)}) = g(c_{(1)})f(c_{(2)}) = g(c), $$ 
 then $f-g$ is a convolution idempotent map as well. 
 
\end{lema}

In fact, these equalities mean that $f*g = g*f = g$ in $\Hom_\K (C,A)$ and 
this result is valid for any two idempotents $f,g$ in a ring such that $fg=gf=g$.

\begin{obs}
    Let $H$ be a Hopf algebra and $A$ be an algebra. We already know that if $A$ is a partial left $H$ module algebra, that is, if there exists a left partial action $\cdot:H\x A \to A$, then for each idempotent element $E$ of $A$, the map
    $$ \begin{array}{cccc}
        \alpha_E: & H & \larr & A, \\
           & h & \lmap & h\cdot E,
    \end{array} $$
    is a convolution idempotent. Notice that the unit $1_A$ of $A$ is an idempotent. Furthermore 
\[
    \alpha_E(h)  =  h\cdot E   =  h\cdot(E1_A)  = (h_{(1)}\cdot E)(h_{(2)}\cdot 1_A)  =  \alpha_E(h_{(1)})\alpha_{1_A}(h_{(2)}), 
\]
and we also have $\alpha_E(h) = \alpha_{1_A}(h_{(1)})\alpha_E(h_{(2)})$ by a similar computation. Hence, conclude that $(\alpha_E - \alpha_{1_A})$ is a convolution idempotent map for any idempotent element $E$ of $A$.
\end{obs}

As we already mentioned, we are now going to prove that if two convolution idempotent maps in $\Hom_{\K}(C,A)$ coincide on the \textit{coradical} of $C$, then they coincide on the whole coalgebra $C$.

\begin{definicao}\label{definicao.coradical.coalgebra}

    Let $C$ be a coalgebra. The \textbf{coradical} of $C$ is the sum of all simple subcoalgebras of $C$, denoted by $C_0$.
    
\end{definicao}

The wedge product of two subcoalgebras is also a subcoalgebra, and this allows the construction of the \textbf{coradical filtration} of a coalgebra $C$ as follows: starting at the coradical $C_0$, consider the following family of subcoalgebras of $C$, defined recursively by 
\begin{equation} \label{equation.coradical.filtration}
C_1 := C_0\wedge C_0, \ \ C_n:= C_{n-1}\wedge C_0, \ \textrm{for all } n\geq 1, 
\end{equation}
where, as defined in equation \eqref{wedge}, $V\wedge W =\Delta^{-1}(V\x C + C\x W)$ for any subspaces $V, W$ of $C$. Then we have
\begin{equation}\label{equation.coradical.filtration.encaixe}
    C_0 \subseteq C_1 \subseteq  \ldots  \subseteq \bigcup_{n\geq 0} C_n = C. 
\end{equation}
A proof of the inclusions (\ref{equation.coradical.filtration.encaixe}) can be found in \cite[Prop.4.1.5]{DavidERadford}.

{ The following result was stated in \cite{arthur2024partial} in a more specific setting, namely in the case where the convolution idempotents arise from a partial representation of a Hopf algebra.}

\begin{teorema}\label{teorema.idempotent.convolution.coradical}
    Let $C$ be a coalgebra and $A$ an algebra, and $f,g:C \to A$ convolution idempotent maps. Then the following are equivalent:
    \begin{itemize}
        \item[i)] 
        $f = g$.

        \item[ii)] $\left\{\begin{array}{rcl}
            f(c_{(1)})g(c_{(2)}) & = & g(c_{(1)})f(c_{(2)}) \ = \ g(c), \textrm{ for all } c\in C; \\
            f(x) & = & g(x), \textrm{ for all } x \in C_0,\textrm{ where } C_0 \textrm{ is the coradical of } C.
        \end{array}\right.$ 
    \end{itemize}
\end{teorema}

\begin{proof}
    The fact that a) implies b) is obvious. Conversely, assume that the conditions in b) hold, and notice that since 
    $$ f(c_{(1)})g(c_{(2)}) = g(c_{(1)})f(c_{(2)}) = g(c), $$
    by Lemma \ref{lema.convolution.f-g.idempotent} we have that $(f-g):C \to A$ is also a convolution idempotent map. By assumption, we conclude that $(f-g)(x) = 0$ for all $x \in C_0$. Since $C_n := C_{n-1}\wedge C_0$, for all $n \geq 1$, it follows from Lemma \ref{lema.convolution.wedge.product} that $(f-g)(x) = 0$ for all $x \in C_n$ with $n\geq 0$. Finally, since $\bigcup_{n\geq 0} C_n = C$ we conclude that  $(f-g)(c) = 0$ for all $c \in C$.
\end{proof}

\begin{obs}\label{obs.idempotent.convolution.sizes}
    Note that in the previous theorem, it seems that the maps $f$ and $g$ play the same role, but the identity
    $$ f(c_{(1)})g(c_{(2)}) = g(c_{(1)})f(c_{(2)}) = g(c), \ \textrm{ for all } c\in C $$
    actually distinguishes them. 
 That identity can be rewritten as $ f*g=g*f=g $ which, in terms of ideals of $\Hom_\K (C,A)$, tell us that the left (resp. right) ideal generated by $f$ contains the left (resp. right) ideal generated by $g$.   
\end{obs}
	
	\section{Partial representations of pointed Hopf algebras}
	\label{se:smash}

A  \textbf{pointed} Hopf algebra $H$ is a Hopf algebra whose simple subcoalgebras are all one-dimensional, that is, whose coradical $H_0$ coincides with the Hopf subalgebra $\K G(H)$ generated by its group $G(H)$ of grouplike elements. 

A relevant feature of pointed Hopf algebras for the study of partial representations is the following result. 

\begin{proposicao}
\cite[Cor. 7.6.4]{DavidERadford}
If $H$ is a pointed Hopf algebra then its antipode $S$ is invertible. 
\end{proposicao}
Hence $\Hpar$ is a Hopf algebroid when $H$ is a pointed Hopf algebra (Thm \ref{teorema.Hpar.igual.AH}).

Before actually dealing with partial representations of  pointed Hopf algebras, we shall explore the relation between convolution idempotent maps and grouplike elements.

 \subsection{Grouplikes and Idempotents}
    \label{se:grouplikes_idempotents}
\begin{proposicao}\label{proposicao.epsg.idempotent.central}
Let $\pi:H \to B$ be a partial representation of a Hopf algebra $H$ on an algebra $B$. Consider the group $G:=G(H)$ of grouplike elements of $H$, and consider the map $\epsilon^\pi:H \to B$ defined by $\epsilon^\pi(h) = \pi(h_{(1)})\pi(S(h_{(2)}))$ for all $h \in H$. Then 
$$ \epsilon^{\pi}(g)^2 = \epsilon^{\pi}(g) \ \ \textrm{and} \ \ \epsilon^\pi(h)\epsilon^\pi(g) = \epsilon^\pi(g)\epsilon^\pi(h), $$ for all $g \in G$ and $h \in H$.
\end{proposicao}

\begin{proof}
    The restriction of $\pi$ to $G$ yields a partial representation of the group $G$ on $B$, and then those equalities hold by \cite{piccione}. 
\end{proof}

Now let $\pi:H \to B$ be a partial representation and consider the subalgebra $A_\pi(B)$ of $B$ generated by all elements of the form $\epspi(h) =  \pi(h_{(1)})\pi(S(h_{(2)}))$, for any $h \in H$. By the universal property of $\apar$ (Theorem \ref{teorema.A.propriedade.universal}), we can define the surjective algebra map $\overline{\epspi}: \apar \to A_\pi(B)$, given by $\overline{\epspi}(\eps_h) = \epspi(h)$. This subalgebra $A_{\pi}(B)$ has a canonical left partial $H$-action defined by 
$$ h\cdot a = \pi(h_{(1)})a\pi(S(h_{(2)})), \ \forall h \in H, a \in A_{\pi}(B). $$
More precisely, if $b = \epspi(k)$ for some $k \in H$, then
$$\begin{array}{ccl}
    h\cdot \epspi(k) & = & \pi(h_{(1)})\pi(k_{(1)})\pi(S(k_{(2)}))\pi(S(h_{(2)}))  \\
     & = & \pi(h_{(1)}k_{(1)})\pi(S(k_{(2)}))\pi(S(h_{(2)}))  \\
     & = & \pi(h_{(1)}k_{(1)})\pi(S(h_{(2)}k_{(2)})h_{(3)}k_{(3)})\pi(S(k_{(4)}))\pi(S(h_{(4)}))  \\
     & = & \pi(h_{(1)}k_{(1)})\pi(S(h_{(2)}k_{(2)}))\pi(h_{(3)}k_{(3)})\pi(S(k_{(4)}))\pi(S(h_{(4)}))  \\
     & = & \pi(h_{(1)}k_{(1)})\pi(S(h_{(2)}k_{(2)}))\pi(h_{(3)}k_{(3)}S(k_{(4)}))\pi(S(h_{(4)}))  \\
     & = & \pi(h_{(1)}k_{(1)})\pi(S(h_{(2)}k_{(2)}))\pi(h_{(3)})\pi(S(h_{(4)}))  \\
     & = & \epspi(h_{(1)}k)\epspi(h_{(2)}).
\end{array}$$
For a generic element $b = \epspi(k^1)\epspi(k^2)  \widedots  \epspi(k^n) \in A_{\pi}(B)$ where $k^1, k^2,  \ldots , k^n \in H$, we have
\begin{equation} \label{formula.H.action.on.Apar}
    h\cdot b = \epspi(h_{(1)}k^1)\epspi(h_{(2)}k^2)   \widedots  \epspi(h_{(n)}k^n)\epspi(h_{(n+1)}). 
\end{equation}
In other words, equation (\ref{formula.H.action.on.Apar}) tells us that 
$$ h\cdot b = \overline{\epspi}(h\cdot a), $$
where $a \in \apar$ and $\overline{\epspi}(a) = b$.
Therefore $A_\pi(B)$ is a partial left $H$-module algebra, and $\epspi(g) \in A_\pi(B)$ is a central idempotent for every grouplike element $g$ of $H$.

The above argument suggests using idempotent elements in order to prove the following key lemma. We already know that if 
$$ \begin{array}{cccc}
    \cdot : & H\x A & \larr & A, \\
     & h\x a & \lmap & h\cdot a
\end{array} $$
is a left partial action and $E$ an idempotent element of $A$, the map $\alpha_E:H \to A$, given by $\alpha_E(h) = h\cdot E$ is a convolution idempotent map.

 \begin{lema}\label{lema.partialaction.idempotent.fE}
    Let $\cdot: H\x A \to A$ be a left partial action of a Hopf algebra $H$ on an algebra $A$. If $E$ is a central idempotent element of $A$, and $f:H \to A$ a convolution idempotent map such that
    $$\left\{\begin{array}{rcl}
        f(h_{(1)})(h_{(2)}\cdot E) & = & (h_{(1)}\cdot E)f(h_{(2)}) \ \ = \ \ h\cdot E, \textrm{ for all } h \in H, \\
        x\cdot E & = & f(x)E, \textrm{ for all } x\in H_0, \textrm{ where } H_0 \textrm{ is the coradical of } H;
    \end{array}\right.$$
    then $h\cdot E = f(h)E$, for all $h \in H$.
\end{lema}

\begin{proof}
    The main idea is to construct convolution idempotent maps, and then show that they coincide on the coradical $H_0$. Consider the following maps:
    $$ \begin{array}{ccccl}
        \alpha_E: & H & \larr & A, & \textrm{defined by } \alpha_E(h)=h\cdot E,  \\
        \alpha'_E: & H & \larr & A, & \textrm{defined by } \alpha'_E(h)=\bigl(h\cdot E\bigr)E, \\
        f': & H & \larr & A, & \textrm{defined by } f'(h)=f(h)E,  
    \end{array}$$
    for all $h \in H$. It is readily seen that these three maps are convolution idempotent. Now consider $x \in H_0$. Since $x\cdot E = f(x)E$ by assumption, we have:
    $$ \alpha_E(x) = x\cdot E = f(x)E = f'(x), \ \ \textrm{and} \ \ \alpha'_E(x) = \bigl(x\cdot E\bigr)E = f(x)E^2=f(x)E = f'(x). $$
    Therefore all three maps $\alpha_E, \alpha'_E, f':H \to A$ coincide on the coradical $H_0$. By assumption, we know that
    \[
        \alpha'_E(h_{(1)})f'(h_{(2)})  =  \Bigl(\bigl(h_{(1)}\cdot E\bigr)E\Bigr)f(h_{(2)})E 
         =  \Bigl(\bigl(h_{(1)}\cdot E\bigr)f(h_{(2)})\Bigr)E^2  
          =  (h\cdot E)E  
         =  \alpha'_E(h), 
    \] 
    where the second equality comes from the centrality of $E$. Also, we have
     \[
        f'(h_{(1)})\alpha'_E(h_{(2)}) =  f(h_{(1)})E\Bigl(\bigl(h_{(2)}\cdot E\bigr)E\Bigr) 
          =  \Bigl(f(h_{(1)})\bigl(h_{(2)}\cdot E\bigr)\Bigr)E^2  
          =  (h\cdot E)E  
          =  \alpha'_E(h). 
    \] 
    Here, the second equality also comes from the centrality of $E$. Applying Theorem \ref{teorema.idempotent.convolution.coradical}, we deduce that $(h\cdot E)E = \alpha'_E(h) = f'(h) = f(h)E$, for all $h \in H$. Furthermore, we also have
    \begin{eqnarray*}
        & \, & \alpha_E(h_{(1)})\alpha'_E(h_{(2)})  =  (h_{(1)}\cdot E)(h_{(2)}\cdot E)E  \\
         & = & (h\cdot E)E  
          =  \alpha'_E(h)  
         =  (h\cdot E)E  \\
         & = & (h_{(1)}\cdot E)(h_{(2)}\cdot E)E  
          =  (h_{(1)}\cdot E)E(h_{(2)}\cdot E)  \\
         & = & \alpha'_E(h_{(1)})\alpha_E(h_{(2)}),
    \end{eqnarray*} 
    so that, by applying Theorem \ref{teorema.idempotent.convolution.coradical}, again, we obtain $(h\cdot E)E = \alpha'_E(h) = \alpha_E(h) = h\cdot E$, for all $h \in H$. Finally, we see that
    $$ h\cdot E = \alpha_E(h) = \alpha'_E(h) = f'(h) = f(h)E, $$
    which concludes the proof.
\end{proof}

\subsection{Symmetric left partial action of a pointed Hopf algebra}

\label{se:Partial_Action}

Consider $H$ a Hopf algebra, together with a symmetric partial left $H$-module algebra $A$. Recall from Definition \ref{definicao.acao.parcial} that a structure of (symmetric) left partial $H$-module algebra on $A$ is a linear map
$$ \cdot:H\x A \to A, \ \ h\x a \mapsto h\cdot a, $$
such that
$$ 1_H\cdot a = a; \ \ \ \ h\cdot(ab) = (h_{(1)}\cdot a)(h_{(2)}\cdot b), \ \ \ \ 
 h\cdot (k\cdot a) = (h_{(1)}\cdot 1\um)(h_{(2)}k\cdot a) = (h_{(1)}k\cdot a)(h_{(2)}\cdot 1\um), $$
for all $h,k \in H$ and $a,b \in A$. Note that if $g$ is a grouplike element, and $a,b \in A$, then $g\cdot (ab) = (g\cdot a)(g\cdot b)$. 
Also, for any $h \in H$, we have:
 \begin{eqnarray*}
     & \, & (g\cdot 1\um)(h\cdot 1\um)  =  (g\cdot 1\um)\bigl((gg^{-1})h\cdot 1\um\bigr)  \\
     & = & (g\cdot 1\um)\bigl(g(g^{-1}h)\cdot 1\um\bigr)  
     =  g\cdot \bigl((g^{-1}h)\cdot 1\um\bigr)  \\
     & = & \bigl(g(g^{-1}h)\cdot 1\um\bigr)(g\cdot 1\um) 
      =  (h \cdot 1\um)(g\cdot 1\um). \\
\end{eqnarray*} 
Hence $g\cdot 1\um$ is central on the subalgebra of $A$ generated by the elements of the form $h\cdot 1\um$, for $h \in H$. The next lemma shows that $g\cdot 1\um$ is in fact central in the whole algebra $A$, for any grouplike element $g$.
\begin{lema}\label{lema.g1.central.idempotent.of.A}
    Let $A$ be a symmetric partial left $H$-module algebra of a Hopf algebra $H$. If $g$ is a grouplike element of $H$ then $g\cdot 1\um$ is a central idempotent of $A$.
\end{lema}
\begin{proof}
    Given $a \in A$ and $g \in G(H),$ we have: 
    $$ 
        a(g\cdot 1\um) =  (gg^{-1}\cdot a)(g\cdot 1\um) 
          =  g\cdot (g^{-1} \cdot a)  
         =  (g\cdot 1\um)(gg^{-1}\cdot a)  
          =  (g\cdot 1\um)a,
    $$
    hence $g\cdot 1\um$ is central in $A$. The fact that $g\cdot 1\um$ is idempotent follows from
    $$ g\cdot 1\um = (g\cdot(1\um 1\um)) = (g\cdot 1\um)(g\cdot 1\um), $$
    which concludes the proof.
\end{proof}

In the sequel we shall assume that $H$ is a pointed Hopf algebra. We will denote by $G:=G(H)$ the group of grouplike elements of $H$, and we shall also assume that $G$ is a finite group. Also, we assume that $A$ is a symmetric partial left $H$-module algebra, where 
$$ \cdot:H\x A \to A, \ \ h\x a \mapsto h\cdot a, $$
is the left partial action of $H$ on $A$.

\begin{obs}\label{obs.Psa.PSAa}
   Let $G$ be a finite group and let $A$ be a partial symmetric $\K G$-module algebra; recall that the linear map $\pi : \K G \to End(A) $ defined by $\pi(g) (a) = g \cdot a$ is a partial representation of $\K G$. For any $g \in G$, consider the element $\eps(g) = \pi(g)\pi(g^{-1}) \in \End(A)$. In \cite{piccione}, the authors define, for any subset $S$ of $G$, the following element in $\End(A)$:
$$ P_S = \prod_{s\in S}\eps(s) \prod_{s\in G \smallsetminus S} (\textrm{Id}\um-\eps(s)). $$
Notice also that, for all $a \in A$ and $g \in G$,
$$ \eps(g)(a)  =  \pi(g)\pi(g^{-1})(a)  
      = g\cdot(g^{-1}\cdot a)  
      =  (g\cdot 1\um)(gg^{-1}\cdot a)  
      = (g\cdot 1\um)a \ ;
$$
therefore, we have that
\begin{equation}
P_S(a) = \prod_{s\in S}(s \cdot 1\um) \prod_{s \in G \smallsetminus S} (1\um - (s\cdot 1\um))a.      
\end{equation}
Now considering the elements of the form 
$$ P_S^A = \prod_{s \in S}
(s \cdot 1\um) 
\prod_{s \in G \smallsetminus S} (1\um- (s \cdot 1\um)) \in A, $$
then we have, 
\begin{equation}
    P_S(a) = P_S^Aa,
\end{equation}
for all $a \in A$.
\end{obs}

We shall now extend this construction to any Hopf algebra $H$. Since $(g\cdot 1\um)$ is a central idempotent of $A$ for all $g \in G = G(H)$, it follows that $\bigl(1\um-(g\cdot 1\um)\bigr)$ also is a central idempotent. Hence 
$$ 
     1\um  =  \displaystyle \prod_{g \in G}\one  
        = \displaystyle \prod_{g \in G}\bigl((g\cdot 1\um) + 1\um-(g\cdot 1\um)\bigr)  
     =  \displaystyle \sum_{X \subseteq G}\prod_{x \in X}(x\cdot
        1\um) \prod_{y\in G\smallsetminus X} (1\um- (y\cdot 1\um)).
 $$
 So, given $X \subseteq G$, we may define the following element of $A$,
 \begin{equation} \label{def.equation.PXA}
     P_X^A = \prod_{x\in X}(x\cdot 1\um)\prod_{y\in G \smallsetminus X}(\one-(y\cdot 1\um)). 
 \end{equation} 
 We easily see that if $y\cdot 1\um = \one$, then $\one-(y\cdot 1\um) = 0$, therefore $P_X^A = 0$ if $y \notin X$. Hence, we only have to take into account elements $P_X^A \in A$ for $X \subseteq G$ such that $1_H \in X$, given that $1_H\cdot a = a$, for all $a \in A$.

 Denote by $\Px_1(G) = \{X \subseteq G \mid 1_H \in X\}$, and notice that 
$$ \one = \sum_{X \in \Px_1(G)}P^A_X. $$

 One important thing to note is that $P_X^A$ can eventually be zero even when $X \in \Px_1(G)$, it only depends on the left partial action of $H$ on $A$. 
Besides, it follows from the definition of $P_X^A$ that 
\begin{equation}\label{equation.PxPy.equal.0.or.Px}
    P_X^AP_Y^A = \left\{\begin{array}{cc}
    P_X^A, & \textrm{if } X=Y, \\
    0, & \textrm{otherwise}.
\end{array}\right.
\end{equation}
Finally, note that the elements $P^A_X$ are  central idempotents of $A$, since each one is a product of central idempotents, and they are pairwise orthogonal, so we have the following.
\begin{proposicao}\label{proposicao.A.as.PxA.sommation}
    Let $H$ be a pointed Hopf algebra with finite group $G$ of grouplike elements. Then any symmetric partial left $H$-module algebra $A$ decomposes as an algebra, into a direct sum 
    $$ A = \bigoplus_{X\in \Px_1(G)} A\ P^A_X $$
    of unital ideals,     where $\Px_1(G) = \{X \subseteq G \mid 1_H \in X\}$, and 
    $$P_X^A = \prod_{x \in X}(x\cdot
        \one) \prod_{y\in G\smallsetminus X} (\one- (y\cdot 1\um)) $$
    are central idempotents of $A$, for all $X \in \Px_1(G)$.
\end{proposicao}
For all $X \in \Px_1(G)$, $A \, P_X^A$ is an ideal of $A$ generated by $P_X^A$. However, in general $A \, P_X^A$ will not be stable under the left $H$-action. 

\begin{proposicao} \label{prop:acao.grouplike.P^A_X}
    Let $H$ be a pointed Hopf algebra with finite group $G$ of grouplike elements. Given $g \in G$ and $a \in A$, we have
\begin{equation}
g\cdot (aP_X^A) = (g\cdot a)(g\cdot P_X^A) =  \left\{\begin{array}{cc}
    (g\cdot a)P_{gX}^A, & \textrm{if } g^{-1}\in X, \\
    0, & \textrm{otherwise}.
\end{array}\right. ,
\end{equation} 
 for all $X \in \Px_1(G)$. 
 \end{proposicao}

\begin{proof}
   Given $g \in G$, we have
 $$ \begin{array}{ccl}
     g\cdot P_X^A & = & \displaystyle g\cdot\Bigl[\prod_{x\in X}(x\cdot 1\um)\prod_{y\in G \setminus X}(\one-(y\cdot 1\um))\Bigr]  \\
      & = & \displaystyle \prod_{x\in X}g\cdot (x\cdot 1\um)\prod_{y\in G \setminus X}g\cdot\bigl(1\um-y\cdot 1\um\bigr)  \\
      & = & \displaystyle \prod_{x\in X}g\cdot (x\cdot 1\um)\prod_{y\in G \setminus X}\bigl((g\cdot 1\um)-g\cdot(y\cdot 1\um)\bigr)  \\
      & = & \displaystyle \prod_{x\in X}(gx\cdot 1\um)(g\cdot 1\um)\prod_{y\in G \setminus X}\bigl((g\cdot 1\um)-(g\cdot 1\um)(gy\cdot 1\um)\bigr)  \\
      & = & \displaystyle \prod_{x\in X}(gx\cdot 1\um)(g\cdot 1\um)\prod_{y\in G \setminus X}(g\cdot 1\um)\bigl(\one-(gy\cdot 1\um)\bigr)  \\
      & = & \displaystyle \prod_{x\in X}(gx\cdot 1\um)(g\cdot 1\um)\prod_{y\in G \setminus X}\bigl(\one-(gy\cdot 1\um)\bigr)  \\
      & = & \displaystyle \prod_{x\in X}(gx\cdot 1\um)\prod_{y\in G \setminus X}\bigl(\one-(gy\cdot 1\um)\bigr),
 \end{array} $$
 where in the last equality we used the fact that $x = 1_H \in X$.
  Hence,  
 \begin{equation}\label{equation.gPx.equal.0.or.Pgx}
     g\cdot P_X^A = \left\{\begin{array}{cl}
    P_{gX}^A, & \textrm{if } g^{-1} \in X, \\
    0, & \textrm{otherwise},
\end{array}\right.
 \end{equation}
then 
$$ g\cdot (aP_X^A) = (g\cdot a)(g\cdot P_X^A) = \left\{\begin{array}{cl}
    (g\cdot a)P_{gX}^A, & \textrm{if } g^{-1} \in X, \\
    0, & \textrm{otherwise}.\end{array}\right. $$
\end{proof}

We remark that equality \eqref{equation.gPx.equal.0.or.Pgx} provides the useful equality 
\begin{equation}
\label{equation.g-1.g.Px.equals.Px}
     g^{-1} \cdot (g\cdot P_X^A) = P_X^A \ \ \textrm{ when } g^{-1} \in X.
 \end{equation}
  \begin{proposicao}\label{proposicao.APX.iso.APgX}
    Let $H$ be a pointed Hopf algebra with finite group $G$ of grouplike elements. If $A$ is a symmetric partial left $H$-module algebra, and $X\in \Px_1(G)$, then for all $g \in G$, such that $g^{-1} \in X$, 
    $$ \begin{array}{cccl}
        g\cdot - :& AP^A_X & \larr & AP^A_{gX}  \\
         & aP^A_X & \lmap & g\cdot (aP^A_X) = (g\cdot a)P^A_{gX},
    \end{array} $$
    is an algebra isomorphism.
\end{proposicao}

\begin{proof}
    Since $g \in G$ is a grouplike element it follows that $g\cdot (ab) = (g\cdot a)(g\cdot b) $, for all $a,b \in A$, and  $g\cdot - :A \to A$ is a multiplicative map. By Proposition \ref{prop:acao.grouplike.P^A_X}, when $g^{-1} \in X$ the restriction of this map to $A P^A_X$ takes this ideal into $A P^{A}_{gX}$, and takes the unit $P^A_X$ of $AP^A_X$ to $(g\cdot 1_A)P^{A}_{gX} = P^{A}_{gX}$, which is the unit of $AP^A_{gX}$; hence we have an algebra morphism $g\cdot - :A P^A_X \to A P^{A}_{gX}$. Since $g \in gX$ the map $g^{-1}\cdot - : AP^A_{gX} \to AP^A_{X}$ is the inverse of $g\cdot - :AP^A_X \to AP^A_{gX}$.
\end{proof}
\black
Proposition \ref{prop:acao.grouplike.P^A_X} shows that if $g^{-1} \in X$ then $g \cdot P^A_X = P^A_{gX}$. It follows immediately that if $g$ satisfies the equality $g \cdot P^A_X = P^A_{X}$ (that is, $g$ fixes $P^A_X$) then it comes from the \textit{isotropy subgroup} of $X$
$$ G_X = \{g\in G \mid g^{-1} \in X, \  gX = X\} $$ 
which is a group for every $X \in \mathcal{P}_1(G)$ (see  \cite{piccione}). We also remark that if $ g^{-1} \in X$ then the conjugation map $g(-)g^{-1}:G_X \to G_{gX}$, given by $h \mapsto ghg^{-1}$, is an isomorphism of groups.

\begin{lema}\label{lema.hGammaA.h1GammaA}
    Let $H$ be a pointed Hopf algebra with finite group $G$ of grouplike elements. Let $A$ be a symmetric partial left $H$-module algebra. Given $X \in \mathcal{P}_1(G)$, the element
    \begin{equation} \label{def:equation.GammaXA}
    \Gamma^A_X = \dfrac{1}{|G_X|}\sum_{g\in G} g\cdot P^A_X = \dfrac{1}{|G_X|}\sum_{g^{-1}\in X}P^A_{gX},     
    \end{equation} 
    
    is a central idempotent element of $A$, and satisfies
    $$ h\cdot \Gamma^A_X = (h\cdot 1\um)\Gamma^A_X, \ \textrm{ for all } h\in H. $$
\end{lema}

\begin{proof}
    First, we show that $\Gamma^A_X$ is an idempotent element. We have: 
    $$ \Gamma^A_X\Gamma^A_X  = \dfrac{1}{|G_X|}\dfrac{1}{|G_X|}\sum_{g^{-1}\in X}P^A_{gX}\sum_{k^{-1}\in X}P^A_{kX} = \dfrac{1}{|G_X|^2}\sum_{g^{-1}\in X}\sum_{k^{-1}\in X}P^A_{gX}P^A_{kX}. $$
    Next, $P^A_{gX}P^A_{kX}=P^A_{gX}$ if and only if $gX=kX$ or $g^{-1}kX=X$, otherwise $P^A_{gX}P^A_{kX}$ vanishes, therefore:
    $$ \Gamma^A_X\Gamma^A_X = \displaystyle 
        \dfrac{1}{|G_X|^2}\sum_{g^{-1}\in X}\sum_{{k^{-1}\in X,} \atop {g^{-1}k\in G_X}}P^A_{gX}. $$
    If we consider $y:=g^{-1}k \in G_X$, then $k=gy$ and $k^{-1} = y^{-1}g^{-1}\in y^{-1}X = X$, hence
    $$ 
        \Gamma^A_X\Gamma^A_X  =  
        \displaystyle \dfrac{1}{|G_X|^2}\sum_{g^{-1}\in X}\sum_{y\in G_X}P^A_{gX}  
         =  \displaystyle \dfrac{1}{|G_X|^2}\sum_{g^{-1}\in X}|G_X|P^A_{gX}  
         =  \displaystyle \dfrac{1}{|G_X|}\sum_{g^{-1}\in X}P^A_{gX} \;  = \;   \Gamma^A_X. 
    $$
    We shall now prove that $h\cdot\Gamma^A_X = (h\cdot 1\um)\Gamma^A_X$. Consider the map $({-}\cdot 1\um):H \to A$, given by $h \to (h\cdot 1\um)$, which is a convolution idempotent map, satisfying:
    $$ \left\{\begin{array}{rcl}
        h\cdot \Gamma^A_X & = & (h_{(1)}\cdot 1\um)(h_{(2)}\cdot \Gamma^A_X), \\
        h\cdot \Gamma^A_X & = & (h_{(1)}\cdot \Gamma^A_X)(h_{(2)}\cdot 1\um),
    \end{array}\right. $$
    for all $h \in H$. Given $k \in  G$, we have
    \begin{eqnarray*}
        & \, & k\cdot\Gamma^A_X =  \displaystyle \dfrac{1}{|G_X|}\sum_{g\in G}k\cdot\bigl(g\cdot P^A_X\bigr)  \\
        & = & \displaystyle \dfrac{1}{|G_X|}\sum_{g\in G}(k\cdot 1\um)(kg\cdot P^A_X)  
        =  \displaystyle(k\cdot 1\um)\biggl(\dfrac{1}{|G_X|}\sum_{g\in G}(kg\cdot P^A_X)\biggr)  \\
         & = & \displaystyle(k\cdot 1\um)\biggl(\dfrac{1}{|G_X|}\sum_{y\in G}(y\cdot P^A_X)\biggr)  
         =  (k\cdot 1\um)\Gamma^A_X.
    \end{eqnarray*}
     Hence for any $x \in H_0=\K G$, we have $x\cdot \Gamma^A_X = (x\cdot 1\um)\Gamma^A_X$. Since $\Gamma^A_X$ is a central idempotent element in $A$, from Lemma \ref{lema.partialaction.idempotent.fE} we conclude that 
     $$h \cdot \Gamma^A_X = (h\cdot 1\um)\Gamma^A_X,$$
     for all $h \in H$.
\end{proof}

\begin{obs}
   Since $g \cdot P^A_X = P^A_{gX} $ when $g^{-1} \in X$, the partial left $H$-action on $A$ induces a partial $G$-action on the set $\{P^A_X; X \in \mathcal{P}_1(G) \}$ which is isomorphic to the canonical partial $G$-action on $\mathcal{P}_1(G)$; we may view the definition of the idempotents $\Gamma_X^A$ as the result of ``taking the average'' of the idempotents $P^A_X$ under the partial $G$-action. {If $L$ is a subgroup of $G$, by definition 
   \begin{equation}\label{equation.GammaL.igual.PL}
       \Gamma_L^A = P_L^A,
   \end{equation} 
   given that the $G$-partial action stabilizes the subgroups of $G$.}  See \cite{DE05} for more details on partial actions of groups on sets and algebras. 
\end{obs}

     By Proposition \ref{proposicao.A.as.PxA.sommation}, we have that 
$$ \one = \sum_{X \in \mathcal{P}_1(G)}P^A_X. $$ 
The next definition helps us to better understand the elements $\Gamma^A_X$ for any $X \in \Px_1(G)$.
\begin{definicao}\label{equivalence.relation}
    Let $G$ be a finite group, consider the set $\Px_1(G) = \{X \subseteq G \mid 1_G \in X\}$. The canonical $G$-partial action on $\Px_1(G)$, given by the left multiplication, induces an equivalence relation on $\Px_1(G)$: 
    \[ X \sim Y \textrm{ if there exists } g \in G \textrm{ such that } g^{-1}\in X \textrm{ and }  gX = Y.  \]
\end{definicao}
Let $\mathcal{O}_1, \mathcal{O}_2,  \ldots , \mathcal{O}_n$ denote the distinct equivalence classes of the equivalence relation defined above. Since $\mathcal{P}_1(G) = \bigcupdot_{k=1}^n\mathcal{O}_k$ it follows that
$$ \one = \sum_{k=1}^n\Bigl(\sum_{X \in \mathcal{O}_k}P^A_X\Bigr). $$
It is not hard to see that
$$ \Gamma^A_X = \dfrac{1}{|G_X|}\sum_{g^{-1}\in X}P^A_{gX} = \sum_{Y\sim X} P^A_Y, $$
for any $X \in \mathcal{P}_1(G)$. Then, choosing representatives $X_1, X_2,  \ldots , X_n$ in $\mathcal{P}_1(G)$ with $X_k \in \mathcal{O}_k$ for all $k=1,2, \ldots ,n$, we have:
$$ \one = \sum_{k=1}^n\Bigl(\sum_{X \sim X_k}P^A_X\Bigr) = \sum_{k=1}^n\Gamma^A_{X_k}. $$
Finally, by Lemma \ref{lema.hGammaA.h1GammaA}, $\Gamma^A_X$ is a central idempotent of $A$ and we have 
$$ A = \bigoplus_{k=1}^n \bigl(A\Gamma^A_{X_k}\bigr), $$
and again by Lemma \ref{lema.hGammaA.h1GammaA}, we have
$$ h\cdot(a\Gamma_X^A) = (h_{(1)}\cdot a)\bigl(h_{(2)}\cdot \Gamma_X^A\bigr) = (h_{(1)}\cdot a)(h_{(2)}\cdot 1\um)\Gamma_{X}^A = (h\cdot a)\Gamma_X^A, $$
therefore for all $X \in \Px_1(G)$ the ideal $A \, \Gamma_X^A$ generated by $\Gamma_X^A$ is stable under the $H$-action.

In other words, we have the main theorem of this section.

\begin{teorema}\label{teorema.A.GammaA.direct.somma}   
Let $H$ be a pointed Hopf algebra with finite group $G$ of grouplike elements, and let $A$ be a symmetric partial left $H$-module algebra.  If $X_1, X_2,  \ldots , X_n$ constitute a complete set of representatives in $\mathcal{P}_1(G)$ for the equivalence relation introduced in Definition \ref{equivalence.relation}, then
    $$ A = \bigoplus_{k=1}^{n}(A \, \Gamma^A_{X_k}) $$
    where $A \, \Gamma^A_{X_k}$ is the ideal of $A$ generated by $\Gamma_{X_k}^A$, which is stable under the left partial action of $H$, for all $k=1,2,  \ldots , n$. Moreover, for every $X \in \Px_1(G)$, every $a \in A$ and $h \in H$,  
    $$ h\cdot(a\Gamma^A_X) = (h\cdot a)\Gamma^A_X.$$   
\end{teorema}

We close this section with a consequence of this theorem that will be key to Section \ref{section:PG-component}, where we show that $H$ is isomorphic to a unital ideal of $\Hpar$ and that each partial $H$-module is the direct sum of a global $H$-module and a purely partial $H$-module (that is, a partial $H$-module that does not contain a  nontrivial global $H$-submodule).

\begin{corolario}\label{lema.hGammaAG.epshGammaAG} 
Let $H$ be a pointed Hopf algebra with finite group $G$ of grouplike elements, and let $A$ be a symmetric partial left $H$-module algebra.
 We have that 
    $$ h\cdot \Gamma^A_G = \epsilon(h)\Gamma^A_G,$$ 
    for all $h \in H$.
\end{corolario}

\begin{proof}
    Notice that if $X \in P_1(G)$ is a subgroup of $G$, then $G_X = X$ and therefore
$$ \Gamma_X^A = \dfrac{1}{|G_X|}\sum_{g^{-1}\in X}P^A_{gX} = \dfrac{1}{|X|}\sum_{g\in X}P^A_{X} = P^A_X = \prod_{x\in X}(x\cdot 1\um)\prod_{y\in G\setminus X}(\one-(y\cdot 1\um)). $$
In particular, for $X=G$, 
$$ \Gamma_G^A = P_G^A = \prod_{g\in G}(g\cdot 1\um) $$
and, for each $k \in G$,
$$ k\cdot \Gamma^A_G = \displaystyle (k\cdot
     1\um)\Gamma^A_G = \displaystyle (k\cdot 1\um)\prod_{g \in G}(g\cdot 1\um) = \displaystyle \prod_{g \in G}(g\cdot 1\um) = \epsilon(k)\Gamma^A_G. $$

On the other hand, we also have
$$  h\cdot \Gamma^A_G = \epsilon(h_{(1)})\bigl(h_{(2)}\cdot \Gamma^A_G\bigr) = \bigl(h_{(1)}\cdot \Gamma^A_G\bigr)\epsilon(h_{(2)}), $$
for all $h \in H$, so we conclude by Lemma \ref{lema.partialaction.idempotent.fE} that
$$ h\cdot \Gamma^A_G = \epsilon(h)\Gamma^A_G, \ \ \textrm{for all }h \in H. $$
\end{proof}

	\subsection{On the structure of $A_{par}$} \label{se:Apar_Pointed}

 The aim of this subsection is to study the algebra $\apar$ when $H$ is a pointed Hopf algebra  with finite group of grouplikes $G$.

The linear map $[-]:H \to \Hpar$ defined by $H \ni h \mapsto [h] \in \Hpar$ is a partial representation. The corresponding convolution idempotent map $\epsilon^{[-]}:H \to \Hpar$ is actually the linear map
\[
    \epsilon_{-}:  H  \larr  \Hpar, \ \ \ \ 
      h  \lmap  \epsilon_h=[h_{(1)}][S(h_{(2)})].
\]
Let us denote by $A$ the subalgebra $A = \apar=\langle \epsilon_h \in H \mid h \in H \rangle$ of $\Hpar$. By Proposition \ref{proposicao.epsg.idempotent.central}, if $g \in G$ then  $\epsilon_g$ is a central idempotent element in $A$. Recall that, for any $h \in H$ and $a=\epsilon_{k^1}  \widedots  \epsilon_{k^n} \in A$, the element
$$ h\cdot a = [h_{(1)}]a[S(h_{(2)})], $$
is also an element of $A$ since, by equality \eqref{formula.H.action.on.Apar},
$$ h\cdot a = \epsilon_{h_{(1)}k^1}  \widedots   \, \epsilon_{h_{(n)}k^n}\epsilon_{h_{(n+1)}}. $$

\bigskip

The linear map $H\x A \to A$ that takes $h\x a$ to $h\cdot a$ defines a \emph{symmetric left partial action} of $H$ on $A$ (see \cite[Thm 4.8]{alves2015partial}). Hence the results of Section \ref{se:Partial_Action} can be incorporated to the case $A=\apar$. 

Given $X \in \Px_1(G)$, we shall simplify the notation of the associated idempotents defined in \eqref{def.equation.PXA} and \eqref{def:equation.GammaXA} by dropping the superscript ``$A$'',  and we will use the notations $P_X = P_X^A$ and $\Gamma_X = \Gamma_X^A$.
We note that since $h\cdot 1\um = \eps_h$ for all $h \in H$ we may write 
$$ P_X  = \prod_{x\in X}\epsilon_x\prod_{y\in G \smallsetminus X}(1\um -\epsilon_y). $$

In order to describe $H_{par}$ we do need to describe the subalgebra $A= A_{par}$. 
Since we have the direct sum decomposition 
\[
A \Gamma_{X} = \bigoplus_{g^{-1} \in X} A P_{gX} 
\]
as an algebra, it is enough to describe each (sub)component $A P_{gX}$. 
We have already shown that $A P_X \simeq A P_Y$ when $X \sim Y$ in $\Px_1(G)$  (Prop. \ref{proposicao.APX.iso.APgX}). 
In this subsection we will prove that $A P_X \simeq A P_{G_X}$ as algebras.

We begin by the construction of an induced partial action which extends Proposition 1 of \cite{alves2010enveloping}.

\begin{proposicao} \label{prop:induced.partial.action}
  Let $B$ be a partial left $H$-module algebra, with partial action given by $h\x b \mapsto h\cdot b$. Let $E$ be a central idempotent of $B$ and let $B'=\langle E \rangle$ be the ideal generated by $E$. Then the map 
$$ \begin{array}{cccl}
    \rhd: & H\x B' & \larr & B', \\
     & h\x b' & \lmap & h\rhd b' = (h\cdot b')E,
\end{array} $$  
is a partial $H$-action on $B'$, which is symmetric if the original partial action is symmetric. 
\end{proposicao}
\begin{proof}
The verification of (PA1) and (PA2) of Definition \ref{definicao.acao.parcial} is straightforward. For (PA3), 
given $h, k \in H$ and $b \in B'$, 
 \begin{eqnarray*}
    & \, & h \rhd (k \rhd b)  =  [h\cdot((k\cdot b)E)]E \\
     & = & [h\cdot(E(k\cdot b))]E 
    =  (h_{(1)}\cdot E)\bigl(h_{(2)}\cdot (k\cdot b)\bigr)E \\
     & = & (h_{(1)}\cdot E)(h_{(2)}\cdot 1_B)(h_{(3)}k\cdot b)E 
    =  (h_{(1)}\cdot E)E(h_{(2)}k\cdot b)E \\
    & = & (h_{(1)}\rhd E)(h_{(2)}k \rhd b),
    \end{eqnarray*}    
    and since $E$ is the unit of $B'$, we have that $h\rhd (k\rhd b) = (h_{(1)}\rhd 1_{B'})(h_{(2)}k\rhd b)$. An analogous computation shows that 
    if the partial action of $H$ on $B$ is symmetric, then the partial action of $H$ on $B'$ is also symmetric.
\end{proof}

Let us apply this result to the algebra $A= \apar = \langle \eps_h \mid h \in H \rangle \subseteq \Hpar$, which is a symmetric partial $H$-module algebra via 
$$ h\cdot a =[h_{(1)}]a[S(h_{(2)})], \ \ \textrm{ for } h \in H \textrm{ and } a \in A. $$
We recall that, by Proposition \ref{proposicao.A.as.PxA.sommation}, we have a decomposition of $A$ as a direct sum of unital ideals
$$A = \bigoplus_{X\in \Px_1(G)} AP_X. $$
For each $X \in \Px_1(G)$, the element $P_X$ is a central idempotent of $A$; by Proposition \ref{prop:induced.partial.action}, the linear map 
$$ \begin{array}{ccl}
     H\x AP_X & \larr & AP_X, \\
     h\x aP_X & \lmap & \bigl(h\cdot(aP_X)\bigr)P_X,
\end{array} $$
is a symmetric partial action of $H$ on $AP_X$.

\begin{lema}\label{lema.varphiX.algebra.sobre}
    Let $H$ be a pointed Hopf algebra finite group $G$ of grouplike elements.  If $A = \apar  \subseteq \Hpar$, and $X\in \Px_1(G)$, then 
    $$ \varphi_{X}: A  \larr AP_X, \ \textrm{ defined by } \ \varphi_X(\eps_h) = (h\cdot P_X)P_X, $$
    is a surjective algebra map.
    Also $AP_{Y} \subseteq \ker\varphi_X$, for all $Y \neq G_X$.
\end{lema}
\begin{proof}
The map $\varphi_X$ will be constructed via Theorem \ref{teorema.A.propriedade.universal}. Consider the map
    $$ \begin{array}{cccl}
       f: & H & \larr & AP_X \\
          & h & \lmap & (h\cdot P_X)P_X.
    \end{array} $$
    It is clear that
    $$ f(1_H) = P_X = 1_{AP_X}, \ \ \ \ f(h) = (h\cdot P_X)P_X = (h_{(1)}\cdot P_X)(h_{(2)}\cdot P_X)P_X = f(h_{(1)})f(h_{(2)}), $$
    and it is not hard to see that the equality $$f(h_{(1)}k)f(h_{(2)}) =  f(h_{(1)})f(h_{(2)}k)$$ holds for all $h,k \in H$:
    \begin{eqnarray*}
        & \, & f(h_{(1)}k)f(h_{(2)})  =  (h_{(1)}k\cdot P_X)(h_{(2)}\cdot P_X)P_X  \\
         & = & (h_{(1)}k\cdot P_X)(h_{(2)}\cdot 1_A)(h_{(3)}\cdot P_X)P_X  
          =  \bigl(h_{(1)}\cdot (k\cdot P_X)\bigr)(h_{(2)}\cdot P_X)P_X  \\
         & = & \bigl[h\cdot ((k\cdot P_X)P_X)\bigr]P_X  
          =  \bigl[h\cdot (P_X(k\cdot P_X))\bigr]P_X  \\
         & = & (h_{(1)}\cdot P_X)\bigl(h_{(2)}\cdot (k\cdot P_X)\bigr)P_X  
          =  (h_{(1)}\cdot P_X)(h_{(2)}\cdot 1_A)(h_{(3)}k\cdot P_X)P_X  \\
         & = & (h_{(1)}\cdot P_X)(h_{(2)}k\cdot P_X)P_X  
          =  f(h_{(1)})f(h_{(2)}k).
    \end{eqnarray*} 
    It follows from Theorem \ref{teorema.A.propriedade.universal}  that there is an algebra map $\varphi_X:A \to AP_X$, given by $\varphi_X(\eps_h) = f(h) = (h\cdot P_X)P_X$, for all $h \in H$. 

    We now prove that this is a surjective map. It is enough to show that the algebra generators $\eps_hP_X$ lie in the image of $\varphi_X$. By Lemma \ref{lema.hGammaA.h1GammaA} the central idempotent 
    $$\Gamma_X = \displaystyle \dfrac{1}{|G_X|}\sum_{g^{-1}\in X}(g\cdot P_{X})$$ 
    satisfies the equality $h\cdot \Gamma_X = \eps_h\Gamma_X$ for each $h \in H$; hence
     \begin{eqnarray*}
        \eps_hP_X & = & \eps_h(\Gamma_XP_X) 
          =  (\eps_h\Gamma_X)P_X  
          =  (h\cdot \Gamma_X)P_X  \\
         & = & \displaystyle \dfrac{1}{|G_X|}\sum_{g^{-1}\in X}\bigl(h\cdot (g\cdot P_{X})\bigr)P_X
         \\ & \stackrel{\eqref{equation.g-1.g.Px.equals.Px}}{=} &          \displaystyle \dfrac{1}{|G_X|}\sum_{g^{-1}\in X}\bigl(h\cdot (g\cdot g^{-1} \cdot (g \cdot P_{X}))\bigr)P_X  \\
         & \stackrel{\eqref{eq.ex.1.1-PR2}}{=} & \displaystyle \dfrac{1}{|G_X|}\sum_{g^{-1}\in X}\bigl(hg\cdot g^{-1} \cdot (g \cdot P_{X}))\bigr)P_X  \\
         & = & \displaystyle \dfrac{1}{|G_X|}\sum_{g^{-1}\in X}\bigl(hg\cdot P_{X}\bigr)P_X  \\
         & = & \displaystyle \dfrac{1}{|G_X|}\sum_{g^{-1}\in X}f(hg)  
          =  \displaystyle \dfrac{1}{|G_X|}\sum_{g^{-1}\in X}\varphi_X(\eps_{hg}), 
    \end{eqnarray*} 
    which concludes the proof that $\varphi_X$ is an algebra map.

    Now, by Proposition \ref{proposicao.A.as.PxA.sommation}, we have that $\displaystyle A = \bigoplus_{Y\in \Px_1(G)} AP_Y $; given $g \in G$, we have that
    $$ \varphi_X(\eps_g) = (g\cdot P_X)P_X = \left\{\begin{array}{cl}
        P_X, & \textrm{if } g\in G_X,  \\
        0, & \textrm{otherwise.}
    \end{array}\right. $$
    
    It follows that 
    $$ \varphi_X(P_Y) = \prod_{y\in Y}\varphi_X(\eps_y)\prod_{z\in G\smallsetminus Y}(P_X-\varphi_X(\eps_z)). $$

    If $Y \not\subseteq G_X$, let $y_0\in Y \setminus G_X$, then $\varphi_X(\eps_{y_0}) = 0$ and $\varphi_X(P_Y)=0$. Analogously, if $G_X \not \subseteq Y$, let $z_0 \in G_X \setminus Y$; then $P_X - \varphi_X (\eps_{z_0}) = 0$ and $\varphi_X(P_Y)=0$. Finally, when $Y = G_X$,  $\varphi_X(\eps_y) = P_X$  for each  $y\in Y$ and $\varphi_X(\eps_z) = 0$ for each $z\notin Y$, therefore 
    $$ \varphi_X(P_Y) = \prod_{y\in Y}\varphi_X(\eps_y)\prod_{z\in G\smallsetminus Y}(P_X-\varphi_X(\eps_z)) = \prod_{y\in Y}P_X\prod_{z\in G\smallsetminus Y}(P_X-0) = P_X. $$

    We conclude that  
    $$ \varphi_X(P_Y) =  \left\{\begin{array}{cl}
        P_X, & \textrm{if } Y = G_X,  \\
        0, & \textrm{otherwise,}
    \end{array}\right. $$
and it follows that $AP_Y \subseteq \ker\varphi_X$ when $Y \neq G_X$.

\end{proof}
    The previous result implies that 
    $$ \bigoplus_{{Y\in \Px_1(G),}\atop {Y \neq G_X}}AP_Y \subseteq \ker\varphi_X, $$
    and it follows that the algebra map $\varphi_X|_{AP_{G_X}}:AP_{G_X} \to AP_X$ remains surjective. We will prove that this map is an \emph{isomorphism} of algebras by constructing its inverse.

    \begin{lema}\label{lema.psiX.algebra.ker}
Let $H$ be a pointed Hopf algebra with finite group $G$ of grouplike elements.  If $A = \apar \subseteq \Hpar$, and $X\in \Px_1(G)$, then 
    $$ \psi_{X}: A  \larr AP_{G_X}, \ \textrm{ defined by } \ \psi_X(\eps_h) = \dfrac{1}{|G_X|}\sum_{g^{-1}\in X}\eps_{hg}P_{G_X}, $$
    is an algebra map. Also $AP_{Y} \subseteq \ker\psi_X$, for all $Y \neq X$.
    \end{lema}

\begin{proof}
   Consider the map
    $$ \begin{array}{cccl}
       f': & H & \larr & AP_{G_X}, \\
          & h & \lmap & \displaystyle \dfrac{1}{|G_X|}\sum_{g^{-1}\in X}\eps_{hg}P_{G_X}.
    \end{array} $$
    Since $G_X = \{t\in G \mid tX=X\} \subseteq X$ is a subgroup of $G$, we have that $P_{G_X} = \Gamma_{G_X}$ and it follows from Lemma \ref{lema.hGammaA.h1GammaA} that
    $$ h\cdot P_{G_X} = \eps_hP_{G_X}, \ \forall h \in H; $$
    hence $f'$ can be rewritten as 
    $$ f'(h) = \dfrac{1}{|G_X|}\sum_{g^{-1}\in X}\eps_{hg}P_{G_X} = \dfrac{1}{|G_X|}\sum_{g^{-1}\in X}hg\cdot P_{G_X}. $$
    We will use Theorem \ref{teorema.A.propriedade.universal} to prove now that $f':H \to AP_{G_X}$ induces an algebra map $\psi_X:A\to AP_{G_X}$, defined by $\psi_X(\eps_h) = f'(h)$, for all $h \in H$. We proceed to proving that equalities  \eqref{eq:Arelation11}, \eqref{eq:Arelation22} and \eqref{eq:Arelation33} hold for $f'$. 
    
    First, notice that 
    $$ f'(1_H) = \dfrac{1}{|G_X|}\sum_{g^{-1}\in X}\eps_{g}P_{G_X} = \dfrac{1}{|G_X|}\sum_{g^{-1}\in G_X} P_{G_X} = P_{G_X} = 1_{AP_{G_X}}. $$
    
    Furthermore, $f'$ is a convolution idempotent in $\Hom(H,AP_{G_X})$: given $h \in H$,
    \begin{eqnarray*}
        & \, & f'(h_{(1)})f'(h_{(2)})  =  \displaystyle \dfrac{1}{|G_X|^2}\sum_{{g^{-1}\in X,}\atop{k^{-1}\in X}}(h_{(1)}g\cdot P_{G_X})(h_{(2)}k\cdot P_{G_X})  \\
        & = & \displaystyle \dfrac{1}{|G_X|^2}\sum_{{g^{-1}\in X,}\atop{k^{-1}\in X}}(h_{(1)}g\cdot P_{G_X})(h_{(2)}g\cdot 1_H)(h_{(3)}(gg^{-1})k\cdot P_{G_X})  \\
        & = & \displaystyle \dfrac{1}{|G_X|^2}\sum_{{g^{-1}\in X,}\atop{k^{-1}\in X}}(h_{(1)}g\cdot P_{G_X})\bigl(h_{(2)}g\cdot(g^{-1}k\cdot P_{G_X})\bigr)  \\
        & = & \displaystyle \dfrac{1}{|G_X|^2}\sum_{{g^{-1}\in X,}\atop{k^{-1}\in X}} hg\cdot \bigl(P_{G_X}(g^{-1}k\cdot P_{G_X})\bigr)  
         =  \displaystyle \dfrac{1}{|G_X|^2}\sum_{{g^{-1}\in X,}\atop{k^{-1}\in X}} hg\cdot \bigl(\eps_{g^{-1}k}P_{G_X}\bigr)  \\
        & = & \displaystyle \dfrac{1}{|G_X|^2}\sum_{g^{-1}\in X}\Bigl(\sum_{k^{-1}\in X} hg\cdot \bigl(\eps_{g^{-1}k}P_{G_X}\bigr)\Bigr)
         \stackrel{x=g^{-1}k}{=}  \displaystyle \dfrac{1}{|G_X|^2}\sum_{g^{-1}\in X}\Bigl(\sum_{x^{-1}g^{-1}\in X} hg\cdot \bigl(\eps_{x}P_{G_X}\bigr)\Bigr). \\
        \end{eqnarray*} 
        
        Notice that $\eps_{x}P_{G_X} = 0$ if $x\notin G_X$, so the nonzero terms in the last sum satisfy $x\in G_X$ and, in this case,  $x^{-1}g^{-1} \in x^{-1}X = X$; therefore,  

          \begin{eqnarray*}
        & \, & f'(h_{(1)})f'(h_{(2)})  =  \displaystyle \dfrac{1}{|G_X|^2}\sum_{g^{-1}\in X}\Bigl(\sum_{x^{-1}g^{-1}\in X} hg\cdot \bigl(\eps_{x}P_{G_X}\bigr)\Bigr)  \\
        & = & \displaystyle \dfrac{1}{|G_X|^2}\sum_{g^{-1}\in X}\Bigl(\sum_{x\in G_X} hg\cdot \bigl(\eps_{x}P_{G_X}\bigr)\Bigr)  \\
         & \stackrel{\eps_xP_{G_X} = P_{G_X}}{=} & \displaystyle \dfrac{|G_X|}{|G_X|^2}\sum_{g^{-1}\in X} hg\cdot \bigl(P_{G_X}\bigr)  \\
        & = & \displaystyle \dfrac{1}{|G_X|}\sum_{g^{-1}\in X} hg\cdot \bigl(P_{G_X}\bigr) \ = \ f'(h).
    \end{eqnarray*} 
    
    Finally, consider $h,t \in H$; we have that
    \begin{eqnarray*}
        & \, & f'(h_{(1)}t)f'(h_{(2)})  =  \displaystyle \dfrac{1}{|G_X|^2}\sum_{{g^{-1}\in X,}\atop{k^{-1}\in X}}\eps_{h_{(1)}tg}\eps_{h_{(2)}k}P_X  \\
        & = & \displaystyle \dfrac{1}{|G_X|^2}\sum_{{g^{-1}\in X,}\atop{k^{-1}\in X}}\eps_{h_{(1)}k(k^{-1}tg)}\eps_{h_{(2)}k}P_X  
         =  \displaystyle \dfrac{1}{|G_X|^2}\sum_{{g^{-1}\in X,}\atop{k^{-1}\in X}}\eps_{h_{(1)}k}\eps_{h_{(2)}k(k^{-1}tg)}P_X \\ 
         & = &  \displaystyle \dfrac{1}{|G_X|^2}\sum_{{g^{-1}\in X,}\atop{k^{-1}\in X}}\eps_{h_{(1)}k}\eps_{h_{(2)}tg}P_X  
         =  f'(h_{(1)})f'(h_{(2)}t).
    \end{eqnarray*} 
    By Theorem \ref{teorema.A.propriedade.universal}, there is a unital algebra map 
    $$ \psi_{X}:A \to AP_{G_X}, \ \ \textrm{defined by } \psi_X(\eps_h) = \dfrac{1}{|G_X|}\sum_{g^{-1}\in X}\eps_{hg}P_{G_X}. $$

To conclude the proof of this lemma, let us examine the kernel of the map $\psi_X:A \to AP_X$. Let us study the image of $\eps_k$ for $k \in G$, which is given by  
    $$ \psi_X(\eps_k) = \dfrac{1}{|G_X|}\sum_{g^{-1}\in X}\eps_{kg}P_{G_X}. $$
    By definition,  $\eps_{kg}P_{G_X} \neq 0$ if and only if $kg \in G_X$, and if $kg \in G_X$ then  $k=(kg)g^{-1} \in kgX = X$; hence, if $k \notin X$ then $\eps_{kg}P_X=0$ for all $g^{-1}\in X$ and 
    $$ \psi_X(\eps_k) = \dfrac{1}{|G_X|}\sum_{g^{-1}\in X}\eps_{kg}P_X = 0. $$
    Assume now that $k \in X$. For all $x\in G_X$ we have that $g^{-1}=xk \in X$, and then $kg = k(k^{-1}x^{-1}) = x^{-1} \in G_X$. Since $G_X$ is a subgroup of $G$, then 
    $$ \psi_X(k) = \dfrac{1}{|G_X|}\sum_{g^{-1}\in X}\eps_{kg}P_{G_X} = \dfrac{1}{|G_X|}\sum_{x\in G_X}\eps_{x}P_{G_X} = \dfrac{1}{|G_X|}\sum_{x\in G_X}P_{G_X} = P_{G_X} = 1_{AP_{G_X}}. $$
    Finally, we have that 
    $$ \psi_X(P_Y) = \left\{\begin{array}{cl}
        P_{G_X}, & \textrm{if } Y=X  \\
        0, & \textrm{otherwise},
    \end{array}\right. $$
    which provides the inclusion  
    $$ \bigoplus_{{Y \in \Px_1(G)}\atop{Y \neq X}}AP_{Y} \subseteq \ker\psi_X. $$    
\end{proof}
    
We will show that $\psi_X \mid_{A P_X} : A P_X \to A P_{G_X}$ is a left inverse to $\varphi_X\mid_{A P_{G_X}} : A P_{G_X} \to A P_X$. In order to prove that we will derive another formula for $\varphi_X$.

\begin{lema}\label{lema.varphiX(h).formula}
    Let $H$ be a pointed Hopf algebra with finite group $G$ of grouplike elements.  If $A = \langle \eps_h \mid h \in H\rangle \subseteq \Hpar$, for all $X \in \Px_1(G)$, if $X=\{g_1, g_2,  \ldots , g_n\}$, then 
    \[
        \varphi_X(\eps_h) = (h\cdot P_X)P_X = \eps_{h_{(1)}g_1}\eps_{h_{(2)}g_2}   \widedots  \eps_{h_{(n)}g_n}\eps_{h_{(n+1)}}P_X,
    \]
    for all $h \in H$.
\end{lema}

\begin{proof}
    First, for all $X \in \Px_1(G)$, we define the element $\displaystyle Q_X = \prod_{x\in X}\eps_x$. This is a central idempotent of $A$ and it satisfies the equalities
    \[
    Q_XP_Y = \left\{\begin{array}{cl}
        P_Y, & \textrm{ if } X \subseteq Y,  \\
        0, & \textrm{ otherwise. }
    \end{array}\right.
    \]
    Moreover, if $X=\{g_1, g_2,  \ldots , g_n\}$ then 
    $$ h\cdot Q_X = \eps_{h_{(1)}g_1}\eps_{h_{(2)}g_2}  \widedots  \eps_{h_{(n)}g_n}\eps_{h_{(n+1)}}, $$
    so what we want to prove is that
    \begin{equation}\label{equation.hPXPX.iqual.HQXPX}
        (h\cdot P_X)P_X = (h\cdot Q_X)P_X,
    \end{equation}
    for all $h \in H$.
    It is clear that it holds for $h = g \in G$, since 
    \[
    (g\cdot Q_X)P_X = Q_{gX\cup \{1\}}P_X = \left\{\begin{array}{cc}
        P_X, & \textrm{ if } gX = X \\
        0, & \textrm{ otherwise. }
    \end{array}
    \right.
    \]
    The maps $h \lmap (h\cdot P_X)P_X$ and $h\lmap (h\cdot Q_X)P_X$ are convolution idempotent, and satisfy the equalities
    $$ \left\{\begin{array}{ccl}
        \left[(h_{(1)}\cdot Q_X)P_X\right]\left[(h_{(2)}\cdot P_X)P_X\right] & = & (h\cdot (Q_XP_X)P_X) \ = \ (h\cdot P_X)P_X,  \\
        \left[(h_{(1)}\cdot P_X)P_X\right]\left[(h_{(2)}\cdot Q_X)P_X\right] & = & (h\cdot (P_XQ_X)P_X) \ = \ (h\cdot P_X)P_X.  \\
    \end{array}\right.
  $$
It follows from Theorem \ref{teorema.idempotent.convolution.coradical} that Equation (\ref{equation.hPXPX.iqual.HQXPX}) holds for each $h \in H$.
    
\end{proof}

\color{black}    \begin{teorema}\label{teorema.APX.simeq.APXG}
Let $H$ be a pointed Hopf algebra with finite group $G$ of grouplike elements.  If $A = \langle \eps_h \mid h \in H\rangle \subseteq \Hpar$, $X\in \Px_1(G)$, and $G_X=\{t\in G \mid tX=X\}$, then the surjective algebra map 
        $$ \varphi_{X}: A  \larr AP_X, \ \textrm{ defined by } \ \varphi_X(\eps_h) = (h\cdot P_X)P_X $$
        induces an isomorphism of algebras $\varphi_{X}\mid_{AP_{G_X}}:AP_{G_X} \to AP_X$.
    \end{teorema}    
    \begin{proof}
        In order to prove that 
        $$ \varphi_{X}\mid_{AP_{G_X}}:AP_{G_X} \to AP_X,$$
        given by 
        $$\varphi_{X}(\eps_hP_{G_X}) = (h\cdot P_X)P_X\varphi_{X}(P_{G_X}) = (h\cdot P_X)P_XP_X = (h\cdot P_X)P_X, $$
        for all $h \in H$, is an isomorphism, it is enough to show that it has a left inverse, since it is already surjective. Applying Lemma \ref{lema.psiX.algebra.ker} to $G_X$ in place of $X$, we have the following algebra map 
        $$ \psi_{X}: A  \larr AP_{G_X}, \ \  \ \psi_X(\eps_h) = \dfrac{1}{|G_X|}\sum_{g^{-1}\in X}\eps_{hg}P_{G_X}, $$
        which also satisfies 
        $ \displaystyle \bigoplus_{{Y \in \Px_1(G)}\atop{Y \neq X}}AP_{Y} \subseteq \ker\psi_X$.  
        Hence, we can consider the restriction \\ $\psi_{X}|_{AP_X}: AP_X  \to AP_{G_X}$. By Lemma \ref{lema.varphiX(h).formula}, if $X=\{g_1, \ldots ,g_n\}$, then $\varphi_X(\eps_h) = \eps_{h_{(1)}g_1}  \widedots  \eps_{h_{(n)}g_n}\eps_{h_{(n+1)}}P_X$. Therefore 
        \begin{eqnarray*}
            & \, & (\psi_X\circ\varphi_X)(\eps_hP_{G_X}) =  \displaystyle \psi_X\left(\eps_{h_{(1)}g_1}  \widedots  \eps_{h_{(n)}g_n}\eps_{h_{(n+1)}}P_X\right) \\
             & = & \displaystyle \dfrac{1}{|G_X|^{n+1}}\sum_{t_1^{-1}, \ldots ,t_{n+1}^{-1} \in X}\eps_{h_{(1)}g_1t_1}  \widedots  \eps_{h_{(n)}g_nt_n}\eps_{h_{(n+1)}t_{n+1}}P_{G_X} \\
             & = & \displaystyle \dfrac{1}{|G_X|^{n+1}}\sum_{t_1^{-1}, \ldots ,t_{n+1}^{-1} \in X}ht_{n+1}\cdot(\eps_{t_{n+1}^{-1}g_1t_1}  \widedots  \eps_{t_{n+1}^{-1}g_nt_n})P_{G_X}  \\
             & \stackrel{P_{G_X} = \, \Gamma_{G_X}}{=} & \displaystyle \dfrac{1}{|G_X|^{n+1}}\sum_{t_1^{-1}, \ldots ,t_{n+1}^{-1} \in X}ht_{n+1}\cdot\left(\eps_{t_{n+1}^{-1}g_1t_1}  \widedots  \eps_{t_{n+1}^{-1}g_nt_n}P_{G_X}\right) \\
              & = & \displaystyle \dfrac{1}{|G_X|^{n+1}}\sum_{t_{n+1}^{-1}\in X}ht_{n+1}\cdot\left(\sum_{t_1^{-1}, \ldots ,t_{n}^{-1} \in X}\eps_{t_{n+1}^{-1}g_1t_1}  \widedots  \eps_{t_{n+1}^{-1}g_nt_n}P_{G_X}\right).
        \end{eqnarray*}
                Now fix $t_1, t_2, \ldots, t_{n+1} \in X$. We have that 
                $\eps_{t_{n+1}^{-1}g_1t_1}  \widedots  \eps_{t_{n+1}^{-1}g_nt_n}P_{G_X} \neq 0$ if and only if $t_{n+1}^{-1}g_1t_1,  \ldots , t_{n+1}^{-1}g_nt_n \in G_{X}$, which implies that $$t^{-1}_{n+1}g_i = (t^{-1}_{n+1}g_it_i)t^{-1}_i \in X$$
                for all $i=1, \ldots ,n$, and this says that $t_{n+1} \in G_X$ because $X=\{g_1, \ldots ,g_n\}$. Therefore the term $\eps_{t_{n+1}^{-1}g_1t_1}  \widedots  \eps_{t_{n+1}^{-1}g_nt_n}P_{G_X}$ is nonzero only if $t_{n+1} \in G_X$, and 
        \[\begin{array}{ccl}
            (\psi_X\circ\varphi_X)(\eps_hP_{G_X})  & = & \displaystyle \dfrac{1}{|G_X|^{n+1}}\sum_{t_{n+1}^{-1}\in G_X}ht_{n+1}\cdot\left(\sum_{t_1^{-1}, \ldots ,t_{n}^{-1} \in X}\eps_{t_{n+1}^{-1}g_1t_1}  \widedots  \eps_{t_{n+1}^{-1}g_nt_n}P_{G_X}\right) \\
             & = & \displaystyle \dfrac{1}{|G_X|^{n+1}}\sum_{t_{n+1}^{-1}\in G_X}ht_{n+1}\cdot\left(\sum_{t_1^{-1}, \ldots ,t_{n}^{-1} \in X}\eps_{t_{n+1}^{-1}g_1t_1}  \widedots  \eps_{t_{n+1}^{-1}g_nt_n}\eps_{t_{n+1}^{-1}}P_{G_X}\right) \\
             & = & \displaystyle \dfrac{1}{|G_X|^{n+1}}\sum_{t_{n+1}^{-1}\in G_X}ht_{n+1}\cdot\left( t_{n+1}^{-1}\cdot \left(\sum_{t_1^{-1}, \ldots ,t_{n}^{-1} \in X}\eps_{g_1t_1}  \widedots  \eps_{g_nt_n}P_{G_X}\right)\right). \\
        \end{array}
        \]
        For each sequence $t_1, \ldots, t_n \in X$ we have that 
        $$ \eps_{g_1t_1}  \widedots  \eps_{g_nt_n}P_{G_X} = \left\{\begin{array}{cl}
            P_{G_X}, & \textrm{ if } g_1t_1,  \ldots , g_nt_n \in G_X , \\
            0, & \textrm{ otherwise. } 
        \end{array}\right. $$
        Hence we may rewrite the summation as 
        $$ \sum_{t_1^{-1}, \ldots ,t_{n}^{-1} \in X}\eps_{g_1t_1}  \widedots  \eps_{g_nt_n}P_{G_X} = \sum_{k_1,  \ldots , k_n \in G_{X}}\eps_{k_1}  \widedots  \eps_{k_n}P_{G_X} = |G_X|^nP_{G_X}, $$
        and it follows that
         \begin{eqnarray*}
           & \, & (\psi_X\circ\varphi_X)(\eps_hP_{G_X})   =   \displaystyle \dfrac{1}{|G_X|^{n+1}}\sum_{t_{n+1}^{-1}\in G_X}ht_{n+1}\cdot\left( t_{n+1}^{-1}\cdot \left(|G_X|^{n}P_{G_X}\right)\right) \\
            & = &  \displaystyle \dfrac{1}{|G_X|}\sum_{t_{n+1}^{-1}\in G_X}h\cdot \left(t_{n+1}\cdot\left( t_{n+1}^{-1}\cdot P_{G_X}\right)\right) 
             =   \displaystyle \dfrac{1}{|G_X|}\sum_{t_{n+1}^{-1}\in G_X}h\cdot \left(\eps_{t_{n+1}}P_{G_X}\right) \\
             & = &  h\cdot P_{G_X} \ = \ \eps_hP_{G_X}. \\
        \end{eqnarray*}
        
       We have just proved that the map $\psi_{X}|_{AP_X}: AP_X \to AP_{G_X}$ is a left inverse for $\varphi_{X}|_{AP_{G_X}}: AP_{G_X}  \to AP_{X}$, which is surjective, hence the latter is an isomorphism. 
    \end{proof}

The main result that summarizes the algebra structure of $\apar$ follows directly from Proposition \ref{proposicao.APX.iso.APgX}
and 
Theorem \ref{teorema.APX.simeq.APXG}. 

\begin{teorema} \label{teorema.decomposition.A.stabilizers}
    Let $H$ be a pointed Hopf algebra with finite group $G=G(H)$ of grouplike elements, and let $A = \langle \eps_h \mid h \in H \rangle \subseteq \Hpar$.  If $X_1, X_2,  \ldots , X_n$ constitute a complete set of representatives in $\mathcal{P}_1(G)$ for the equivalence relation defined in \ref{equivalence.relation}, then there are algebra isomorphisms
    $$ A = \bigoplus_{k=1}^{n}(A \, \Gamma_{X_k}) \simeq  \bigoplus_{k=1}^{n}(A \, P_{G_{X_k}})^{n_k}. $$
\end{teorema}

\subsection{Partial Action on $\apar$}

Consider $A = \apar$ as before. Under the conditions of Theorem \ref{teorema.APX.simeq.APXG}, the map $\varphi_X\mid_{AP_{G_X}}: AP_{G_X} \to AP_X$ defined by $\varphi_X(\eps_h P_{G_X}) = (h \cdot P_X) P_X$ for all $h \in H$ is an isomorphism of algebras. 
In fact, this map preserves another structure as well. By Lemma \ref{lema.hGammaA.h1GammaA}, each component $A P_{G_X}$ associated to an isotropy subgroup is invariant under the $H$-action on $A$, since $h \cdot P_{G_X} = h \cdot \Gamma_{G_X} = \epsilon_h \Gamma_{G_X} = \epsilon_h P_{G_X}$. 
On the other hand, by Proposition \ref{prop:induced.partial.action}, each component $A P_X$ of $A$ has a canonical structure of partial $H$-module algebra: $h \rhd aP_X := (h\cdot aP_X)P_X$. The map $\varphi_X$ turns out to be an isomorphism of partial $H$-module algebras with respect to those structures: given $a = \eps_{k_1}  \widedots  \eps_{k_n} \in A$,  
\begin{eqnarray*}
    & \, & h \rhd \varphi_X(aP_{G_X})  = \bigl[h\cdot\left(\varphi_X(\eps_{k_1}  \widedots  \eps_{k_n}P_{G_X})\right)\bigr]P_{X} \\
     & = & \bigl[h\cdot((k_1\cdot P_X)  \widedots  (k_n\cdot P_X)P_X)\bigr]P_X 
      =  (h_{(1)}k_1\cdot P_X)  \widedots  (h_{(n)}k_n\cdot P_X)(h_{(n+1)}\cdot P_X)P_X \\
     & = & \varphi_X(\eps_{h_{(1)}k_1}  \widedots  \eps_{h_{(n)}k_n}\eps_{h_{(n+1)}}P_{G_X}) 
      =  \varphi_X((h\cdot a)P_{G_X}) \\
     & = & \varphi_X(h\cdot (aP_{G_X})).
\end{eqnarray*} 
Therefore $\varphi_X$ preserves the $H$-action: 
\begin{equation}\label{equation.varphi.morfismo.de.Hmodulos}
     h \rhd \varphi_X(aP_{G_X}) = h\cdot\bigl(\varphi_X(aP_{G_X})\bigr) P_X =   \varphi_X\bigl((h\cdot a)P_{G_X}\bigr) = \varphi_X\bigl(h\cdot(aP_{G_X})\bigr)
 \end{equation}

 \begin{lema} \label{cor:AP_X=AP_{gX}}
Let $H$ be a pointed Hopf algebra with finite group $G$ of grouplike elements. Given $X \in \mathcal{P}_1(G)$ and $g \in G$ such that $g^{-1} \in X$, then the conjugation by $g$ in $H$ induces the following algebra isomorphism
\[A P_{G_X} \simeq A P_{g G_X g^{-1}} = A P_{G_{gX}}.\]
\end{lema}
\begin{proof} Given $g \in G$, the map 
 $$ \begin{array}{cccc}
     c_g: &  H & \larr & H \\
         &  h & \lmap & ghg^{-1} \\
 \end{array} $$
 is a Hopf algebra automorphism, hence it induces the algebra isomorphism  
 $$  \begin{array}{cccc}
     c'_g: & \Hpar & \larr & \Hpar \\
      & [h] & \lmap & [ghg^{-1}]
 \end{array} $$
which satisfies $c'_g(\eps_h) = \eps_{ghg^{-1}}$; thus $c'_g$ restricts to an algebra automorphism of $A$ and, from the definition of $P_X$, it follows that $c'_g(P_X) = P_{gXg^{-1}}$. Moreover, given that $g G_X g^{-1} = G_{gX}$, 
\[
c'_g(A P_{G_X} ) = c'_g(A) c'_g(P_{G_X})  = A P_{gG_Xg^{-1}} = A P_{G_{gX}}. 
\] 
\end{proof}

Proposition \ref{proposicao.APX.iso.APgX} provides an isomorphism between $AP_X$ and $AP_{gX}$, whenever $g^{-1} \in X \subseteq \Px(G)$. Lemma \ref{cor:AP_X=AP_{gX}} tells us that we can translate this same isomorphism in the layer of the stabilizers of $X$ and $gX$, that is, 
$AP_{G_X} \simeq AP_{G_{gX}}$. Finally, Theorem  \ref{teorema.APX.simeq.APXG} establishes the isomorphisms  
$AP_X \simeq AP_{G_X}$ and $ AP_{gX}\simeq AP_{G_{gX}}$. 
The next result relates all those isomorphisms. 

\begin{proposicao}\label{proposicao.gvarphi=varphicg}
    Let $H$ be a pointed Hopf algebra with finite group $G$ of grouplike elements. Given $X \in \Px_1(G)$ and $g \in G$ such that $g^{-1} \in X$, the following diagram of algebra isomorphisms is commutative:
    $$\xymatrix{
    AP_{G_X} \ar[rr]^{c'_g} \ar[dd]_{\varphi_X} & & AP_{G_{gX}} \ar[dd]^{\varphi_{gX}} \\
    & & \\
    AP_X \ar[rr]_{g\cdot -} & & AP_{gX}
    }$$
\end{proposicao}
\begin{proof}
    Since all the morphisms in the diagram are algebra morphisms, it is enough to evaluate the compositions on elements of the form $\epsilon_h P_{G_X}$.
\begin{eqnarray*}
    & \, & g \cdot \varphi_X(\epsilon_h P_{G_X})  = g \cdot ((h \cdot P_X) P_X) \\
     & = & (g \cdot (h \cdot P_X))(g \cdot P_X) 
      =  (gh \cdot P_X)(g \cdot 1_A)(g \cdot P_X) \\
     & = & (gh \cdot P_X)(g \cdot P_X) 
\stackrel{\eqref{equation.gPx.equal.0.or.Pgx}}{=}  (gh \cdot P_X) P_{gX}  
 \stackrel{\eqref{equation.g-1.g.Px.equals.Px}}{=}  (gh \cdot (g^{-1} \cdot (g \cdot P_X))) P_{gX} \\
&  \stackrel{\eqref{eq.ex.1.1-PR4}}{=} &  (ghg^{-1} \cdot P_{gX}) P_{gX}  
       =   \varphi_{gX}(\epsilon_{ghg^{-1}} P_{G_{gX}}) \\
     &  = &  (\varphi_{gX} \circ c'_g)(\epsilon_{ghg^{-1}} P_{G_{gX}}),
\end{eqnarray*}
which proves the proposition.
\end{proof}

The proposition above gives us a glance of how to deal with the $H$-action on the algebra $A$. 
By Theorem \ref{teorema.A.GammaA.direct.somma} there exist $X_1,  \ldots , X_t$ such that $ A=\bigoplus_{i=1}^t A\Gamma_{X_i} $, where each $A\Gamma_{X_i}$ is a partial $H$-module algebra.
So in order to describe the partial $H$-action on $A$, it will be enough to describe the action on $A\Gamma_X$ for $X \in \Px_1(G)$. 

For $X \in \Px_1(G)$, there exist $g_1,  \ldots , g_n \in G$, such that $\Gamma_X = P_{g_1X} +   \widedots   + P_{g_nX}$, and we have a sequence of algebra isomorphisms

$$ \xymatrix{
\displaystyle A\Gamma_X = \bigoplus_{l=1}^n AP_{g_lX} \ar[rr]^-{\Pi_{l=1}^n(\varphi_{g_lX})^{-1}} & & \displaystyle \bigoplus_{l=1}^{n}AP_{G_{g_lX}} \ar[rr]^{\Pi_{l=1}^n(c'_{g_l})^{-1}} & & (AP_{G_X})^{n} 
}
$$
Now we are going to construct a partial $H$-action on $(AP_{G_X})^{n}$ that turns  the composition above into  an isomorphism of partial $H$-module algebras. First, we explicitly write its inverse: 
$$ \begin{array}{cccl}
    \Phi_X: & (AP_{G_X})^{n} & \larr & A\Gamma_X \\
      & (a_1P_{G_X},  \ldots , a_nP_{G_X}) & \lmap & \displaystyle \sum_{j=1}^{n} (\varphi_{g_jX}\circ c'_{g_j})(a_jP_{G_X})
\end{array} $$
which, written in terms of generators, yields the formula 
$$ \Phi_X(\eps_{k^{1}}P_{G_X},  \ldots , \eps_{k^{n}}P_{G_X}) = \sum_{j=1}^n((g_jk^{j}g_j^{-1})\cdot P_{g_jX})P_{g_jX}. $$

Now, for $\Phi_X:(AP_{G_X})^{n} \to A\Gamma_X$ to become a partial $H$-module algebra isomorphism, we endow  $(AP_{G_X})^{n}$ with an $H$-action by the equality 
\begin{equation}
h\cdot (a_1P_{G_X},  \ldots , a_nP_{G_X})  =  \Phi_X^{-1}( h\cdot \Phi_X(a_1P_{G_X},  \ldots , a_nP_{G_X}))     
\end{equation}
or, equivalently, 
\begin{equation} \label{eq:def.H.action.APGX}
\Phi_X(h\cdot (a_1P_{G_X},  \ldots , a_nP_{G_X}))  =   h\cdot \Phi_X(a_1P_{G_X},  \ldots , a_nP_{G_X}).     
\end{equation}

We will develop the right hand side of \eqref{eq:def.H.action.APGX} in order to compute explicitly this $H$-action on an element of $(AP_{G_X})^{n}$.
\begin{eqnarray*}
    & \, &  h\cdot \Phi_X(a_1P_{G_X},  \ldots , a_nP_{G_X})  
    =  \displaystyle \sum_{j=1}^n h\cdot\bigl((\varphi_{g_jX}\circ c'_{g_j})(a_jP_{G_X})\bigr) \\ 
    &  \stackrel{\ref{proposicao.gvarphi=varphicg}}{=}&  \displaystyle \sum_{j=1}^n h\cdot\bigl[g_j\cdot(\underbrace{\varphi_{X}(a_jP_{G_X})}_{\in AP_X})\bigr] 
     =  \displaystyle \sum_{j=1}^n h\cdot\bigl[g_j\cdot g_j^{-1} \cdot g_j \cdot(\varphi_{X}(a_jP_{G_X}))\bigr] \\
     & = &  \displaystyle \sum_{j=1}^n hg_j\cdot g_j^{-1} \cdot g_j \cdot(\varphi_{X}(a_jP_{G_X})) 
     =  \displaystyle \sum_{j=1}^n hg_j\cdot(\varphi_{X}(a_jP_{G_X})). \\
\end{eqnarray*}

Since $\Gamma_X$ is the unit of $A \Gamma_X$, 
\begin{eqnarray*}
    & \, & \displaystyle \sum_{j=1}^n hg_j\cdot(\varphi_{X}(a_jP_{G_X})) =  \displaystyle \sum_{j=1}^n hg_j\cdot(\varphi_{X}(a_jP_{G_X}))\Gamma_X \\
     & = & \displaystyle \sum_{i=1}^n\sum_{j=1}^nhg_j\cdot(\varphi_{X}(a_jP_{G_X}))P_{g_iX} 
      =  \displaystyle \sum_{i=1}^n\sum_{j=1}^ng_ig_i^{-1}hg_j\cdot(\varphi_{X}(a_jP_{G_X}))(g_i\cdot P_{X}) \\
     & = & \displaystyle \sum_{i=1}^n\sum_{j=1}^ng_i\cdot\bigl[(g_i^{-1}hg_j\cdot(\varphi_{X}(a_jP_{G_X}))P_{X}\bigr] 
     \stackrel{(\ref{equation.varphi.morfismo.de.Hmodulos})}{=} \displaystyle \sum_{i=1}^n\sum_{j=1}^ng_i\cdot\bigl[\varphi_{X}((g_i^{-1}hg_j\cdot a_j)P_{G_X})\bigr] \\
     & \stackrel{\ref{proposicao.gvarphi=varphicg}}{=} & \displaystyle \sum_{i=1}^n\sum_{j=1}^n(\varphi_{g_iX}\circ c'_{g_i})((g_i^{-1}hg_j\cdot a_j)P_{G_X}) 
      =  \displaystyle \sum_{i=1}^n(\varphi_{g_iX}\circ c'_{g_i})\biggl(\sum_{j=1}^n(g_i^{-1}hg_j\cdot a_j)P_{G_X}\biggr) \\
     & = & \displaystyle \Phi_X\biggl(\sum_{j=1}^n(g_1^{-1}hg_j\cdot a_j)P_{G_X},  \ldots , \sum_{j=1}^n(g_n^{-1}hg_j\cdot a_j)P_{G_X}\biggr).
\end{eqnarray*} 
It follows from \eqref{eq:def.H.action.APGX} and the previous computations that 
\[
\Phi_X(h\cdot (a_1P_{G_X},  \ldots , a_nP_{G_X})) 
=  \displaystyle \Phi_X\biggl(\sum_{j=1}^n(g_1^{-1}hg_j\cdot a_j)P_{G_X},  \ldots , \sum_{j=1}^n(g_n^{-1}hg_j\cdot a_j)P_{G_X}\biggr).
\]

Therefore, the map $\Phi_X = \prod_{i=1}^{n}(\varphi_{g_iX}\circ c'_{g_i}):(AP_{G_X})^{n} \to A\Gamma_X $ is a partial $H$-module algebra isomorphism if the $H$-action on $(AP_{G_X})^{n}$ is defined in the following manner:
\begin{equation} \label{equation.H.action.on.APGn}
     h \cdot (a_1P_{G_X},  \ldots , a_nP_{G_X})  =  \displaystyle \biggl(\sum_{j=1}^n(g_1^{-1}hg_j\cdot a_j)P_{G_X},  \ldots , \sum_{j=1}^n(g_n^{-1}hg_j\cdot a_j)P_{G_X}\biggr) 
\end{equation}
or, in matrix form, 
\begin{equation} \label{eq:MatrixForm:H-action.on.(AP_{G_X})^{n}}
     h \cdot (a_1P_{G_X},  \ldots , a_nP_{G_X})  =  
     \left[\begin{array}{cccc}
         g_1^{-1}hg_1\cdot a_1 & g_1^{-1}hg_2\cdot a_2 &  \widedots  & g_1^{-1}hg_n\cdot a_n \\
         g_2^{-1}hg_1\cdot a_1 & g_2^{-1}hg_2\cdot a_2 &  \widedots  & g_2^{-1}hg_n\cdot a_n \\
         \vdots & \vdots & \ddots & \vdots \\
         g_n^{-1}hg_1\cdot a_1 & g_n^{-1}hg_2\cdot a_2 &  \widedots  & g_n^{-1}hg_n\cdot a_n \\ 
     \end{array}\right]
     \left[\begin{array}{c}
          P_{G_X} \\
          P_{G_X} \\
          \vdots \\
          P_{G_X}
     \end{array}\right] 
\end{equation}

This proves the following.
    
{\begin{proposicao}\label{proposicao.APGn.iso.AGammaX}
    The partial $H$-action defined at \eqref{equation.H.action.on.APGn} on the algebra $(AP_{G_X})^{n}$ makes the isomorphism $\Phi_X:(AP_{G_X})^{n} \to A\Gamma_X $ a morphism of partial $H$-module algebras.
\end{proposicao}}

\black 

\subsection{On the structure of $\Hpar$}
In \cite{piccione}, the authors proved that $\K_{\mathrm{par}}G$ is isomorphic to a direct sum of matrix algebras of the form $\mathrm{Mat}_{m}(\K L)$, where $L$ ranges over certain subgroups of $G$. Moreover, they explicitly determined the multiplicities of each matrix algebra appearing in the decomposition. The formula describing these multiplicities was subsequently corrected in \cite{dokuchaev2004isomorphisms}. In this section, we generalize this result to $\Hpar$ when $H$ is a pointed Hopf algebra.

We begin by showing that 
\[
\Hpar = \bigoplus_{k=1}^{n}(\Hpar\Gamma_{X_k}) 
\]
as a direct sum of unital ideals. 

Recall from Theorem \ref{teorema.Hpar.igual.AH} that $\Hpar \simeq \underline{A\# H}$ as algebras and that, by Proposition \ref{lema.hGammaA.h1GammaA} and Theorem \ref{teorema.A.GammaA.direct.somma}, the element 
$$\Gamma_X = \dfrac{1}{|G_X|}\sum_{g^{-1}\in X} g\cdot P_X = \dfrac{1}{|G_X|}\sum_{g^{-1}\in X} P_{gX} = \sum_{Y \sim X} P_Y,  $$
is a central idempotent of $A$ for any $X \in \Px_1(G)$. Moreover, by Theorem \ref{teorema.A.GammaA.direct.somma},  
$$ h\cdot \Gamma_X = (h\cdot 1\um)\Gamma_X = \eps_h\Gamma_X, \ \forall h \in H $$
and there exist $X_1, X_2,  \ldots , X_n \in \Px_1(G)$, such that 
$$ A = \bigoplus_{k=1}^n (A\Gamma_{X_k}). $$
Hence for all $a\# h \in \underline{A\# H}$, 
$$ a\# h = \sum_{k=1}^n(a\Gamma_{X_k}\# h),$$
and since $\Gamma_X$ is central in $A$,
$$ a\# h = \sum_{k=1}^n(a\Gamma_{X_k}\# h) = \sum_{k=1}^n(\Gamma_{X_k}a\# h) = \sum_{k=1}^n(\Gamma_{X_k}\# 1_H)(a\# h). $$
Next, consider $a \in A$, $h \in H$ and $X\in \Px_1(G)$. We have: 
$$ 
     (a\# h)(\Gamma_X \# 1_H)  =  a(h_{(1)}\cdot \Gamma_X)\# h_{(2)}  =  a\eps_{h_{(1)}}\Gamma_X\# h_{(2)}  =  \Gamma_Xa\eps_{h_{(1)}}\# h_{(2)}  =  (\Gamma_X\# 1_H)(a\# h),
 $$

 which shows that $\Gamma_X\# 1_H$ is central in $\underline{A\# H}$. Furthermore, since the map   
 $$ \begin{array}{ccc}
     \underline{A\# H} & \larr & \Hpar  \\
     a\# h & \lmap & a[h]
 \end{array} $$
 is an isomorphism of algebras, $\Gamma_X$ is a central idempotent of $\Hpar$. 
 We have just proved the following.

\begin{proposicao}\label{proposicao.Gamma.central.idempotent}
     Let $H$ be a pointed Hopf algebra with finite group $G$ of grouplike elements.
    For any $X \in \mathcal{P}_1(G)$, the element
    $$ \Gamma_X = \dfrac{1}{|G_X|}\sum_{g\in G} g\cdot P_X = \dfrac{1}{|G_X|}\sum_{g^{-1}\in X}P_{gX}, $$
    is a central idempotent element in $\Hpar$. Furthermore, for any $h\in H$, we have:
    \begin{equation} \label{eq:h.GammaX}
        h\cdot \Gamma_X = \epsilon_h\Gamma_X. 
    \end{equation}
\end{proposicao}
If $X_1, X_2,  \ldots , X_n$ constitute a complete set of representatives in $\mathcal{P}_1(G)$ for the equivalence relation defined in \ref{equivalence.relation}, and from Theorem \ref{teorema.A.GammaA.direct.somma} it follows that
 \[
 1_{\apar} = 1_{\Hpar} = \sum_{k=1}^n\Bigl(\sum_{X \sim X_k}P_X\Bigr) = \sum_{k=1}^n\Gamma_{X_k}.
 \]
Therefore, we obtain a decomposition of $\Hpar$ similar to the previous one for $\apar$ (Theorem \ref{teorema.A.GammaA.direct.somma}).

\begin{proposicao}\label{proposicao.pointed.Hpar.direct.sum}
      Let $H$ be a pointed Hopf algebra finite group $G$ of grouplike elements.
      If $X_1, X_2,  \ldots , X_n$ constitute a complete set of representatives in $\mathcal{P}_1(G)$ for the equivalence relation defined in \ref{equivalence.relation}, then
    $$ \Hpar = \bigoplus_{k=1}^{n}(\Hpar\Gamma_{X_k}), $$
    where $\Hpar\Gamma_{X_k}$ is an ideal of $\Hpar$ for all $k=1,2,  \ldots , n$.
\end{proposicao}

In the following, we shall sometimes denote the ideal of $A$ generated by $P_X$ as either $P_XA$ or $AP_X$. Similarly, the ideal of $\Hpar$ generated by $\Gamma_X$ will sometimes be represented as $\Hpar\Gamma_X$ or $\Gamma_X\Hpar$, depending on convenience.

Consider $X \in \Px_1(G)$ and let $\Gamma_X = P_{g_1X} +  \widedots  + P_{g_nX}$. 
In Proposition \ref{proposicao.APGn.iso.AGammaX} we have seen that the partial $H$-module algebras $A\Gamma_X$ and $(AP_{G_X})^{n}$ are isomorphic, where the $H$-action on $(AP_{G_X})^{n}$ is the one defined in \eqref{equation.H.action.on.APGn}. 

As a consequence we have algebra isomorphisms
\begin{equation}\label{equation.HparGammaX.APGnH}
    \Hpar\Gamma_X \simeq \underline{A\Gamma_X\# H} \simeq \underline{ (AP_{G_X})^{n}\# H},
\end{equation}
hence in order to understand  the algebra structure of $\Hpar\Gamma_X$, it is enough to study the algebra structure of $\underline{(AP_{G_X})^{n}\# H}$. Notice that $P_{G_X}$ is the identity of the algebra $AP_{G_X}$, so we are going to use the notation $1_{G_X}$ to denote $P_{G_X}$, at least when we are talking about the algebra $AP_{G_X}$.

In this subsection we will construct an isomorphism of algebras 
$$\Theta:  \underline{(AP_{G_X})^n\# H}  \to  Mat_{n}\left(\underline{AP_{G_X}\# H}\right).$$ This isomorphism will be defined via a \textit{covariant pair} $(f,\pi)$.

We recall from \cite{alves2015partial} that a covariant pair $(f,\pi)$ associated to a triple $(H,B,C)$, where $H$ is a bialgebra, $B$ is a  partial $H$-module algebra and  $C$ is an algebra, consists of an algebra morphism $f: B \to C$ and a partial representation $\pi : H\to  C$ such that 
\begin{itemize}
    \item[(CP1)] $f(h\cdot b) = \pi(h_{(1)})f(b)\pi(S(h_{(2)}))$, for all $b \in B$ and $h \in H$;
    \item[(CP2)] $f(b)\pi(S(h_{(1)}))\pi(h_{(2)}) = \pi(S(h_{(1)}))\pi(h_{(2)})f(b)$, for all $b \in B$ and $h \in H$.
\end{itemize}
By  \cite[Thm 3.9]{alves2015partial}, if $(f,\pi)$ is a covariant pair then the map $
\Theta : \underline{ B \# H} \to C$ defined by $\Theta(\underline{b \# h}) = f(b)\pi(h)$ is an algebra morphism.

Let us begin by showing that the linear map 
$$ \begin{array}{cccl}
    \pi: & H & \larr & Mat_{n}(\underline{AP_{G_X}\# H}) \\
     & h & \lmap & \displaystyle \sum_{1 \leq i,j \leq n} (\umgx\#g^{-1}_ihg_j)E_{i,j}
\end{array} $$
is a partial representation, where $\{E_{i,j} \mid 1\leq i,j \leq n\}$ is the canonical basis
 of the free left $AP_{G_X}\# H$-module $Mat_{n}(\underline{AP_{G_X}\# H})$.
First let us show that $\pi$ preserves the identity. The image of $1_H$ is  
    $$ \pi(1_H) = \sum_{1\leq i,j \leq n}(\umgx\# g^{-1}_ig_j)E_{i,j} $$
    and $\umgx\# g_i^{-1}g_j \neq 0$ if and only if $g_i^{-1}g_j \in G_X$; since $\Gamma_X = \sum_{i=1}^nP_{g_iX}$, we conclude that $g^{-1}_ig_j \in G_X$ only when $i=j$. Hence 
    $$ \pi(1_H) = \sum_{1\leq i,j \leq n}(\umgx\# g^{-1}_ig_j)E_{i,j} = \sum_{i=1}^n(\umgx\# 1_H)E_{i,i}, $$
    which is the identity of the algebra  $Mat_{n}(\underline{AP_{G_X}\# H})$. Now, notice that for all $h \in H$ we have 
    $$ \begin{array}{ccl}
        \pi(h_{(1)})\pi(S(h_{(2)})) & = & \displaystyle \sum_{1\leq i,j \leq n}\big(\sum_{t=1}^n (\umgx\# g^{-1}_ih_{(1)}g_{t})(\umgx\# g^{-1}_tS(h_{(2)})g_{j})\big)E_{i,j} \\
        & = & \displaystyle \sum_{1\leq i,j\leq n} \left(\sum_{t=1}^n(g_{i}^{-1}hg_t)\cdot \umgx \# g_i^{-1}g_j\right) E_{i,j} \\
        & = & \displaystyle \sum_{i = 1}^{n} \left(\left(\sum_{t=1}^n(g_{i}^{-1}hg_t)\cdot\umgx\right) \# 1_H\right) E_{i,i} \\
    \end{array} $$
    hence
    $$  \pi(h_{(1)})\pi(S(h_{(2)}))\pi(k) = \sum_{1\leq i,j\leq n}\left(\Biggl(\sum_{t=1}^n(g_{i}^{-1}hg_t)\cdot\umgx\Biggr) \# g_i^{-1}kg_j\right)E_{i,j}. $$
    On the other hand, 

    $$ \begin{array}{ccl}
         \pi(h_{(1)})\pi(S(h_{(2)})k) & = & \displaystyle \sum_{1\leq i,j \leq n}\left(\sum_{l=1}^n(\umgx\# g_i^{-1}h_{(1)}g_l)(\umgx\#g_{l}^{-1}S(h_{(2)})kg_j)\right)E_{i,j}  \\
         & = & \displaystyle \sum_{1\leq i,j \leq n}\left(\Biggl(\sum_{l=1}^n(g_i^{-1}hg_l)\cdot\umgx\Biggr)\#g_i^{-1}kg_j\right)E_{i,j},  \\
    \end{array}
     $$
     hence $\pi(h_{(1)})\pi(S(h_{(2)}))\pi(k) = \pi(h_{(1)})\pi(S(h_{(2)})k)$, for all $h,k \in H$. A similar computation shows that $\pi(kh_{(1)})\pi(S(h_{(2)})) = \pi(k)\pi(h_{(1)})\pi(S(h_{(2)}))$. Therefore, the linear map  
     $$ \begin{array}{cccl}
    \pi: & H & \larr & Mat_{n}(\underline{AP_{G_X}\# H}) \\
     & h & \lmap & \displaystyle \sum_{1 \leq i,j \leq n} (\umgx\#g^{-1}_ihg_j)E_{i,j}
\end{array} $$
is a partial representation of $H$. 

Now we will prove that $(f,\pi)$ is a covariant pair, where the algebra map $f$ is given by 
$$ \begin{array}{cccl}
    f: & (AP_{G_X})^{n} & \larr & Mat_{n}(\underline{AP_{G_X}\# H}) \\
     & (a_1, \ldots ,a_n) & \lmap & \displaystyle \sum_{i=1}^n (a_i\# 1_H)E_{i,i}.
\end{array}  $$

\begin{lema}\label{lema.f.pi.convariant.pair}
    The pair of maps $(f,\pi)$ is a covariant pair.
\end{lema}
        
\begin{proof}

         (CP1) Given $h \in H$ and $(a_1, \ldots ,a_n) \in (AP_{G_X})^n$, we have that
          \begin{eqnarray*}
            & \, &  \pi(h_{(1)})f(a_1, \ldots ,a_n)\pi(S(h_{(2)}))  =  \displaystyle \sum_{1\leq i,j \leq n}\left(\sum_{t=1}^n (\umgx\# g_{i}^{-1}h_{(1)}g_{t})(a_t\# g^{-1}_{t}S(h_{(2)})g_j) \right)E_{i,j} \\
              & = & \displaystyle \sum_{1\leq i,j \leq n}\left(\sum_{t=1}^n (g_{i}^{-1}h_{(1)}g_{t}\cdot a_t)\# g_{i}^{-1}h_{(2)}g_{t}g^{-1}_{t}S(h_{(3)})g_j \right)E_{i,j} 
               =  \displaystyle \sum_{1\leq i,j \leq n}\left(\sum_{t=1}^n (g_{i}^{-1}hg_{t}\cdot a_t)\# g_{i}^{-1}g_j \right)E_{i,j} \\
              & = & \displaystyle \sum_{i= 1}^n\left(\sum_{t=1}^n (g_{i}^{-1}hg_{t}\cdot a_t)\# 1_H \right)E_{i,i} 
               =  \displaystyle \sum_{i= 1}^n\left(\Bigl(\sum_{t=1}^n g_{i}^{-1}hg_{t}\cdot a_t\Bigr)\# 1_H \right)E_{i,i} \\
              & = & \displaystyle f\Bigl(\sum_{t=1}^ng_1^{-1}hg_t\cdot a_t,  \ldots , \sum_{t=1}^ng_n^{-1}hg_t\cdot a_t\Bigr) 
               \stackrel{(\ref{equation.H.action.on.APGn})}{=}  f\bigl(h\cdot(a_1, \ldots ,a_n)\bigr).
         \end{eqnarray*} 

         (CP2) First, by Proposition \ref{proposicao.APGn.iso.AGammaX}, we have that $\Phi_X:(AP_{G_X})^n \to A\Gamma_X$ is an isomorphism of partial $H$-module algebras, so for any $h \in H$, we have that 
         $$ \Phi_X(h\cdot(\umgx, \ldots ,\umgx)) = h\cdot\Phi_X(\umgx, \ldots ,\umgx) = h\cdot \Gamma_X = \eps_h\Gamma_X, $$
         and since $A\Gamma_X = \langle\eps_h\Gamma_X \mid h \in H\rangle$, then $(AP_{G_X})^n = \langle h\cdot(\umgx, \ldots ,\umgx) \mid h \in H \rangle$. So, to prove that $$f(a_1, \ldots ,a_n)\pi(S(h_{(1)}))\pi(h_{(2)}) = \pi(S(h_{(1)}))\pi(h_{(2)})f(a_1, \ldots ,a_n),$$ for all $(a_1, \ldots ,a_n) \in (AP_{G_X})^{n}$ and $h \in H$, we only have to prove it for elements of the form $k\cdot(\umgx, \ldots ,\umgx)$, for $k \in H$. In other words, the equality that should hold is 
         $$ f(k\cdot(\umgx, \ldots ,\umgx))\pi(S(h_{(1)}))\pi(h_{(2)}) = \pi(S(h_{(1)}))\pi(h_{(2)})f(k\cdot(\umgx, \ldots ,\umgx)), \ \forall h,k \in H, $$
         and by (CP1), we have that 
         $$f(k\cdot(\umgx, \ldots ,\umgx)) = \pi(k_{(1)})\pi(\umgx, \ldots ,\umgx)\pi(S(k_{(2)})) = \pi(k_{(1)})\pi(S(k_{(2)})). $$
         Since $\pi:H \to Mat_{n}(\underline{AP_{G_X}\# H}) $, is a partial representation, it is straightforward that $$\pi(k_{(1)})\pi(S(k_{(2)}))\pi(S(h_{(1)}))\pi(h_{(2)}) = \pi(S(h_{(1)}))\pi(h_{(2)})\pi(k_{(1)})\pi(S(k_{(2)})),$$ for all $h,k \in H$, then 
         \begin{align*}
         & f(k\cdot(\umgx, \ldots ,\umgx))\pi(S(h_{(1)}))\pi(h_{(2)})  =  \pi(k_{(1)})\pi(S(k_{(2)}))\pi(S(h_{(1)}))\pi(h_{(2)}) \\
              & = \pi(S(h_{(1)}))\pi(h_{(2)})\pi(k_{(1)})\pi(S(k_{(2)}))  =  \pi(S(h_{(1)}))\pi(h_{(2)})f(k\cdot(\umgx, \ldots ,\umgx)),
         \end{align*}
         for all $h,k \in H$. 
\end{proof}

Since $(f,\pi)$ is a \textit{covariant pair}, we have that 
$$ \begin{array}{lccl}
    \Theta: & \underline{(AP_{G_X})^n\# H} & \larr & Mat_{n}(\underline{AP_{G_X}\# H})  \\
    & (a_1, \ldots ,a_n)\# h & \lmap & f(a_1, \ldots ,a_n)\pi(h) = \sum_{1\leq i,j\leq n}(a_i\# g_i^{-1}hg_j)E_{i,j}
\end{array} $$
is an algebra map.

\begin{teorema}\label{teorema.APGnH.iso.MatnAPGH}
    The algebra map
    $$ \begin{array}{lccl}
    \Theta: & \underline{(AP_{G_X})^n\# H} & \larr & Mat_{n}(\underline{AP_{G_X}\# H})  \\
    & (a_1, \ldots ,a_n)\# h & \lmap & f(a_1, \ldots ,a_n)\pi(h) = \sum_{1\leq i,j\leq n}(a_i\# g_i^{-1}hg_j)E_{i,j}
\end{array} $$
is an isomorphism of algebras.
\end{teorema}
\begin{proof}
    In order to prove that
\[
\Theta : \underline{(AP_{G_X})^n \# H} \longrightarrow \mathrm{Mat}_{n}\bigl(\underline{AP_{G_X} \# H}\bigr)
\]
is an isomorphism, we proceed by first establishing its surjectivity and subsequently constructing a left inverse.

Let $a \in AP_{G_X}$ and $h \in H$. Observe that
\[
(a \# h) E_{i,j} \in \mathrm{Mat}_{n}\bigl(\underline{AP_{G_X} \# H}\bigr)
\quad \text{and that} \quad
(a \# h) E_{i,j}
=
\bigl[(a \# 1_H) E_{i,i}\bigr]
\bigl[(\umgx \# h) E_{i,j}\bigr].
\]
We may therefore consider the element
\[
(\delta_{i,1} a, \delta_{i,2} a, \ldots, a, \ldots, \delta_{i,n} a) \# 1_H
\in \underline{(AP_{G_X})^n \# H},
\]
where $\delta_{i,j}$ denotes the Kronecker delta, then 
$$ \Theta\Bigl((\delta_{i,1} a, \delta_{i,2} a, \ldots, a, \ldots, \delta_{i,n} a) \# 1_H\Bigr) = (a \# 1_H) E_{i,i}. $$
Hence, to prove that $\Theta$ is surjective, it is enough to show that $(\umgx \# h) E_{i,j}$ lies in the image of the map $\Theta$. 
\begin{eqnarray*}
   & \, & (\umgx \# h) E_{i,j}  =  \displaystyle (\umgx\# 1_H)E_{i,i}\biggl[\sum_{1\leq t,l \leq n} (\umgx\#g_t^{-1}g_ihg_j^{-1}g_l)E_{t,l}\biggr](\umgx\# 1_H)E_{j,j} \\
    & = & \displaystyle (\umgx\# 1_H)E_{i,i}\biggl[\Theta\bigl((\umgx, \ldots ,\umgx) \# g_ihg_j^{-1}\bigr) \biggr](\umgx\# 1_H)E_{j,j} \\
     & = & \Theta\bigl((\delta_{i,1}\umgx, \ldots ,\delta_{i,n}\umgx)\# 1_H\bigr)\biggl[\Theta\bigl((\umgx, \ldots ,\umgx) \# g_ihg_j^{-1}\bigr) \biggr]\Theta\bigl((\delta_{j,1}\umgx, \ldots ,\delta_{j,n}\umgx)\# 1_H\bigr) \\
     & \stackrel{(\ref{equation.H.action.on.APGn})}{=} & \Theta\bigl((\delta_{i,1}\umgx, \ldots ,\delta_{i,n}\umgx)\# 1_H\bigr)\Theta\Bigl(\bigl((g_1^{-1}g_{i}h_{(1)}\cdot\umgx),  \ldots , (g_n^{-1}g_ih_{(1)}\cdot \umgx)\bigr) \# g_ih_{(2)}g_j^{-1}\Bigr) \\
     & = & \Theta\Bigl(\bigl(\delta_{i,1}(g_1^{-1}g_{i}h_{(1)}\cdot\umgx),  \ldots  , (h_{(1)}\cdot \umgx) ,  \ldots . , \delta_{i,n}(g_n^{-1}g_ih_{(1)}\cdot \umgx)\bigr) \# g_ih_{(2)}g_j^{-1}\Bigr) \\
     & = & \Theta\Bigl(\bigl(0,  \ldots  , \underbrace{(h_{(1)}\cdot \umgx)}_{i\textrm{-th position} } ,  \ldots  , 0\bigr) \# g_ih_{(2)}g_j^{-1}\Bigr) 
\end{eqnarray*}
To summarize, for any $(a\#h)E_{i,j} \in \mathrm{Mat}_{n}\bigl(\underline{AP_{G_X} \# H}\bigr)$, we have that
\begin{equation}\label{equation.Theta.subrejetiva}
    \Theta\Bigl(\bigl(0,  \ldots  , \underbrace{a(h_{(1)}\cdot \umgx)}_{i\textrm{-th position} } ,  \ldots  , 0\bigr) \# g_ih_{(2)}g_j^{-1}\Bigr) = (a\# h)E_{i,j},
\end{equation}
which shows the fact that $\Theta : \underline{(AP_{G_X})^n \# H} \longrightarrow \mathrm{Mat}_{n}\bigl(\underline{AP_{G_X} \# H}\bigr)$ is surjective.

Now we use Equation \eqref{equation.Theta.subrejetiva} to construct the inverse of $\Theta$. Given any $M \in \mathrm{Mat}_{n}\bigl(\underline{AP_{G_X} \# H}\bigr)$ we may write

$$ \begin{array}{ccl}
    M & = & \displaystyle\sum_{1\leq i,j \leq n} (a_{i,j}\# h_{i,j})E_{i,j}  \\
     & = & \displaystyle\sum_{1\leq i,j \leq n}\Theta\Bigl(\bigl(0,  \ldots  , \underbrace{a_{i,j}({(h_{i,j})}_{(1)}\cdot \umgx)}_{i\textrm{-th position} } ,  \ldots  , 0\bigr) \# g_i{(h_{i,j})}_{2}g_j^{-1}\Bigr) \\
     & = & \displaystyle \Theta \biggl(\sum_{1\leq i,j \leq n}\bigl(0,  \ldots  , \underbrace{a_{i,j}({(h_{i,j})}_{(1)}\cdot \umgx)}_{i\textrm{-th position} } ,  \ldots  , 0\bigr) \# g_i{(h_{i,j})}_{2}g_j^{-1}\biggr) 
\end{array} $$
Hence, we are going to define the algebra map
$$ \begin{array}{cccl}
    \Psi: & \mathrm{Mat}_{n}\bigl(\underline{AP_{G_X} \# H}\bigr) & \larr & \underline{(AP_{G_X})^n\# H}  \\
     & \displaystyle\sum_{1\leq i,j \leq n} (a_{i,j}\# h_{i,j})E_{i,j} & \lmap &  \displaystyle \sum_{1\leq i,j \leq n}\bigl(0,  \ldots  , \underbrace{a_{i,j}({(h_{i,j})}_{(1)}\cdot \umgx)}_{i\textrm{-th position} } ,  \ldots  , 0\bigr) \# g_i{(h_{i,j})}_{2}g_j^{-1}
\end{array} $$
Now to see that $\Psi$ is a left inverse of $\Theta$, take any $((a_1, \ldots ,a_n)\# h) \in \underline{(AP_{G_X})^n\# H}$

 \begin{eqnarray*}
    & \, & (\Psi\circ\Theta)\bigl((a_1, \ldots ,a_n)\# h\bigr) =  \displaystyle \Psi\biggl(\sum_{1\leq i,j\leq n}(a_i\# g_i^{-1}hg_j)E_{i,j}\biggr) \\
     & = & \displaystyle \sum_{1\leq i,j \leq n}\bigl(0,  \ldots  , \underbrace{a_i((g_i^{-1}h_{(1)}g_j)\cdot \umgx)}_{i\textrm{-th position} } ,  \ldots  , 0\bigr) \# h_{(2)} \\
     & = & \displaystyle \Bigl(a_1\sum_{j=1}^n(g_1^{-1}h_{(1)}g_j)\cdot\umgx,  \ldots , a_n\sum_{j=1}^n(g_n^{-1}h_{(1)}g_j)\cdot\umgx\Bigr)\# h_{(2)} \\
     & \stackrel{(\ref{equation.H.action.on.APGn})}{=} & \bigl((a_1, \ldots ,a_n)\bigr)\bigl(h_{(1)}\cdot(\umgx, \ldots ,\umgx)\bigr)\# h_{(2)}  =  (a_1, \ldots ,a_n)\# h,
\end{eqnarray*} 
which concludes the proof.
\end{proof}

The previous theorem implies that for any $X \in \Px_1(G)$, 
\[
\Hpar\Gamma_X \stackrel{(\ref{equation.HparGammaX.APGnH})}{\simeq} \underline{(AP_{G_X})^n\# H} \stackrel{\ref{teorema.APGnH.iso.MatnAPGH}}{\simeq} 
\mathrm{Mat}_{n}\left(\underline{AP_{G_X}\# H}\right) \simeq \mathrm{Mat}_{n}\left(\Hpar P_{G_X}\right),
\]
where $n = \frac{|X|}{|G_X|}$, hence 
\begin{equation}\label{equation.HparGammaX.iso.MatnHparGammaGX}
    \Hpar\Gamma_X \simeq \mathrm{Mat}_n\left(\Hpar P_{G_X}\right).
\end{equation}

 The next theorem is a structure theorem for $\Hpar$ which comes as a consequence of the previous results  in this subsection.

\begin{teorema}\label{teorema.Hpar.as.summation.of.MatnHparPGX.}
    Let $H$ be a pointed Hopf algebra with finite group $G$ of grouplike elements. If $X_1, X_2,  \ldots , X_n$ constitute a complete set of representatives in $\mathcal{P}_1(G)$ for the equivalence relation defined in \ref{equivalence.relation}, then
    $$ \Hpar \simeq \bigoplus_{k=1}^t \textrm{Mat}_{n_k}(\Hpar P_{G_{X_k}}), $$
    where $G_{X_k} = \{g \in G \mid gX_k = X_k\}$, $n_k = \dfrac{|X_k|}{|G_{X_k}|}$ and $\Hpar \Gamma_{G_{X_k}}$ is an ideal of $\Hpar$ for all $k=1, \ldots ,t$.
\end{teorema}
\begin{proof}
    This fact follows from this sequence of isomorphisms of algebras
    \[
    \Hpar \stackrel{\ref{proposicao.pointed.Hpar.direct.sum}}{\simeq} \bigoplus_{k=1}^t\Hpar\Gamma_{X_k} \stackrel{(\ref{equation.HparGammaX.iso.MatnHparGammaGX})}{\simeq} \bigoplus_{k=1}^t \textrm{Mat}_{n_k}(\Hpar P_{G_{X_k}})
    \]
\end{proof}

The previous theorem can be rewritten in a form such that the multiplicities of the components are parametrized by subgroups of $G$. This reflects the direct sum decomposition for $\K_{par} G$ obtained in \cite{piccione}  and revisited in \cite{dokuchaev2004isomorphisms}. Instead of indexing by equivalence classes in $\mathcal{P}_1(G)$, corresponding to the connected components of the groupoid $\Gamma(G)$, we label the summands by conjugacy classes of subgroups.
Note that each component of $\Gamma(G)$ determines uniquely such a conjugacy class by associating, to any of its vertices, the conjugacy class of its isotropy subgroup. 

Following \cite{dokuchaev2004isomorphisms}, given a subgroup $L$ of a finite group $G$, let $c_m(L)$ be the number of components of the groupoid $\calG(G)$ having a vertex whose isotropy subgroup is conjugate to $L$, and whose vertices have $m|L|$ elements. In terms of the representatives $X_1, \ldots, X_t$, denote by $\mathcal{O}_i$ the component of $X_i$ in $\calG(G)$, 
\[
c_m(L) = \left| \left\{ \mathcal{O}_i \  ; \ |X_i| = m |L|, \ G_{X} \sim L     \right\} \right|.
\]
 Let $\mathcal{C}$ be a complete set of representatives for the conjugacy classes of subgroups of $G$. Then, from \cite[Thm 2.1]{dokuchaev2004isomorphisms}, 
\[
\K_{par} G  \simeq \bigoplus_{\substack{L \in \mathcal{C} \\ 1 \leq m \leq [G:L]}} c_m(L) \textrm{Mat}_{m}(\K L).
\]
The coefficients $c_m(L)$  are computed recursively via the formula \cite{dokuchaev2004isomorphisms}
\begin{equation}\label{equation.cmL.recursiva}
    c_m(L) = \dfrac{1}{m}[G : N_G(L)] = 
    \left(
    \binom{[G : L] - 1}{m-1} - 
\left( \sum_{ \substack{L < T \leq G, \\  [T:L]\ | \ m}}
    \dfrac{m/[T:L] c_{m/[T:L]}(T)}{[N_G(T):T]}
\right).
\right)
\end{equation}

We can thus present an extension of \cite[Thm 2.1]{dokuchaev2004isomorphisms} for pointed Hopf algebras.

\begin{teorema}\label{corolario.Hpar.iso.Soma.cmMatAH}
     Let $H$ be a pointed Hopf algebra with finite group $G$ of grouplike elements.
     Let $\mathcal{C}$ be a complete set of representatives for the conjugacy classes of subgroups of $G$ and let $c_m(L)$ be as above. Then 
     \[
     \Hpar \hspace{0.1cm} \simeq \bigoplus_{\substack{L \in \mathcal{C} \\ 1 \leq m \leq [G:L]}}  c_m(L) \textrm{Mat}_{m}(\underline{AP_{L}\# H}) { \hspace{0.1cm} \simeq \bigoplus_{\substack{L \in \mathcal{C} \\ 1 \leq m \leq [G:L]}}  c_m(L) \textrm{Mat}_{m}(\Hpar P_L) }
     \]
     where $c_m(L)$ are the multiplicities of the direct summands $\textrm{Mat}_{m}(\underline{AP_{L}\# H})\cong \textrm{Mat}_m (H_{par}P_L)$.
     \end{teorema}

\begin{proof}
Given subgroups $L$ and $L'$ of $G$, let us write $L \sim L'$ in case they are conjugated. Let $X_1, X_2,  \ldots , X_t \in \mathcal{P}_1(G)$ be as in Theorem \ref{teorema.Hpar.as.summation.of.MatnHparPGX.}. We may rearrange the direct sum in that result in the following manner:
    \[
        \Hpar  = \bigoplus_{i=1}^ t \Hpar \Gamma_{X_i} 
         = \bigoplus_{\substack{L \in \mathcal{C} \\ 1 \leq m \leq [G:L]}} 
        \bigoplus_{\substack{|X_i| = m |L|, \\ G_{X_i} \sim L}} \Hpar \Gamma_{X_i};  
    \]
    and then we have the algebra isomorphisms
\[
        \bigoplus_{\substack{L \in \mathcal{C} \\ 1 \leq m \leq [G:L]}} 
        \bigoplus_{\substack{|X_i| = m |L|, \\ G_{X_i} \sim L}} \Hpar \Gamma_{X_i} 
        \stackrel{(\ref{equation.HparGammaX.iso.MatnHparGammaGX})}{\simeq} 
        \bigoplus_{\substack{L \in \mathcal{C} \\ 1 \leq m \leq [G:L]}} 
        \bigoplus_{\substack{|X_i| = m |L|, \\ G_{X_i} \sim L}} \mathrm{Mat}_{m} \left(\Hpar P_{G_{X_i}}\right), 
    \]
 and since $G_X$ is conjugated to $L$, for all $X \in \Px_1(G)$, it follows that  $\Hpar \Gamma_{G_X} \simeq \Hpar \Gamma_L$. In fact, if $g \in G$ is such that $L = gG_Xg^{-1}$,  let  $c'_g:\Hpar \to \Hpar$ be the automorphism given by $c'_g([h]) = [ghg^{-1}]$ (see the proof of Lemma \ref{cor:AP_X=AP_{gX}}). Then $c'_g$  induces an isomorphism of algebras 
 \[
 \begin{array}{cccl}
     c'_g|_{\Hpar P_{G_X}}: & \Hpar P_{G_X} & \larr & \Hpar P_L \\
      & [h]P_{G_X} & \lmap & [ghg^{-1}]P_{gG_Xg^{-1}} = [ghg^{-1}]P_L . 
 \end{array}
 \]
 Finally we have the following 
 
 \begin{eqnarray*}
      & \, & \displaystyle \bigoplus_{\substack{L \in \mathcal{C} \\ 1 \leq m \leq [G:L]}} 
        \bigoplus_{\substack{|X_i| = m |L|, \\ G_{X_i} \sim L}} \mathrm{Mat}_{m} \left(\Hpar P_{G_{X_i}}\right)  \simeq  \displaystyle \bigoplus_{\substack{L \in \mathcal{C} \\ 1 \leq m \leq [G:L]}} 
        \bigoplus_{\substack{|X_i| = m |L|, \\ G_{X_i} \sim L}} \mathrm{Mat}_{m} \left(\Hpar P_{L}\right) \\
        & = & \displaystyle \bigoplus_{\substack{L \in \mathcal{C} \\ 1 \leq m \leq [G:L]}}         c_m(L) \mathrm{Mat}_{m}\left(\Hpar P_{L}\right) 
         \simeq  \displaystyle \bigoplus_{\substack{L \in \mathcal{C} \\ 1 \leq m \leq [G:L]}}         c_m(L) \mathrm{Mat}_{m}\left(\underline{AP_L\# H}\right)
 \end{eqnarray*}
\end{proof}


\subsection{The component $\Hpar P_G$ and the global core of a partial $H$-module \label{section:PG-component}}
Proposition \ref{proposicao.pointed.Hpar.direct.sum} provides a decomposition of $\Hpar$ as a direct sum of the unital ideals $\Hpar \Gamma_X$, and Theorems \ref{teorema.Hpar.as.summation.of.MatnHparPGX.} and 
\ref{corolario.Hpar.iso.Soma.cmMatAH} describe the components  $\Hpar \Gamma_X$ and their multiplicities. Nevertheless,  computing those components still must be done on a case-by-case basis, as it was already the case for group algebras in \cite{piccione}. The component generated by $\Gamma_G= P_G$, however, can be described at once for every pointed Hopf algebra $H$. 

Recall that 
\begin{equation} \label{lema.hGammaG.epshGammaG}
    h\cdot P_G = \varepsilon(h)P_G,
\end{equation}  
    for all $h \in H$ (Corollary \ref{lema.hGammaAG.epshGammaAG}). We can now prove the following.

\begin{teorema}\label{teorema.thetaH.multiplicative.section.pH}
          Let $H$ be a pointed Hopf algebra with finite group $G$ of grouplike elements. Consider the central idempotent  $ \displaystyle P_G = \prod_{g \in G} \epsilon_g. $ Then the map 
    $$\begin{array}{cccc}
        \theta_H: & H & \larr & \Hpar  \\
         & h & \lmap & P_G[h]
    \end{array} $$
    is a multiplicative map. Moreover, $\theta_H$ is a section of the surjective algebra map $p_H:\Hpar \to H$ defined by $p_H([h^1][h^2]   \widedots  [h^n]) = h^1h^2  \widedots   h^n$, for all $h^1, h^2,  \ldots , h^n \in H$.
\end{teorema}
\begin{proof}
    First consider $h,k \in H$ and notice that in $\Hpar$ we have $[h][k] = \epsilon_{h_{(1)}}[h_{(1)}k]$. Recall also that for all $h \in H$ we have that $ h\cdot\Gamma_X = \epsilon_h\Gamma_X $; therefore, for all $X\subseteq G$ with $1\in X$, it follows that
    $$ \Gamma_X[h][k] = \Gamma_X\epsilon_{h_{(1)}}[h_{(2)}k] = \epsilon_{h_{(1)}}\Gamma_X[h_{(2)}k] = \bigl(h_{(1)}\cdot\Gamma_X\bigr)[h_{(2)}k]. $$
    When $X=G$, by Equation \eqref{lema.hGammaG.epshGammaG}, $h\cdot P_G=\varepsilon(h)P_G$, hence
    $$ P_G[h][k] = \bigl(h_{(1)}\cdot P_G\bigr)[h_{(2)}k] = \varepsilon(h_{(1)})P_G[h_{(2)}k] = P_G[hk], $$
    for all $h,k \in H$. Since $P_G$ is a central idempotent of $\Hpar$, it follows that $\theta_H$ is multiplicative:
    $$
        \theta_H(h)\theta_H(k)  =  P_G[h]P_G[k]  =  \bigl(P_G\bigr)^2[h][k]  =  P_G[hk]  = \theta_H(hk),
    $$
    for all $h,k \in H$.

    Next, in order to see that $\theta_H: H \to \Hpar$ is a section of $p_H: \Hpar \to H$, notice that $$ p_H([h_{(1)}][S(h_{(2)})]) = h_{(1)}S(h_{(2)}) = \varepsilon(h)1_H $$
    for each $h \in H$, hence 
    $$ p_H(P_G) = \prod_{g\in G}p_H(\epsilon_g) = \prod_{g\in G}\varepsilon(g)1_H = \biggl(\prod_{g\in G}1_\K\biggr)1_H = 1_H.  $$

    It follows that for any $h \in H$: 
    $$ \bigl(p_H\circ\theta_H\bigr)(h) = p_H\bigl(P_G[h]\bigr) = 1\um h = h. $$
    
\end{proof}

\begin{corolario}\label{corolario.HparGammaG.iso.H}
          Let $H$ be a pointed Hopf algebra with finite group $G$ of grouplike elements, and consider the central idempotent $ \displaystyle P_G = \dfrac{1}{|G|}\sum_{g\in G}g\cdot P_G = P_G = \prod_{g \in G} \epsilon_g$. The ideal $P_G\Hpar$ of $\Hpar$ generated by $P_G$ is isomorphic to $H$ as an algebra. In other words, we have
    $$ \xymatrix{
    H \ar[r]_-{\theta_H}^-{\simeq} & P_G\Hpar \ar@{^{(}->}[r] & \Hpar
     }, $$
    and since $P_G$ is a central idempotent of $H_{\textrm{par}}$, we have the isomorphism of algebras
    $$ H_{\textrm{par}} = (1_{\Hpar}-P_G)H_{\textrm{par}} \oplus P_GH_{\textrm{par}} \simeq (1_{\Hpar}-P_G)H_{\textrm{par}} \oplus H. $$    
\end{corolario}

It follows from the previous corollary that every partial $H$-module decomposes as a direct sum of partial $H$-modules
\begin{equation} \label{eq:module.decomposition}
M = P_G M \oplus (1_{\Hpar} - P_G)M.    
\end{equation}
This decomposition has an interesting interpretation in terms of the global core and the global shadow introduced in \cite{alves2019dilations}. 
\begin{definicao}
    \cite{alves2019dilations}
   Let $M$ be a partial $H$-module. The  \emph{global core} $c(M)$ is the greatest global $H$-submodule of $M$, and the \emph{global shadow} $s(M)$ is the greatest global quotient module of $M$.
\end{definicao}
The global core is right adjoint to the inclusion functor $H - \Mod \larr \Hpar - \Mod$, and the global shadow is left adjoint to the same funtor. 

We begin by utilizing decomposition \eqref{eq:module.decomposition} to determine the global core of $M$. It can be shown that 
\begin{equation}
c(M) = 
\{ m \in M \mid  [h] [k]  m = [hk] m, \forall h,k \in H\},    
\end{equation}
and we also have that 
\begin{equation} \label{eq:c(M)}
c(M) = \{ m \in M \mid  \epsilon_h   m = \varepsilon(h) m, \forall h \in H\}.
\end{equation}
This second equality follows from the fact that a partial representation $\pi: H \to B$ is a global representation if and only if $\epsilon_h^{\pi} = \varepsilon(h) 1_B$ for every $h \in H$ \cite[p. 147]{alves2015partial}.

A partial $H$-module which has $c(M) = 0$ is then a partial $H$-module without  global $H$-submodules. 

\begin{definicao}
    A partial $H$-module $M$ is a \emph{purely partial} $H$-module if $c(M) = 0$.
\end{definicao}

The structure of partial modules over the Sweedler algebra $H_4$ was investigated in \cite[Example 2.7]{alves2019dilations}, where it was shown that the global core of every partial $H_4$-module is a direct summand and that the global core and global shadow of a partial $H_4$-module coincide. We will see that both results hold for partial $H$-modules when $H$ is a pointed Hopf algebra with finite group of grouplikes. 

\begin{proposicao}
              Let $H$ be a pointed Hopf algebra with finite group $G$ of grouplike elements. Let $M$ be a left $\Hpar$-module and let $c(M)$ be its global core. Then 
 \begin{enumerate}[(i)]
 \item $ c(M) = P_G M$;
 \item $ (1_{\Hpar } - P_G) M$ is a purely partial $H$-module.
 \end{enumerate}
\end{proposicao}
\begin{proof}
(i) 
If $m \in P_G M$, then 
$m = P_G \ m$ and
\[
[h] [k] m = 
[h] [k] (P_G \ m ) 
= [hk]P_G \ m = [hk] m.
\]
Conversely, if $m \in c(M)$ then $\epsilon_g m = m$ for every grouplike $g$, and hence $m = P_G \ m \in  P_G M$. 

(ii) It follows from (i), because if $N$ is a partial $H$-submodule of $M$ then $c(N) =  N \bigcap c(M)$.
\end{proof}

Now for the global shadow. We define $Q(M)$ as the partial $H$-submodule of $M$ generated by elements of the form 
\[
[h] [k]  m - [hk] m,
\]
where $ h,k \in H, \ m \in M$, and then it is clear that 
\begin{equation}
   s(M) = \dfrac{M}{Q(M)}.
\end{equation}
It is clear that $Q(M) = 0$ if $M$ is a global $H$-module, and hence $s(M) = M$ in this case. 

\begin{proposicao} \label{globalcore=globalshadow}
Let $H$ be a pointed Hopf algebra with finite group $G$ of grouplike elements. Let $M$ be a partial $H$-module, $c(M)$ be its global core and $s(M)$ be its global shadow. Then 
 $ c(M) \simeq s(M)$. 
\end{proposicao}

\begin{proof}
Since $\Gamma_{X}$ and this is a central idempotent in $\Hpar$ it follows from Proposition \ref{proposicao.Gamma.central.idempotent} that
\[
[h][k]\Gamma_{X} m - [hk] \Gamma_{X} m = 
\Gamma_{X} ([h][k]  - [hk])  m  \in \Gamma_{X} M
\]
for each $h,k \in H$ and $m \in M$. Hence, if $X_1, X_2, \ldots, X_t$ are as before, then 
\[
Q(M) = \displaystyle \bigoplus_{i=1}^t Q(\Gamma_{X_i} M), 
\]
and there is a canonical  isomorphism of global $H$-modules 
\[
s(M) = 
\dfrac{M}{Q(M)} \simeq \bigoplus_{i=1}^t \dfrac{\Gamma_{X_i} M}{Q(\Gamma_{X_i} M)}
= \bigoplus_{i=1}^t s(\Gamma_{X_i} M). 
\]

Note that for each $g \in G(H)$ and $m \in M$, 
\[
(\epsilon_g - 1_{\Hpar} ) m = (\epsilon_g - \varepsilon(g) ) m
\in Q(M). 
\]

Recall that, as a vector space 
$$\Gamma_XM = \bigoplus_{Y \sim X}P_YM. $$
Assume that $X \in \Px_1(G), X \neq G$. For each $Y \sim X$, there exists $g \in G \setminus Y$ and then $P_Ym = (1_{\Hpar} - \eps_g)P_Ym$, hence $P_YM \subseteq Q(\Gamma_XM)$, what lead us to the fact that $Q(\Gamma_X M) = \Gamma_X M$.

Finally, $P_G M$ is a global $H$-module and hence 
\[
s(M) \simeq  P_G M = c(M). 
\]
  
\end{proof}

	
	\section{Examples}
	\label{se:examples}

 We are going to present three examples.  The first one is the algebra of partial representations of a finite group, which was already described in the article \cite{piccione}. The other two examples come from two finite-dimensional pointed Hopf algebras of rank one \cite{krop2006finite} that share the same group of grouplike elements.

The strategy we will use to describe the $\Hpar$ of the previously mentioned algebras is as follows. According to Theorem \ref{teorema.Hpar.as.summation.of.MatnHparPGX.} there exist $X_1, X_2,  \ldots , X_t \in \Px_1(G)$ such that
\[
\Hpar \simeq \bigoplus_{i=1}^t \Mat_{n_i}(\Hpar P_{G_{X_i}}).
\]
Since $P_{G_{X_i}} \in A$, we have that $\Hpar P_{G_{X_i}}\simeq \underline{AP_{G_{X_i}}\# H}$. In order to describe the ideals $\Hpar P_{G_{X_i}}$, we first describe $AP_{G_{X_i}}$ and use it in the partial smash product.

\subsection{The group algebra}
\label{sub.Group_Algebra}

We now consider the case where $H$ is the group algebra, $H=\K G$, of a finite group $G$. Our aim in this subsection is to recover the description of the algebra $\Hpar$ presented in  \cite{piccione}. We are going to use the notation $\kpar$ instead of $\Hpar$.

By Theorem \ref{corolario.Hpar.iso.Soma.cmMatAH},
\begin{equation}\label{equation.KparG.soma.cmLMmKparGammaL}
    \kpar \simeq \bigoplus_{\substack{L \in \mathcal{C} \\ 1 \leq m \leq [G:L]}}  c_m(L) \textrm{Mat}_{m}(\kpar P_L),
\end{equation}
where $\mathcal{C}$ is a complete set of representatives for the conjugacy classes of $G$, and $c_m(L)$ can be computed recursively via the formula \ref{equation.cmL.recursiva}, recall that $\kpar P_{L} = \underline{AP_{L}\# \K G}$. So to calculate the algebra $\kpar$, first we are going to determine the algebra $A$, after that we are going to focus on the ideals $AP_L$, for $L$ a subgroup of $G$, and on its partial $G$-action. After that, we are going to compute the partial smash product $\underline{AP_L\# \K G}$ for each $L$. Finally, we are going to put everything together, with (\ref{equation.KparG.soma.cmLMmKparGammaL}) to describe the algebra $\kpar$.

It follows directly from the definition of $P_X$ that, for $g^{(1)}, g^{(2)},  \ldots , g^{(n)} \in G$,
$$
\eps_{g^{(1)}}\eps_{g^{(2)}}\cdots\eps_{g^{(n)}}P_X = \left\{\begin{array}{cl}
    P_X, & \textrm{if } g^{(1)}, g^{(2)},  \ldots , g^{(n)} \in X, \\
    0, & \textrm{otherwise}. 
\end{array} \right.
$$
Hence, for every $X$, the ideal $AP_X $ is isomorphic to $ \K$ as a vector space. Also, the elements $P_X$ are pairwise orthogonal central idempotents of the algebra $A$, hence by Proposition \ref{proposicao.A.as.PxA.sommation} we conclude that
$$ A = \bigoplus_{X \in \Px_1(G)} AP_X \simeq \bigoplus_{X\in \Px_1(G)} \K. $$

\begin{obs}
    The equality above is correct at the level of algebras but it disregards the partial $H$-action. Thus, to obtain a more precise description in the setting of partial $H$-module algebras, it is preferable to combine Theorem~\ref{teorema.A.GammaA.direct.somma} with Proposition~\ref{proposicao.APGn.iso.AGammaX}. From these results it follows that there exist $X_1,\ldots,X_m \in \Px_1(G)$ such that
\[
A = \bigoplus_{i=1}^m A\Gamma_{X_i}
\cong \bigoplus_{i=1}^m (AP_{G_{X_i}})^{n_i}
\cong \bigoplus_{i=1}^m \K^{n_i},
\]
where $n_i=\frac{|X_i|}{|G_{X_i}|}$.

Moreover, for each $1\leq i\leq m$, the component $\K^{n_i}$ carries a structure of partial $H$-module algebra, which can be regarded as a partial $H$-submodule algebra of $A$.
\end{obs}

Notice that for each subgroup $L$ of $G$, the ideal $AP_L$ carries a partial $G$-action which depends on $L$. This partial $G$-action is given by the restriction of the partial $G$-action on $A$, and explicitly it is given by 
\[
g\cdot P_L = \left\{\begin{array}{cc}
    P_L, & \textrm{ if } g\in L  \\
    0, & \textrm{ otherwise.} 
\end{array}\right.
\]

Now we are going to use the partial smash product to describe the algebra $\Hpar P_L \simeq \underline{AP_L\# H}$, for each subgroup $L$ of $G$. Consider $L$ a subgroup of $G$ and $g \in G$, then 
$$
P_L\# g = P_L(g\cdot 1_A)\otimes g = P_L\eps_g \otimes g = \left\{\begin{array}{cc}
    P_L\x g, & \textrm{if } g\in L, \\
    0, & \textrm{otherwise},
\end{array}\right.
$$

and $ \bigl\{P_L\# g \mid g\in L\bigr\} $
is a basis of $\underline{AP_L\# H}$. Now, notice that for $g,h \in L$ we have that 
$$ (P_L\# g)(P_L\# h) = P_L(g\cdot P_L)\# gh = P_L\# gh, $$
then, we can deduce that $\kpar \Gamma_L \simeq \underline{AP_L\#\K G} \simeq \K L$.

Finally, we have that 
\[
\kpar \simeq \bigoplus_{\substack{L \in \mathcal{C} \\ 1 \leq m \leq [G:L]}}  c_m(L) \textrm{Mat}_{m}(\K L)
\]

\subsection{A Hopf algebra of rank one: nilpotent case} 
\label{section.nilpotent}

Finite-dimensional pointed Hopf algebras of rank one over an algebraically closed field $\K$ of characteristic zero were introduced and classified  in \cite{krop2006finite}. 
A Hopf algebra has rank $n$ if its coradical $H_0$ is a Hopf subalgebra, $H_1$ generates $H$ as an algebra, and the dimension of $\K \otimes_{H_0} H_1$ is $n+1$, where $H_1 = H_0 \wedge H_0$ is the second term of the coradical filtration (see equality \eqref{equation.coradical.filtration}), $\K$ is a right $H_0$-module via the trivial representation and  $H_1$ is a left $H_0$-module via multiplication in $H$.

Note that if $\dim_{\K} (\K \otimes_{H_0} H_1) = 2$ then we have that $H_1 = H_0 \oplus H_0x$ for some $x \in H_1 \setminus H_0$. This element is then a skew-primitive of $H$ and, without loss of generality, we may take $x$ such that $\Delta(x) = x \otimes 1 + a \otimes x$ for some $a \in G$.  

Each pointed Hopf algebra of rank one is associated to a group datum. According to \cite{krop2006finite}, a \textbf{group datum} 
is a 4-tuple $\mathcal{D} = (G,\chi, a,\kappa)$, where $G$ is a finite group, $\chi: G \to \K^{\times}$ is a character for $G$, $a \in Z(G)$ and $\kappa \in \K$, where either  $\chi^n = 1$ or $ \kappa (a^n - 1) = 0$, where $n$ is the
order of $\chi(a)$. 
We say that $\mathcal{D}$ is of nilpotent type if $ \kappa (a^n - 1) = 0$, and otherwise we say that $\mathcal{D}$ is of non-nilpotent type. 
Given a group datum $\mathcal{D}$, there is an associated Hopf algebra $H_{\mathcal{D}}$ which is given as follows: it is generated by the elements of $G$ and by an element $x$, the relations between elements of $G$ hold in $H_{\mathcal{D}}$, and also 
\[
x^n = \kappa (a^n -1), \ \ \ \ xg = \chi(g) gx, 
\]
for all $g \in G$. The comultiplication is given on the generators by  
\[
\Delta(x) = x \otimes 1 + a \otimes x, \ \ \ \ \Delta(g) = g \otimes g
\]
and this determines the counit and antipode: since $x$ is a skew-primitive and $g$ is a grouplike element for each $g \in G$, we have that \[
\varepsilon(x) = 0, \ \ \ \varepsilon(g) = 1, 
\] \[
S(x) = -a^{-1}x, \ \ \ \  S(g) = g^{-1}.
\]
Finally, a $\K$-basis for $H_{\mathcal{D}}$ is the set 
\[
\{
gx^ m; \ g \in G, 0 \leq m < n
\},
\]
and therefore this algebra has dimension $|G|n$. 

In \cite[Thm 1]{krop2006finite} it is shown that each finite-dimensional pointed Hopf algebra of rank one over an algebraically closed field $\K$ of characteristic zero is isomorphic to a Hopf algebra $H_{\mathcal{D}}$, that each Hopf algebra $H_{\mathcal{D}}$ is in fact a pointed Hopf algebra of rank one, and isomorphisms between Hopf algebras $H_{\mathcal{D}}$ and $H_{\mathcal{D'}}$ are described in terms of equivalences between the group data $\mathcal{D}$ and $\mathcal{D'}$.

In the following we consider two rank one Hopf algebras of dimension 8.

\bigskip

Let $H$ be the Hopf algebra generated by $g$ and $x$, subject to the following relations:
\begin{equation}\label{equation.1.nilpotent}
        g^4=1, \ \ x^2=0, \ \ xg = -gx, 
    \end{equation}  
    and with structural maps, as follows: \begin{equation}\label{equation.2.nilpotent}  \Delta(g) = g\x g, \ \ \Delta(x)=x\x 1 + g\x x, 
    \end{equation} 
    \begin{equation}\label{equation.3.nilpotent}
  S(g)=g^{3}, \ \ S(x) = gx. 
    \end{equation} 
    This algebra has the set $\{1, g, g^2, g^3, x, gx, g^2x, g^3x\}$ as a basis. This is a pointed Hopf algebra  
    of rank one with group datum  $\mathcal{D} = (G, \chi, g, 0)$, where the group of grouplike elements is $G=\{1,g,g^2,g^3\} \simeq \Z_4$ and $\chi(g^m) = (-1)^m$.

    Our aim is to describe $\Hpar$ for $H$. In order to do that, we use Theorem \ref{teorema.Hpar.as.summation.of.MatnHparPGX.}, which states that there exist $X_1,  \ldots , X_t \in \Px_1(G),$ such that  
    \[ \Hpar \simeq \bigoplus_{i=1}^t \Mat_{n_i}(\Hpar P_{G_{X_i}}) \simeq \bigoplus_{i=1}^t \Mat_{n_i}\bigl(\underline{AP_{G_{X_i}}\# H}\bigr), \]
    where $G_{X_i}$ is a subgroup of $G \simeq \Z_4$ and $n_i = |X_i|/|G_{X_i}|$.

    As in the previous example, we shall describe the algebra $AP_L$ for each subgroup $L$ of $G$ and use the partial smash product to study the algebras $\underline{AP_{G_{X_i}}\# H}$.

    In order to choose elements $X_1,  \ldots , X_t$ and to compute the numbers $n_1,  \ldots , n_t$, we build the groupoid $\mathcal{G}(G)$:

\color{black}

\begin{center}
    \begin{tikzpicture}[
    ->, thick,
    every node/.style={inner sep=1pt, font=\normalsize},
    shorten >=2pt, shorten <=2pt,
    every loop/.style={looseness=8, min distance=10mm}
  ]

\node (A) at (0,0) {$\{1\}$};
\draw (A) edge[loop above] node {$1$} (A);

\node (B) at (3,0) {$G$};
\draw (B) edge[loop above] node {$1$} (B);
\draw (B) edge[loop below] node {$g$} (B);
\draw (B) edge[loop left]  node {$g^2$} (B);
\draw (B) edge[loop right] node {$g^3$} (B);

\node (D) at (6,0) {$\{1,g^2\}$};
\draw (D) edge[loop above] node {$1$} (D);
\draw (D) edge[loop below] node {$g^2$} (D);

\node (C1) at (9,0) {$\{1,g\}$};
\node (C2) at (11,0) {$\{1,g^3\}$};
\draw (C1) edge[loop left] node {$1$} (C1);
\draw (C2) edge[loop right] node {$1$} (C2);
\draw (C1) to[bend left=15] node[midway, above] {$g^3$} (C2);
\draw (C2) to[bend left=15] node[midway, below] {$g$} (C1);

\node (E1) at (4,-3) {$\{1,g,g^2\}$};
\node (E2) at (8,-3) {$\{1,g,g^3\}$};
\node (E3) at (6,-5) {$\{1,g^2,g^3\}$};

\draw (E1) edge[loop left] node {$1$} (E1);
\draw (E2) edge[loop right] node {$1$} (E2);
\draw (E3) edge[loop below] node {$1$} (E3);

\draw (E1) to[bend left=10] node[sloped, above] {$g^3$} (E2);
\draw (E1) to[bend left=14] node[sloped, above] {$g^2$} (E3);

\draw (E2) to[bend left=10] node[midway, below] {$g$} (E1);
\draw (E2) to[bend left=14] node[sloped, below] {$g^3$} (E3);

\draw (E3) to[bend left=7] node[sloped, below] {$g^2$} (E1);
\draw (E3) to[bend left=12] node[sloped, above] {$g$} (E2);

\end{tikzpicture}

\end{center}

Then we choose the elements $\{1\}; \ G; \ \{1,g^2\}; \ \{1,g\}; \ \{1,g,g^2\}$ of $\Px_1(G)$ as representatives of the components of this groupoid. With respect to the isotropy groups, we have that $G_{\{1\}} = G_{\{1,g\}} = G_{\{1,g,g^2\}} = \{1\}, \ G_{\{1,g^2\}} = \langle g^2 \rangle$, and $ \ G_G = G$. Hence $(n_1,n_2,n_3,n_4,n_5) = (1,1,1,2,3)$, and 
$$\begin{array}{ccl}
    \Hpar & \simeq &  \Hpar P_{G} \oplus \Hpar P_{\{1,g^2\}} \oplus \Hpar P_{\{1\}} \oplus \Mat_{2}(\Hpar P_{\{1\}}) \oplus \Mat_{3}(\Hpar P_{\{1\}}) \\
     & \simeq & \tiny \Bigl(\underline{AP_G\#H}\Bigr) \oplus \Bigl(\underline{AP_{\langle g^2 \rangle}\#H}\Bigr) \oplus \Bigl(\underline{AP_{\{1\}}\#H}\Bigr) \oplus \Mat_{2}\bigl(\underline{AP_{\{1\}}\#H}\bigr) \oplus \Mat_{3}\bigl(\underline{AP_{\{1\}}\#H}\bigr)
\end{array}$$

Hence, it suffices to study the unital ideals  $AP_{\{1\}}, AP_{\langle g^2 \rangle}$ and $AP_{G}$ of $A$, then compute the corresponding partial smash products. 

It follows from Theorem \ref{teorema.thetaH.multiplicative.section.pH} that 
\[
\eps_h P_G = \varepsilon(h) P_G
\]
for every $h \in H$, and therefore $A P_G \simeq \K$. The other two components  $P_{\{1\}}$ and $ P_{\langle g^2 \rangle}$ need further work.

By definition, the generators of the algebra $A$ are subject to the following relations:
\begin{eqnarray}
\epsilon_{1_H}& =& 1_A, 
\label{equality.Apar.0} \\
\label{equality.Apar.1}
\epsilon_h & = &  \epsilon_{h_{(1)}}\epsilon_{h_{(2)}}, \\ 
\label{equality.Apar.2}
\epsilon_{h_{(1)}k}\epsilon_{h_{(2)}} & = & \epsilon_{h_{(1)}}\epsilon_{h_{(2)}k},
\end{eqnarray}
    for all $h,k \in H$; since the mapping $h \in H \mapsto \epsilon_h \in A$ is linear, we only need to consider these equalities for $h,k$ in a basis for $H$. 

\begin{lemma}
 Let $\{1, g, g^2, g^3, x, gx, g^2x, g^3x\}$ be the canonical basis for $H$. Then 
\begin{equation}
    \epsilon_{g^m}^2 = \epsilon_{g^m};
\end{equation}
\begin{equation}
    \label{equation.glx.glxga.gl.glx}
        \epsilon_{g^lx} = \epsilon_{g^lx}\epsilon_{g^{l}} + \epsilon_{g^{l+1}}\epsilon_{g^lx};
    \end{equation}
    \begin{equation}
    \label{equation.gAxgL.gLxgA}
    (-1)^{\alpha+l}\epsilon_{g^{\alpha}x}\bigl[\epsilon_{g^{l+1}}-\epsilon_{g^{l}}\bigr] = \epsilon_{g^{l}x}\bigl[\epsilon_{g^{\alpha+1}}-\epsilon_{g^{\alpha}}\bigr];
    \end{equation}
\begin{equation}\label{equation.gAxgLx.gLxgAx}
        \epsilon_{g^{\alpha}x}\epsilon_{g^lx} = \epsilon_{g^lx}\epsilon_{g^{\alpha+1}x}.
    \end{equation}
    for $m, l, \alpha = 0,1,2,3$.
\end{lemma}
\begin{proof}
     To begin with, by substituting $h=g^m$ in Equality \eqref{equality.Apar.1}, and substituting $h=g^m$ and $k=g^{-m}h'$ in Equality \eqref{equality.Apar.2},  we obtain that $\epsilon_{g^m}$ is a central idempotent of $A$ for $m=0,1,2,3$. 
    
    Now consider $h=g^lx$ in Equality \eqref{equality.Apar.1} for some $l \in \{0,1,2,3\}$.  Since $\Delta(g^lx)=g^lx \x g^{l} + g^{l+1}\x g^lx$, we have 
    \begin{equation*}
        \epsilon_{g^lx} = \epsilon_{g^lx}\epsilon_{g^{l}} + \epsilon_{g^{l+1}}\epsilon_{g^lx}.
    \end{equation*}
    Next, Equality \eqref{equality.Apar.2}  may be studied in two cases according to the choices of $k$ and $h$. 
First,  setting $k=g^m$ and $h=g^lx$, we obtain
    $$ \epsilon_{g^lxg^m}\epsilon_{g^{l}} + \epsilon_{g^{l+1+m}}\epsilon_{g^lx} = \epsilon_{g^lx}\epsilon_{g^{m+l}} + \epsilon_{g^{l+1}}\epsilon_{g^lxg^m}. $$
    Defining a new parameter $\alpha=l+m$, we have
    $$ (-1)^{\alpha+l}\epsilon_{g^{\alpha}x}\epsilon_{g^{l}} + \epsilon_{g^{\alpha+1}}\epsilon_{g^lx} = \epsilon_{g^lx}\epsilon_{g^{\alpha}} + (-1)^{\alpha+l}\epsilon_{g^{l+1}}\epsilon_{g^{\alpha}x}, $$
    therefore
    \begin{equation*}
    \label{equation.gAxgL.gLxgA.1}
(-1)^{\alpha+l}\epsilon_{g^{\alpha}x}\bigl[\epsilon_{g^{l+1}}-\epsilon_{g^{l}}\bigr] = \epsilon_{g^{l}x}\bigl[\epsilon_{g^{\alpha+1}}-\epsilon_{g^{\alpha}}\bigr].
    \end{equation*}
    Now if we substitute $k=g^mx$ and $h=g^lx$ in \eqref{equality.Apar.2} and use the fact that $x^2=0$, we obtain the equations 
    \begin{equation*}
    \label{equation.gAxgLx.gLxgAx.1}
        \epsilon_{g^{\alpha}x}\epsilon_{g^lx} = \epsilon_{g^lx}\epsilon_{g^{\alpha+1}x},
    \end{equation*}
    where $\alpha=l+m$. 
    \end{proof}

    The next result describes the structure of the base algebra $A$.

    \begin{proposicao}  \label{proposicao.Apar.nilpotente}
        Let $H$ be the Hopf algebra defined by \eqref{equation.1.nilpotent}, \eqref{equation.2.nilpotent} and \eqref{equation.3.nilpotent}. Then $A = \apar$ is isomorphic, as an algebra, to 
\[
\mathbb{K} \oplus  \mathbb{K}[X] \oplus  \bigl(\mathbb{K}[X]/\langle X^2 \rangle\bigr)^6.
\]
    \end{proposicao}
\begin{proof}
By Theorem \ref{teorema.A.GammaA.direct.somma}, and the structure of the groupoid $\mathcal{G}(G)$, we conclude that 
\[
A \simeq AP_{G} \oplus AP_{\langle g^2\rangle} \oplus AP_{\{1\}} \oplus (AP_{\{1\}})^2 \oplus (AP_{\{1\}})^3 \simeq AP_{G} \oplus AP_{\langle g^2\rangle} \oplus (AP_{\{1\}})^6,
\]
so to compute the algebra $A$, we only have to calculate the unital ideals $AP_{G}$, $AP_{\langle g^2\rangle}$ and $AP_{\{1\}}$. Multiplying the equalities  \eqref{equation.glx.glxga.gl.glx}, \eqref{equation.gAxgL.gLxgA}, and \eqref{equation.gAxgLx.gLxgAx} by $P_{\{1\}}$, we obtain
\[         \eps_{gx}P_{\{1\}}  =  \eps_{g^2x}P_{\{1\}} \ = \ 0, \ \ \
         \eps_{g^3x}P_{\{1\}}  =  \eps_{x}P_{\{1\}}, \ \ \ 
         (\eps_x)^2P_{\{1\}}  =  0,
\]
hence there is an algebra isomorphism $\K[X]/\langle X^2 \rangle \simeq AP_{\{1\}}$. 
    We now describe the equations associated to  $P_{\langle g^2\rangle}$.
    
    First, we have that  $\eps_{g^l}P_{{\langle g^2\rangle}} = (\delta_{l,0}+\delta_{l,2})P_{{\langle g^2\rangle}}$, 
    so  equation \eqref{equation.glx.glxga.gl.glx} gives us
    $$ \eps_{g^lx}P_{{\langle g^2\rangle}} = \bigl(\delta_{l,0}+\delta_{l,1}+\delta_{l,2}+\delta_{l,3}\bigr)\eps_{g^lx}P_{{\langle g^2\rangle}} = \eps_{g^lx}P_{{\langle g^2\rangle}}, $$
    which does not provide any information. Now working with \eqref{equation.gAxgL.gLxgA}, we see that
    $$ 
    (-1)^{\alpha+l} \eps_{g^{\alpha}x}P_{{\langle g^2\rangle}}\underbrace{\bigl[(\delta_{l,1}+\delta_{l,3})-(\delta_{l,0}+\delta_{l,2})\bigr]}_{=(-1)^l} =  \eps_{g^{l}x}P_{{\langle g^2\rangle}} \underbrace{\bigl[(\delta_{\alpha,1}+\delta_{\alpha,3})-(\delta_{\alpha,0}+\delta_{\alpha,2})\bigr]}_{=(-1)^{\alpha}}
    $$
    $$ \Rightarrow \ \eps_{g^{\alpha}x}P_{\langle g^2\rangle} = \eps_{g^lx}P_{\langle g^2\rangle}; $$
    and once again we do not get any new information from the equations \eqref{equation.gAxgLx.gLxgAx}. Finally, $AP_{\langle g^2\rangle}$ is only subject to the following equations
    $$
    \eps_{g^lx}P_{\langle g^2\rangle} = \eps_{x}P_{\langle g^2\rangle}, \ \forall l=0,1,2,3.
    $$
 Since this is the only equation that constrains the algebra \( AP_{\langle g^2\rangle} \), this algebra is freely generated by a single generator, namely \( \eps_xP_{\langle g^2\rangle}  \), 
which yields an isomorphism of algebras
\(\mathbb{K}[X] \simeq AP_{\langle g^2\rangle}.
\) Since $A = AP_{G} \oplus AP_{\langle g^2\rangle} \oplus (AP_{\{1\}})^6$, the result follows. 
\end{proof}

\begin{obs}
    Observe that the algebra isomorphism established in the previous proposition concerns only the algebra structure and does not explicitly reflect the partial $H$-module structure. To properly describe the latter, it is preferable to express $A$ in the following way:
\[
\begin{array}{rl}
A \ \cong & AP_{G} \oplus AP_{\langle g^2\rangle} \oplus AP_{\{1\}} \oplus (AP_{\{1\}})^2 \oplus (AP_{\{1\}})^3 \\
   \cong & \K \oplus \K[X] \oplus \K[X]/\langle X^2 \rangle 
           \oplus \bigl(\K[X]/\langle X^2 \rangle\bigr)^2 
           \oplus \bigl(\K[X]/\langle X^2 \rangle\bigr)^3 .
\end{array}
\]
Each of these components carries a structure of partial $H$-module algebra and may be regarded as a partial $H$-submodule algebra of $A$.
\end{obs}

We proceed to describe the structure of the algebra $\Hpar$. In order to do that, it will be necessary to define the partial $H$-actions on $AP_{G}, AP_{\{1\}}$ and $AP_{\langle g^2\rangle}$, and use them to describe the algebras $\underline{AP_G\# H}$, $\underline{AP_{\{1\}}\# H}$ and $\underline{AP_{\langle g^2\rangle}\# H}$, defining a basis of $\underline{AP_L\# H}$ for each subgroup $L$ of $G$. To do that, take a basis $\beta_1$ of $AP_L$, and the basis $\beta_2 = \{1,g,g^2,g^3,x,gx,g^2x,g^3x\}$ of $H$; then $\{a\# h \mid a \in \beta_1 \textrm{ and } h \in \beta_2\}$ is a generating set for $\underline{AP_L\# H}$. The identity
$$ a\# h = a\eps_{h_{(1)}}\otimes h_{(2)}, $$
for all $a \in A$ and $h \in H$, will generate some compatibility equations for the basis of $\underline{AP_L\# H}$.

{  It follows from Proposition \ref{proposicao.Gamma.central.idempotent} that $\Gamma_L = P_L$ when $L$ is a subgroup, and also
$$ h\cdot (aP_L) = (h_{(1)}\cdot a)(h_{(2)}\cdot P_L) = (h_{(1)}\cdot a)\eps_{h_{(2)}}P_L = (h_{(1)}\cdot a)(h_{(2)}\cdot 1_{A})P_L = (h\cdot a)P_L, $$
for all  $a \in A$ and $h \in H$.}

\begin{itemize}
    \item[i)] 
    For the algebra $AP_{G}$ it follows, from Corollary \ref{corolario.HparGammaG.iso.H}, that
$$ \underline{AP_G\# H} \simeq \Hpar\Gamma_G \simeq H. $$
    \item[ii)] For the algebra $AP_{\{1\}}$, we recall that 
    \[         \eps_{gx}P_{\{1\}}  =  \eps_{g^2x}P_{\{1\}} \ = \ 0, \ \ \
         \eps_{g^3x}P_{\{1\}}  =  \eps_{x}P_{\{1\}}, \ \ \ 
         (\eps_x)^2P_{\{1\}}  =  0.
\]
Then the algebra $AP_{\{1\}}$ has the set $\{ P_{\{1\}}, \eps_xP_{\{1\}} \}$ as a basis. The partial $H$-action is defined by the following equalities:
\[ \begin{array}{cclc}
     g^l\cdot P_{\{1\}} = \delta_{l,0}P_{\{1\}} & \textrm{and} & g^l\cdot (\eps_xP_{\{1\}}) = \delta_{l,0}\eps_xP_{\{1\}}, & \textrm{for all } i=0,1,2,3,  \\
     & & & \\
     g^lx\cdot P_{\{1\}} = \bigl(\delta_{l,0}+\delta_{l,3}\bigr)\eps_x P_{\{1\}} & \textrm{and} & g^lx\cdot (\eps_xP_{\{1\}}) = 0, & \textrm{for all } i=0,1,2,3.
\end{array}
\]

Let us find a basis for $\underline{AP_{\{1\}}\#H}$. Computing the generators, we have:
$$
    P_{\{1\}}\# g^l \ = \ \delta_{l,0}P_{\{1\}}\x g^l, \ \ \eps_xP_{\{1\}}\# g^l \ = \ \delta_{l,0}\eps_{x}P_{\{1\}}\x g^l.
$$
we also have:
$$ 
\left\{\begin{array}{ccl}
    P_{\{1\}}\# x & = & \eps_{x}P_{\{1\}}\x 1_H + \eps_gP_{\{1\}}\x 1_H \ = \  \eps_{x}P_{\{1\}}\x 1_H, \\
    P_{\{1\}}\# gx & = & \eps_{gx}P_{\{1\}}\x g + \eps_{g^2}P_{\{1\}}\x gx \ = \ 0, \\
    P_{\{1\}}\# g^2x & = & \eps_{g^2x}P_{\{1\}}\x g^2 + \eps_{g^3}P_{\{1\}}\x g^2x \ = \ 0, \\
    P_{\{1\}}\# g^3x & = & \eps_{g^3x}P_{\{1\}}\x g^3 + P_{\{1\}}\x g^3x \ = \ P_{\{1\}}\x g^3x, \\
\end{array}\right.
$$
and by a similar computation 
\[
\eps_xP_{\{1\}}\# g^ix = \delta_{i,0}(\eps_x)^2P_{\{1\}}\otimes 1_H + \delta_{i,3}\eps_xP_{\{1\}}\otimes g^3x, \ \textrm{ for all } i=0,1,2,3.
\]

Hence it follows that $\{P_{\{1\}}\# 1; \ \eps_xP_{\{1\}}\# 1; \ P_{\{1\}}\# g^3x; \ \eps_xP_{\{1\}}\# g^3x \}$ is a generating set of $\underline{AP_{\{1\}}\# H}$. Note also that
$$ \left\{\begin{array}{ccl}
   P_{\{1\}}\# 1 & = & P_{\{1\}}\x 1 \\
   \eps_xP_{\{1\}}\# 1 & = & \eps_xP_{\{1\}}\x 1, \\ 
   P_{\{1\}}\# g^3x & = & P_{\{1\}}\otimes g^3x, \\ 
   \eps_xP_{\{1\}}\# g^3x & = & \eps_xP_{\{1\}}\x g^3x,
\end{array}
\right. $$
where on the right side of the equalities, we have elements that are linearly independent in $A\x H$, hence $\{P_{\{1\}}\# 1; \ \eps_xP_{\{1\}}\# 1; \ P_{\{1\}}\# g^3x; \ \eps_xP_{\{1\}}\# g^3x \}$ is a basis of $\underline{AP_{\{1\}}\# H}$. Now for the algebra structure. We know that
\[
P_{\{1\}}\# 1 = 1_{\underline{AP_{\{1\}}\# H}}, \textrm{ and }  (\eps_xP_{\{1\}}\# 1)( P_{\{1\}}\# g^3x) = \eps_xP_{\{1\}}\# g^3x, \]
so $\underline{AP_{\{1\}}\# H}$, as an algebra, only has two generators $\eps_xP_{\{1\}}\# 1_H$ and $P_{\{1\}}\# g^3x$. Notice that 
\begin{eqnarray*}
    (P_{\{1\}}\# g^3x)(\eps_xP_{\{1\}}\# 1) = g^3x\cdot (\eps_xP_{\{1\}})\# g^3 + \eps_xP_{\{1\}}\# g^3x \\
    = \eps_xP_{\{1\}}\# g^3x = (\eps_xP_{\{1\}}\# 1)( P_{\{1\}}\# g^3x),
\end{eqnarray*}
so they commute, and also
\[
(\eps_xP_{\{1\}}\# 1)^2 =  0  \ \ \textrm{and} \ \
    (P_{\{1\}}\# g^3x)^2 = 0.
\]

Hence, the algebra $\underline{AP_{\{1\}} \# H}$ is generated by two commuting elements, each nilpotent of order two. Consequently, there exists an isomorphism of algebras
\[
\underline{AP_{\{1\}} \# H} \cong \frac{\K[X,Y]}{\langle X^2, Y^2 \rangle}.
\]

    \item[iii)] For the algebra $AP_{\langle g^2\rangle}$, we have seen that the set $\{(\eps_x)^nP_{\langle g^2 \rangle}\}_{n\geq 0}$ is a basis. Then the $H$-partial action is determined by the equalities
\[\begin{array}{l}
     g^l\cdot\big((\eps_x)^nP_{\langle g^2\rangle}\big) = (\delta_{l,0} + \delta_{l,2})(\eps_x)^nP_{\langle g^2\rangle}, \ \textrm{for all } n\geq 0 \textrm{ and } l=0,1,2,3,   \\
     \\
     g^lx\cdot\big((\eps_x)^nP_{\langle g^2\rangle}\big) = (\eps_x)^{n+1}P_{\langle g^2\rangle},  \ \ \textrm{for all } n\geq 0 \textrm{ and } l=0,1,2,3.   
\end{array}
\]
Now we are going to find a basis for the algebra $\underline{
AP_{\langle g^2\rangle}\# H}$, to do that we apply the same ideas of the previous case, to find that
\[
\begin{array}{c}
    (\epsilon_{x})^nP_{\langle g^2\rangle}\# g^l \ = \ \bigl(\delta_{l,0}+\delta_{l,2}\bigr)(\epsilon_{x})^nP_{\langle g^2\rangle}\x g^l;  \\
    \\
    (\epsilon_{x})^nP_{\langle g^2\rangle}\# x \ = \ (\epsilon_{x})^{n+1}P_{\langle g^2\rangle}\x 1; \ \ \ (\epsilon_{x})^nP_{\langle g^2\rangle}\# gx \ = \  (\epsilon_{x})^nP_{\langle g^2\rangle}\otimes gx \\
    \\
    (\epsilon_{x})^nP_{\langle g^2\rangle}\# g^2x \ = \ (\epsilon_{x})^{n+1}P_{\langle g^2\rangle}\x g^2; \ \ \
    (\epsilon_{x})^nP_{\langle g^2\rangle}\# g^3x \ = \ (\epsilon_{x})^{n+1}P_{\langle g^2\rangle}\x g^3x.    
\end{array}
\]
Hence,
$$\bigl\{(\epsilon_x)^{n_1}P_{\langle g^2\rangle}\# 1; \ (\epsilon_x)^{n_2}P_{\langle g^2\rangle}\# g^2; \ (\epsilon_x)^{n_3}P_{\langle g^2\rangle}\# gx; \ (\epsilon_x)^{n_4}P_{\langle g^2\rangle}\# g^3x \bigr\}_{n_1,n_2,n_3,n_4 \geq 0}$$
is a basis of $\underline{P_{\langle g^2\rangle}A\# H}$. 

Now for the algebra structure, notice that,
\[
(\eps_x)^nP_{\langle g^2\rangle}\#h = ((\eps_x)^nP_{\langle g^2\rangle}\# 1)(P_{\langle g^2\rangle}\# h), \ \ \textrm{ for all } h \in H \textrm{ and } n\geq 0,
\]
and also $(P_{\langle g^2\rangle}\# g^2)(P_{\langle g^2\rangle}\# g^3x) = P_{\langle g^2\rangle}\# gx$. So we only have these three generators $(\eps_xP_{\langle g^2\rangle}\# 1)$, $(P_{\langle g^2\rangle}\# g^2)$ and  $(P_{\langle g^2\rangle}\# g^3x)$. Again, these generators commute with each other, however 
\[
(P_{\langle g^2\rangle}\# g^2)^2 = 1_{\underline{P_{\langle g^2\rangle}\# H}}, \ \ \textrm{and} \ \ (P_{\langle g^2\rangle}\# g^3x)^2 = ((\eps_x)^2P_{\langle g^2\rangle})(P_{\langle g^2\rangle}\# g^2).
\]
We just proved that 
\[
\underline{AP_{\langle g^2\rangle}\# H} \simeq \dfrac{\K[X,Y,U]}{\langle U^2-1, \ Y^2-UX^2\rangle}.
\]

\end{itemize}

Then, since we have that 
\[
\Hpar \simeq \Bigl(\underline{AP_G\#H}\Bigr) \oplus \Bigl(\underline{AP_{\langle g^2 \rangle}\#H}\Bigr) \oplus \Bigl(\underline{AP_{\{1\}}\#H}\Bigr) \oplus \Mat_{2}\bigl(\underline{AP_{\{1\}}\#H}\bigr) \oplus \Mat_{3}\bigl(\underline{AP_{\{1\}}\#H}\bigr),
\]
the next result follows. 

\begin{proposicao}
    If $H$ is the Hopf algebra generated by $g$ and $x$, subject to the relations $ g^4 = 1,$ $ x^2=0,$ $xg=-gx,$ and with coalgebra structure given by $ \Delta(g)=g\x g,$ $\Delta(x) = x\x 1 + g\x x$. Then 
    \[
\Hpar \simeq H \oplus \dfrac{\K[X,Y,U]}{\langle U^2-1, \ Y^2-UX^2\rangle} \oplus \dfrac{\K[X,Y]}{\langle X^2,Y^2\rangle} \oplus \Mat_2\left(\dfrac{\K[X,Y]}{\langle X^2,Y^2\rangle}\right)\oplus \Mat_3\left(\dfrac{\K[X,Y]}{\langle X^2,Y^2\rangle}\right).
\]   
\end{proposicao}

\subsection{A Hopf algebra of rank one: non-nilpotent case} \label{section.non-nilpotent}

    Now consider the Hopf algebra $H$ generated by $g$ and $x$, subject to the following relations:
    \begin{equation}\label{equation.1.non.nilpotent}
        g^4=1, \ \ x^2=g^2-1, \ \ xg = -gx, 
    \end{equation} 
    and with structural maps, as follows:
    \begin{equation}\label{equation.2.non.nilpotent}
 \Delta(g) = g\x g, \ \ \Delta(x)=x\x 1 + g\x x, 
    \end{equation}
  \begin{equation}\label{equation.3.non.nilpotent}
  S(g)=g^{3}, \ \ S(x) = gx. 
    \end{equation}
    This Hopf algebra is similar to the previous one, studied in Section \ref{section.nilpotent}: they are isomorphic as coalgebras, but not as algebras. The set $\{1,g,g^2,g^3, x, gx, g^2x, g^3x\}$ is once more a basis of $H$. This is a pointed Hopf algebra of rank one, and the group of grouplike elements is $G=\{1,g,g^2,g^3\} \cong \Z_4$.
    
    Since the underlying groupoid $\mathcal{G}(G)$ and its set of objects $\mathcal{P}_1(G)$ are the same as in the previous section, the central idempotents of $A$ corresponding to $\mathcal{P}_1(G)$ are the same, and the same isomorphisms hold between components $P_XA$, also
\[
A \simeq AP_G\oplus AP_{\langle g^2\rangle} \oplus AP_{\{1\}} \oplus (AP_{\{1\}})^2 \oplus (AP_{\{1\}})^3 ,
\]
hence 
\[
\Hpar \simeq \left(\underline{AP_G\# H}\right)\oplus \left(\underline{AP_{\langle g^2\rangle}\# H}\right) \oplus \left(\underline{AP_{\{1\}}\# H}\right) \oplus \Mat_2\left(\underline{AP_{\{1\}}\# H}\right) \oplus \Mat_3\left(\underline{AP_{\{1\}}\# H}\right) .
\]

Since the Hopf algebra of this example and the Hopf algebra of the previous example are isomorphic as \textit{coalgebras}, the first two equations provide the same information and  Equations (\ref{equation.glx.glxga.gl.glx}) and (\ref{equation.gAxgL.gLxgA}) remain true for this algebra $A$ as well, that is:
    \begin{equation}\label{equation.copia.ZeroUm}
        \epsilon_{g^lx} = \epsilon_{g^lx}\eps_{g^{l}}+\eps_{g^{l+1}}\eps_{g^lx},
    \end{equation}
    \begin{equation}\label{equation.copia.ZeroDois}
        (-1)^{\alpha+l}\eps_{g^{\alpha}x}\bigl[\eps_{g^{l+1}}-\eps_{g^l}\bigr] = \eps_{g^{l}x}\bigl[\eps_{g^{\alpha+1}}-\eps_{g^\alpha}\bigr].
    \end{equation}
    
    However the expression $\eps_{h_{(1)}k}\eps_{h_{(2)}} = \eps_{h_{(1)}}\eps_{h_{(2)}k}$ \eqref{equality.Apar.2}, where $h=g^lx$ and $k=g^mx$, yields the new equation \begin{equation}\label{equation.novidade}
        (-1)^{\alpha+l}(\eps_{g^{\alpha+2}}-\eps_{g^{\alpha}})(\eps_{g^{l+1}}-\eps_{g^l}) + \eps_{g^{\alpha}x}\eps_{g^lx} = \eps_{g^lx}\eps_{g^{\alpha+1}x}.
    \end{equation}

The computations for the last equation are entirely analogous to those of the previous nilpotent case.
In particular, all the tools used there work here as well, with obvious adaptations.
    \begin{proposicao} \label{proposicao.algebraA.exemplo2}
        Let $H$ be the Hopf algebra defined by \eqref{equation.1.non.nilpotent}, \eqref{equation.2.non.nilpotent} and \eqref{equation.3.non.nilpotent}. Then $A = \apar$ is isomorphic to the algebra 
        \[
        \K \oplus \K[X] \oplus  (\K[X]/\langle X^2 + 1\rangle)^6
        \]
    \end{proposicao}
\begin{proof}
    Since the underlying groupoid $\mathcal{G}(G)$ is the same as the one in the nilpotent case, we only have to study the components defined by the central idempotents $P_{\{1\}}$,  $P_{\langle g^2\rangle}$ and $P_{G}$.

We already know that $AP_G \simeq \K$. 

    Multiplying equalities \eqref{equation.copia.ZeroUm}, \eqref{equation.copia.ZeroDois} and \eqref{equation.novidade} by $P_{\{1\}}$, we obtain the relations
 \[            \epsilon_{g^l}P_{\{1\}}  =  \delta_{l,0}P_{\{1\}}, \ \  
         \epsilon_{g^lx}P_{\{1\}}  =  (\delta_{l,0}+\delta_{l,3})\epsilon_{x}P_{\{1\}}, \ \ 
         (\epsilon_x)^2P_{\{1\}}  =  -P_{\{1\}},
\]
which show that the ideal $ AP_{\{1\}} $ is generated by $\epsilon_{x}P_{\{1\}}$ as an algebra, and that 
$$ AP_{\{1\}}  \simeq \K[X]/\langle X^2 + 1 \rangle$$
as algebras. 

For $P_{\langle g^2\rangle}$, we have the equalities
\[
         \epsilon_{g^l}P_{\langle g^2\rangle}  =  (\delta_{l,0}+\delta_{l,2})P_{\langle g^2\rangle}, \ \ 
         \epsilon_{g^lx}P_{\langle g^2\rangle}  =  \epsilon_{x}P_{\langle g^2\rangle}, 
\]
and then we have that 
$$ AP_{\langle g^2\rangle}  \simeq \K[X]$$
as algebras. 

By the same reasoning of Proposition \ref{proposicao.Apar.nilpotente} we have the algebra isomorphisms 
    $$ \hspace{-1.5cm} (P_{\{1\}}A) \simeq (P_{\{1,g\}}A) \simeq (P_{\{1,g^3\}}A) \simeq (P_{\{1,g,g^2\}}A) \simeq (P_{\{1,g,g^3\}}A) \simeq (P_{\{1,g^2,g^3\}}A) \simeq \K[X]/\langle X^2+1\rangle$$
    and the result follows. 

\end{proof}

\begin{obs}
The algebra isomorphism established in the previous Proposition does not take the partial $H$-action into account. To properly incorporate this structure, it is more appropriate to express $A$ in the following form:
\[
A \cong \K \oplus \K[X] \oplus \K[X]/\langle X^2+1 \rangle
      \oplus \bigl(\K[X]/\langle X^2+1 \rangle\bigr)^2
      \oplus \bigl(\K[X]/\langle X^2+1 \rangle\bigr)^3 .
\]
Each of these components carries a structure of partial $H$-module algebra and can be regarded as a partial $H$-submodule algebra of $A$.
\end{obs}

Once more, to describe the algebra $\Hpar$, we will first study the $H$-partial action on each component $AP_{G}$, $AP_{\langle g^2\rangle}$ and $AP_{\{1\}}$ and then study the smash partial product of each algebra. 

\begin{itemize}
    \item[i)] For the algebra $AP_{G}$, it follows again by the Corollary \ref{corolario.HparGammaG.iso.H}, that
    \[
    \underline{AP_G\# H} \simeq H.
    \]
    \item[ii)] The algebra $AP_{\{1\}}$ has the set $\{P_{\{1\}}, \ \eps_xP_{\{1\}}\}$ as a basis and the $H$-partial action is defined like this:
    \[
    \left\{
    \begin{array}{ccl}
        g^l\cdot P_{\{1\}} & = & \delta_{l,0}P_{\{1\}}  \\
        g^lx\cdot (P_{\{1\}}) & = & (\delta_{l,0} + \delta_{l,3})\eps_xP_{\{1\}}, \\
        g^l\cdot (\eps_xP_{\{1\}}) & = & \delta_{l,0}\eps_xP_{\{1\}} \\
        g^lx\cdot (\eps_xP_{\{1\}}) & = & -(\delta_{l,0}+\delta_{l,3})P_{\{1\}}.
    \end{array}
    \right.
    \]
    Also, since $\{P_{\{1\}}, \eps_xP_{\set{1}}\}$ is a basis of $AP_{\set{1}}$, then 
    $$\{P_{\{1\}}\#1, \ P_{\{1\}}\#g, \ P_{\{1\}}\#g^2, \ P_{\{1\}}\#g^3, \ P_{\{1\}}\#x, \\ \ P_{\{1\}}\#gx, \ P_{\{1\}}\#g^2x, \ P_{\{1\}}\#g^3x \}$$ is a set of generators of the vector space $\ul{AP_{\{1\}}\# H}$. We use the same strategy as in the previous example to find a basis, and it is not hard to see that 
    $$\{P_{\{1\}}\# 1; \ \eps_xP_{\{1\}}\# 1; \ P_{\{1\}}\# g^3x; \ \eps_xP_{\{1\}}\# g^3x \}$$ 
    is a basis of $\underline{A P_{\{1\}}\# H}$.
    
    Now for the algebra structure, we have two generators $\epsilon_xP_{\{1\}}\# 1_H$ and $P_{\{1\}}\# g^3x$, they commute and
    \[
    (\eps_xP_{\{1\}}\# 1_H)^2 = (P_{\{1\}}\# g^3x)^2 = - 1_{\underline{AP_{\{1\}}\# H}}
    \]
So we conclude that there is an isomorphism of algebras
$$\ul{AP_{\{1\}}\# H} \simeq \dfrac{\K[X,Y]}{\bk{X^2+1, \ Y^2+1}}$$

    \item[iii)] The algebra $AP_{\langle g^2\rangle}$, has the set $\{(\eps_x)^nP_{\bk{g^2}}\}_{n\geq 0}$ as a basis, and the $H$-partial action is defined like this 
    \[
    \left\{\begin{array}{ccl}
        g^l\cdot \left((\eps_x)^nP_{\bk{g^2}}\right) & = &(\delta_{l,0}+\delta_{l,2})(\eps_x)^nP_{\bk{g^2}} \\
        g^lx\cdot \left((\eps_x)^nP_{\bk{g^2}}\right) & = & (\eps_x)^{n+1}P_{\bk{g^2}}.
    \end{array}
    \right.
    \]
    Also, since $\{(\eps_x)^nP_{\bk{g^2}}\}_{n\geq 0}$ is a basis of $AP_{\bk{g^2}}$, we can find a basis of $\ul{AP_{\bk{g^2}}\# H}$, by again use the same strategy. Hence,
$$\bigl\{(\epsilon_x)^{n_1}P_{\bk{g^2}}\# 1; \ (\epsilon_x)^{n_2}P_{\bk{g^2}}\# g^2; \ (\epsilon_x)^{n_3}P_{\bk{g^2}}\# gx; \ (\epsilon_x)^{n_4}P_{\bk{g^2}}\# g^3x \bigr\}_{n_1,n_2,n_3,n_4 \geq 0}$$
is a basis of $\underline{P_{\bk{g^2}}A\# H}$. Now for the algebra structure, we are only going to need three generators, $(\epsilon_x)P_{\bk{g^2}}\# 1$,  $P_{\bk{g^2}}\# g^2$ and $P_{\bk{g^2}}\# g^3x$. One can prove that those three commute with each other, and also that
\[   (P_{\bk{g^2}}\# g^2)^2 = 1_{\ul{AP_{\bk{g^2}}\# H}} \ \ \textrm{and} \ \ 
     (P_{\langle g^2\rangle}\# g^3x)^2 = \bigl((\eps_x)^2+1\bigr)P_{\bk{g^2}}\# g^2 - P_{\bk{g^2}}\# 1 .
\]
Then there is an isomorphism of algebras  
$$ \ul{AP_{\bk{g^2}}\# H} \simeq \dfrac{\K[X,Y,U]}{\bk{U^2-1, \ Y^2-(X^2+1)U + 1}} $$
\end{itemize}

Since we have that 
\[
\Hpar \simeq \left(AP_{G}\# H\right) \oplus \left(AP_{\bk{g^2}}\# H\right) \oplus \left(AP_{\{1\}}\# H\right) \oplus \Mat_2\left(AP_{\{1\}}\# H\right) \oplus \Mat_3\left(AP_{\{1\}}\# H\right),
\]
the next result follows. 

\begin{proposicao}
    If $H$ is the Hopf algebra generated by $g$ and $x$, subject to the relations $g^4=1$, $x^2=g^2-1$, $xg=-gx$, and with coalgebra structure given by $\Delta(g) = g\x g$ and $\Delta(x)=x\x 1 + g\x x$, then
    \begin{eqnarray*}
    \Hpar \simeq & H \oplus \dfrac{\K[X,Y,U]}{\bk{U^2-1, \ Y^2-(X^2+1)U + 1}} \oplus \dfrac{\K[X,Y]}{\bk{X^2+1, \ Y^2+1}} \oplus \\
    & \oplus \Mat_2\left(\dfrac{\K[X,Y]}{\bk{X^2+1, \ Y^2+1}}\right) \oplus \Mat_3\left(\dfrac{\K[X,Y]}{\bk{X^2+1, \ Y^2+1}}\right)
\end{eqnarray*}
\end{proposicao}

\color{black}

\section{Acknowledgements}

Arthur Rezende Alves Neto was supported during his PhD course by CAPES (Proc. No. 88882.381832/2019-01), during his PDSE by CAPES (Proc. No. 88882.381832/2019-01) and now is supported by FAPESC (Proc. no. 41001010001P6).
Marcelo Muniz Alves was partially supported by Funda\c{c}\~ao de Amparo \`a Pesquisa do Estado de Minas Gerais (FAPEMIG) under Universal Grant APQ-03491-25, and National Council for Scientific  and Technological Development - CNPq under Universal Grant 405540/2023-0 and Research Productivity Fellowship (Bolsa de Produtividade em Pesquisa) grant  309469/2019-8. 

We would like to express our sincere gratitude to William Hautekiet for the valuable discussions and, in particular, for the insightful observation concerning convolution idempotent morphisms, which emerged during the period of Arthur's Sandwich PhD program. His insight was fundamental to the development of this work.

We are also deeply grateful to Joost Vercruysse for his warm hospitality during Arthur's Sandwich PhD period. Beyond that, we especially acknowledge the many discussions and, most importantly, his guidance and directions, which significantly influenced the perspectives and results presented in this article.

\bibliographystyle{plain}
\bibliography{refs}

\end{document}